\documentclass{amsart}

\usepackage{color}
\usepackage{amsmath}
\usepackage{amssymb}
\usepackage{amsthm}
\usepackage{amsxtra}
\usepackage{mathrsfs}
\usepackage{extarrows}
\usepackage{hyperref}
\usepackage{commath}
\usepackage{bm}
\usepackage{upgreek}
\usepackage{tikz}
\usepackage{tikz-cd}

\newtheorem{Prop}{Proposition}[section]
\newtheorem{Lem}[Prop]{Lemma}

\newtheorem{Thm}[Prop]{Theorem}
\newtheorem*{Thm*}{Theorem}

\theoremstyle{definition}
\newtheorem{Def}[Prop]{Definition}

\newcommand{\bZ}{\mathbb{Z}}

\newcommand{\bfV}{\mathbf{V}}

\newcommand{\bC}{\mathbb{C}}

\newcommand{\cJ}{\mathcal{J}}
\newcommand{\cN}{\mathcal{N}}
\newcommand{\sN}{\mathcal{N}}
\newcommand{\sM}{\mathcal{M}}

\newcommand{\fp}{\mathfrak{p}}

\newcommand{\sA}{\mathcal{A}}

\newcommand{\sF}{\mathcal{F}}

\newcommand{\sL}{\mathcal{L}}
\newcommand{\sI}{\mathscr{I}}
\newcommand{\sJ}{\mathcal{J}}
\newcommand{\bP}{\mathbf{P}}

\newcommand{\myi}{\mathrm{i}}

\DeclareMathOperator{\im}{im}
\DeclareMathOperator{\Hom}{Hom}
\DeclareMathOperator{\Tor}{Tor}

\usepackage[style=alphabetic, backend=biber]{biblatex}

\usepackage{mathtools}

\numberwithin{equation}{section}

\DeclareMathOperator{\Pic}{\mathrm{Pic}}
\DeclareMathOperator{\Spec}{\mathrm{Spec}}
\DeclareMathOperator{\cPic}{\mathcal{P}ic}

\DeclareMathOperator{\Inv}{\mathcal{I}nv}
\DeclareMathOperator{\id}{\textnormal{id}}
\DeclareMathOperator{\coker}{\textnormal{coker}}
\DeclareMathOperator{\Ass}{\textnormal{Ass}}

\newcommand{\sO}{\mathcal{O}}
\newcommand{\sR}{\mathcal{R}}

\newcommand{\Nm}{\textnormal{Nm}}

\newcommand{\tor}{\textnormal{tor}}
\newcommand{\cV}{\mathcal{V}}
\newcommand{\cU}{\mathcal{U}}

\newcommand{\fsa}{f\text{-s.a}}
\newcommand{\reg}{\textnormal{reg}}
\newcommand{\pbp}{\partial\bar{\partial}}
\newcommand{\inversepii}{\frac{1}{\uppi\myi}}

\newcommand{\VCV}{\Vert\cdot\Vert}

\bibliography{Reference}

\title{Deligne pairing for equidimensional morphisms}
\author{Shiquan Li}
\address{School of Mathematical Sciences, Peking University, Beijing 100871, China}
\keywords{Deligne pairing, equidimensional morphism, intersection theory, Arakelov theory.}
\email{lisq@pku.edu.cn}
\subjclass[2020]{14C17, 14F06, 14G40}

\date{\today}

\begin{document}

\begin{abstract}
	Let \(S\) be a noetherian normal scheme, and let \(X\to S\) be a surjective projective morphism of pure relative dimension \(d\). We construct a symmetric multi-additive functor \(\cPic(X)^{d+1} \to \cPic(S)\), and prove its functorial properties. Our construction uses Elkik's and Garc\'ia's
	ideas, as well as algebraic Hartogs' theorem. Moreover, our results can be used to define arithmetic intersection theory of hermitian line bundles for equidimensional morphisms.
\end{abstract}

\maketitle

\tableofcontents

\section{Introduction}

Intersection theory is a fundamental tool in algebraic geometry. For a \(d\)-dimensional scheme \(X\) proper over a field \(k\), we can define the intersection number \((\sL_1,\ldots,\sL_d):\Pic^d(X)\to \bZ\), which is a symmetric multi-additive group homomorphism. Then we want to generalize intersection number and construct an intersection theory for a family of schemes \(X\to S\). Let \(\cPic(X)\) denote the category of invertible \(\sO_X\)-modules with morphisms being isomorphisms (we call it a \emph{Picard category}), and similarly for \(\cPic(S)\). For a smooth projective morphism \(X\to S\) of pure relative dimension 1, i.e. a smooth projective family of curves, Deligne \cite{SGA4-3} (Formulaire 1.3.16) defined a functorial pairing \(\cPic(X)\times\cPic(X)\to\cPic(S)\). Then Deligne \cite{DelignePairingOriginProceeding} conjectured that we could construct a functorial map (which we now call \emph{Deligne pairing}) \(\cPic(X)^{d+1}\to\cPic(S)\) for a flat proper morphism \(X\to S\) of pure relative dimension \(d\). Since then, Deligne pairing has been constructed for various families \(X\to S\) by a series of work such as \cite{Elkik}, \cite{Garcia2000}, \cite{Ducrot2005}, \cite{Nakayama}, \cite{ExplicitDP}, and \cite{DP2024}. These results impose different restrictions on the family \(X\to S\). For example, \cite{Garcia2000} requires \(X\to S\) to be a surjective equidimensional projective morphism of noetherian schemes of finite Tor-dimension, and \cite{DP2024} drops the noetherian assumption and requires \(X\to S\) to be a faithfully flat locally projective morphism of finite presentation.

Deligne pairing can also be viewed as the generalization of norm functor. For algebraic integer ring extensions, we have the norm map. In EGAII \cite{EGA2} \S6.5, Grothendieck generalized the norm map to a norm functor \(\cPic(X)\to \cPic(S)\) for a flat finite morphism \(X\to S\), which is exactly the Deligne pairing for a flat proper morphism \(X\to S\) of pure relative dimension \(0\).

In EGAII \cite{EGA2} \S6.5, Grothendieck defined another norm functor \(\cPic(X)\to \cPic(S)\) for a finite morphism \(X\to S\) to a normal scheme \(S\), while the morphism is not necessarily flat. Inspired by Deligne pairing for flat morphisms, we generalize this norm functor to a symmetric multi-additive functor for equidimensional morphisms over normal schemes. We still call our construction \emph{Deligne pairing} since for flat morphisms, our construction is identical to the original Deligne pairing. 

\subsection*{Convention} All rings in this paper are commutative with identity \(1\). We sometimes use ``\(=\)" to denote a canonical isomorphism. For a scheme \(X\), we use \(X^{(1)}\) to denote the set of points of codimension 1 in \(X\), i.e. \(X^{(1)}=\{x\in X\mid \dim\sO_{X,x} = 1\}\). We always use \(d\in\bZ_{\geqslant0}\) to denote the relative dimension of a morphism. We follow the convention that \(\dim (\emptyset) = -1\). In particular, a set of non-negative dimension is tacitly nonempty. 

\subsection{Deligne pairing for equidimensional morphisms}

\begin{Def}
	Let \(f\colon X\to S\) be a morphism of schemes. We say that \(f\) is \emph{of pure relative dimension} \(d\) if for each \(x\in X\), every irreducible component of \(f^{-1}(f(x))\) has dimension \(d\). Such morphism is also called an \emph{equidimensional morphism}.
\end{Def}


For example, a flat morphism between varieties (integral schemes of finite type over a field) is equidimensional.


We state our main result as follows.

\begin{Thm}[Theorem \ref{thm:functorial EDP}, \ref{base change of normal Deligne pairing}, \ref{pull-back formula of normal Deligne pairing}, \ref{projection formula of normal Deligne pairing}, \ref{projection formula of taking closed subscheme}, \ref{birational invariance}, \ref{EDP = DP for all lb}]
	Let \(f\colon X\to S\) be a surjective projective morphism of noetherian schemes of pure relative dimension \(d\). Suppose that \(S\) is normal. Then we can define a symmetric multi-additive functor,
	\begin{equation*}
		\langle\sL_1,\sL_2,\ldots,\sL_{d+1}\rangle_{X/S}\colon\cPic(X)^{d+1}\to\cPic(S)
	\end{equation*}
	which we call Deligne pairing, and satisfies the following functorial properties:
		
	(1) Let \(g\colon S'\to S\) be a good base change (defined in Definition \ref{def: good base change}), \(X' = X\times_SS'\to S'\) and \(g'\colon X'\to X\). Then we have a natural isomorphism
	\[ \langle g'^*\sL_1,g'^*\sL_2,\ldots,g'^*\sL_{d+1} \rangle_{X'/S'}\to g^*\langle \sL_1,\sL_2,\ldots,\sL_{d+1} \rangle_{X/S}. \]
	
	(2) Suppose that \(S\) is integral. Let \(\sL_1,\ldots,\sL_d\) be invertible \(\sO_X\)-modules, and let \(\delta\) be their intersection number on the generic fiber of \(X\to S\). Let \(\sM\) be an invertible \(\sO_S\)-module. Then we have a natural isomorphism
	\[ \langle \sL_1,\ldots,f^*\sM,\ldots,\sL_d \rangle_{X/S}\to \sM^{\otimes\delta},\]
	and the isomorphism is compatible with any good base change \(S'\to S\), where \(S'\) is another noetherian integral normal scheme.
	
	(3) Let \(g\colon Y\to X\) be a surjective projective morphism of noetherian schemes of pure relative dimension \(d'\). Suppose that \(X\) and \(S\) are normal. Then we have a natural isomorphism
	\begin{multline*}
		\langle g^*\sL_1,\ldots,g^*\sL_d,\sM_{d+1},\ldots,\sM_{d+d'+1} \rangle_{Y/S}\\
		\to \langle\sL_1,\ldots,\sL_d,\langle\sM_{d+1},\ldots,\sM_{d+d'+1}\rangle_{Y/X} \rangle_{X/S},
	\end{multline*}
	and the isomorphism is compatible with any good base change \(S'\to S\).
	
	(4) Suppose that \(d\geqslant 1\), and let \(D\) be an effective divisor on \(X\) such that \(D\to S\) is of pure relative dimension \(d-1\). Then we have a natural isomorphism \[\langle\sL_1,\ldots,\sO(D),\ldots,\sL_{d}\rangle_{X/S} \to \langle\sL_{1}|_D,\ldots,\sL_{d}|_D\rangle_{D/S}, \]
	where we just delete \(\sO(D)\) in the right expression. Moreover, this isomorphism is compatible with any good base change \(S'\to S\).
	
	(5) Let \(g\colon Y\to X\) be a projective morphism. Suppose that \(Y\to S\) is a surjective projective morphism of pure relative dimension \(d\), and that there exists an open subset \(U\subseteq X\) which is dense in each fiber of \(f\) and satisfies that \(g\colon g^{-1}(U)\to U\) is an isomorphism. Then we have a natural isomorphism
	\begin{equation*}
		\langle \sL_1,\sL_2,\ldots,\sL_{d+1} \rangle_{X/S}\to \langle g^*\sL_1,g^*\sL_2,\ldots,g^*\sL_{d+1} \rangle_{Y/S}.
	\end{equation*}
	
   (6) If \(d = 0\), then \(\langle\sL\rangle_{X/S} = \Nm_{X/S}(\sL)\), where \(\Nm_{X/S}\) denotes the norm functor for a finite morphism over a normal scheme defined in \cite{EGA2} \S6.5 (which we will review in the next section).
   
   Moreover, if \(f\) is flat, then our Deligne pairing is identical to the original Deligne pairing defined by Deligne \cite{SGA4-3}, Elkik \cite{Elkik}, Garc\'ia \cite{Garcia2000}, Ducrot \cite{Ducrot2005}, Eriksson and Freixas \cite{DP2024}, etc.
\end{Thm}

Now we briefly review the history of Deligne pairing\footnote{For clarity, we may omit some minor conditions of the theorems in the review.}. In \cite{EGA2} \S6.5, Grothendieck defined two kinds of norm functors \(\Nm_{X/S}\colon\cPic(X)\to\cPic(S)\) for a finite morphism \(f\colon X\to S\): one is for the case that \(f\) is flat, and the other is for the case that \(S\) is normal.

For the flat case, there is an extensive literature on the norm functor and its generalization, Deligne pairing for flat morphisms. Deligne \cite{SGA4-3} constructed this Deligne pairing for \(d=1\), and speculated a similar functor for general \(d\). We list some important results for general \(d\) below. Elkik \cite{Elkik} developed a method and constructed Deligne pairing for Cohen--Macaulay flat projective morphisms. Her student, Garc\'ia \cite{Garcia2000}, modified her method and constructed Deligne pairing for equidimensional projective morphisms of finite Tor-dimension (which implies Deligne pairing for flat projective morphisms). Moreover, Ducrot \cite{Ducrot2005} used a different method and constructed Deligne pairing for flat projective morphisms. The above constructions are all over noetherian or locally noetherian schemes. Recently, Eriksson and Freixas \cite{DP2024} modified Ducrot's method and constructed Deligne pairing for flat projective morphisms without assuming the schemes to be noetherian or locally noetherian.

For the second norm functor, on the contrary, there is few research on the topic before. Besides Grothendieck's norm functor, we only knew that N. Nakayama constructed an intersection sheaf in \cite{Nakayama}. However, he needed stronger assumptions to prove that his intersection sheaf was invertible, and he did not consider functorial properties except a base change property. In addition, his construction is only a map between Picard groups and not a functor between Picard categories. And his proof used very involved algebraic geometry and commutative algebra, while our proof is more comprehensible with explicit geometric pictures. Nakayama's results have many applications in algebraic geometry concerning Chow reductions and maximal rationally connected fibrations, which shed light on the potential applications of our results in algebraic geometry.

In this paper, based on Elkik's and Garc\'ia's ideas, we construct Deligne pairing for equidimensional morphisms and prove its functorial properties, which correspond to the functorial properties of Deligne pairing for flat morphisms. Therefore, Our result is a natural generalization of Grothendieck's second norm functor and counterpart of Deligne pairing for flat morphisms.

Now we briefly introduce our method and innovation.

\begin{Def}
	Let \(f\colon X\to S\) be a morphism of noetherian schemes. We say that \(f\) is \emph{of finite Tor-dimension} if there exists an integer \(d_0\) such that for each integer \(d>d_0\), every point \(x\in X\), and every \(\sO_{S,f(x)}\)-module \(M\), we have \(\Tor_d^{\sO_{S,f(x)}}(\sO_{X,x}, M) = 0\).
\end{Def}

Flat morphisms are of finite Tor-dimension. If \(S\) is regular, then any morphism \(f\colon X\to S\) is of finite Tor-dimension.

\begin{Def}
	Let \(f\colon X\to S\) be a projective morphism of noetherian schemes. An invertible \(\sO_X\)-modules \(\sL\) is called \(f\)-\emph{sufficiently ample} if \(\sL\) is \(f\)-ample, and there exists a locally free \(\sO_X\)-module \(\sF\) and a surjective \(\sO_X\)-module homomorphism \(\varphi\colon f^*\sF\to \sL\). We say that \((\varphi,\sF)\) is a \emph{presentation} of \(\sL\), or \(\sF\) generates \(\sL\).
\end{Def}

Locally, every invertible \(\sO_X\)-module is the difference of two \(f\)-sufficiently ample invertible \(\sO_X\)-modules, and thus we can reduce the construction of Deligne pairing for general invertible sheaves to that for sufficiently ample invertible sheaves. Then we use Elkik's and Garc\'ia's ideas. First we use the presentations to construct a projective space bundle \(\bP\) over \(S\) and canonical sections of some invertible sheaves on \(X\times_S\bP\). Then use these canonical sections to construct an invertible sheaf \(\cN\) on some open subset of \(\bP\). After that, we can apply a descent technique to descend \(\cN\) to an invertible sheaf \(\cJ\) on \(S\), which we will prove to be the Deligne pairing.

A crucial point of this method is to prove that \(\cJ\) is independent of the choice of presentations. For this goal, Elkik's and Garc\'ia's original ideas are invalid since they use the restriction of norm functor for flat finite morphisms or finite morphisms of finite Tor-dimension on a closed subscheme, while its counterpart in our case, norm functor for finite morphisms over normal schemes, behaves awful when restricting to an arbitrary closed subscheme. To solve this problem, we use algebraic Hartogs' theorem. We can assume that \(S\) is integral. Over the generic point \(\eta\) or the local ring \(A = \Spec\sO_{S,P}\) of a point \(P\) of codimension 1 in \(S\), the morphisms \(f_\eta\colon X\to \eta\) and \(f_A\colon X_A\to \Spec A\) are both of finite Tor-dimension. Thus the stalks \(\cJ_\eta\) and \(\cJ_P\), viewing as invertible sheaves on \(\eta\) and \(\Spec A\) respectively, are exactly the Deligne pairings for morphisms of finite Tor-dimension constructed by Garc\'ia in \cite{Garcia2000} \S2. Then by Garc\'ia's results in \cite{Garcia2000}, for different presentations, there are unique isomorphisms \(\psi_\eta\) and \(\psi_P\) identifying different \(\cJ_\eta\)'s and \(\cJ_P\)'s respectively, and the restriction of \(\psi_P\) on \(\eta\) is exactly \(\psi_\eta\). Therefore, by algebraic Hartogs' theorem, \(\psi_\eta\) restricts to a unique isomorphism \(\psi\), which identifies different \(\cJ\)'s for different presentations. This trick allows us to glue \(\cJ\)'s on different open subsets and finally get the Deligne pairing.

One application of our results is the construction of the canonical section of Deligne pairing using sections with purely geometric conditions. This gives a clear explanation for some notations in \cite{Zhang1996} and \cite{Moriwaki1999}.

\begin{Thm}[Theorem \ref{thm: non regular section sequence}]
	Let \(X\) and \(S\) be quasi-projective schemes over a field of characteristic 0, and let \(X\to S\) be a surjective projective morphism of pure relative dimension \(d\). Let \(\sL_1,\ldots,\sL_{d+1}\) be invertible \(\sO_X\)-modules, and for each \(1\leqslant i\leqslant d+1\), let \(s_i\) be a global section of \(\sL_i\). Suppose that \(S\) is normal, and that for each \(1\leqslant i\leqslant d\), \(Z(s_1)\cap Z(s_{2})\cdots\cap Z(s_i)\to S\) is of pure relative dimension \(d-i\). Then we can construct a canonical global section of \(\langle\sL_1,\ldots,\sL_{d+1}\rangle_{X/S}\), which we denote by \(\langle s_1,\ldots,s_{d+1}\rangle_{X/S}\). Moreover, if \(\bigcap_{i=1}^{d+1}Z(s_i)  = \emptyset\), then \(\langle s_1,\ldots,s_{d+1}\rangle_{X/S}\) is a regular section of \(\langle\sL_1,\ldots,\sL_{d+1}\rangle_{X/S}\).
\end{Thm}

\subsection{Deligne pairing of hermitian line bundles}\label{subsec:app to arakelov theory}

For quasi-projective schemes over \(\bC\) and hermitian line bundles on \(X\), i.e. invertible sheaves endowed with smooth metrics, we prove that there is a canonical smooth metric on their Deligne pairing, and the metric is compatible with functorial properties. The results are as follows.

\begin{Thm}[Theorem \ref{thm:metric on EDP}]
	Let \(X\) and \(S\) be quasi-projective schemes over \(\bC\), and let \(f\colon X\to S\) be a surjective projective morphism over \(\bC\) of pure relative dimension \(d\). Suppose that \(S\) is normal. Let \((\sL_1,\VCV_1),(\sL_2,\VCV_2),\ldots,(\sL_{d+1},\VCV_{d+1})\) be hermitian line bundles on \(X\). Then we can endow \(\langle\sL_1,\sL_2,\ldots,\sL_{d+1}\rangle_{X/S}\) with a unique smooth metric \(\VCV\) such that
	
	(1)
	\[ c_1(\langle\sL_1,\sL_2,\ldots,\sL_{d+1}\rangle_{X/S},\VCV) = \int_{X/S}\prod_{i=1}^{d+1}c_1(\sL_i,\VCV_i).\]
	
	(2) Let \(d\geqslant 1\), and let \(s_1\) be an \(f\)-regular section of \(\sL_1\). Then
	\begin{equation*}
		\log \Vert[s_1]_{X/S}\Vert = -\int_{X/S}\log\Vert s_1\Vert_1 \prod_{i=2}^{d+1}c_1(\sL_i,\VCV_i).
	\end{equation*}
	
	(3) If \(d = 0\), then
	\[ \log \Vert\Nm_{X/S}(s)\Vert = \int_{X/S}\log\Vert s\Vert_1 \]
	for each nonzero rational section \(s\) of \(\sL_1\).
	
	(4) The natural isomorphisms \(\Sigma\)'s and \(\iota\)'s defined in \eqref{eq:additive of EDP of lb} and \eqref{eq:symmetric of EDP of lb} are isometries.
	
	(5) For each good base change (in the sense of Definition \ref{def: good base change}) \(g\colon S'\to S\), the isomorphism
	\begin{equation*}
		\langle g'^*\sL_1,g'^*\sL_2,\ldots,g'^*\sL_{d+1} \rangle_{X'/S'}\to g^*\langle \sL_1,\sL_2,\ldots,\sL_{d+1} \rangle_{X/S}.
	\end{equation*}
	is an isometry, where \(X' = X\times_SS'\), \(g'\colon X'\to X\), and for each \(1\leqslant i\leqslant d+1\), the metric on \(g'^*\sL_i\) is \(g'^*\VCV_i\).
\end{Thm}

\begin{Thm}[Theorem \ref{DEP of isometry is isometry}]
	Let \(X\) and \(S\) be quasi-projective schemes over \(\bC\), and let \(f\colon X\to S\) be a surjective projective morphism over \(\bC\) of pure relative dimension \(d\). Suppose that \(S\) is normal. Let \((\sL_1,\VCV_1),(\sL_2,\VCV_2),\ldots,(\sL_{d+1},\VCV_{d+1}),(\sL'_1,\VCV'_1),(\sL'_2,\VCV'_2),\ldots,(\sL'_{d+1},\VCV'_{d+1})\) be invertible \(\sO_X\)-modules endowed with smooth metrics. For each \(1\leqslant i \leqslant d+1\), let \(u_i\colon \sL_i\to \sL'_i\) be an isometry. If we endow \(\langle \sL_1,\sL_2,\ldots,\sL_{d+1}\rangle_{X/S}\) and \(\langle \sL'_1,\sL'_2,\ldots,\sL'_{d+1}\rangle_{X/S}\) with the canonical metrics, then the following isomorphism of invertible \(\sO_S\)-modules
	\[ \langle u_1,u_2,\ldots,u_{d+1}\rangle_{X/S}\colon\langle \sL_1,\sL_2,\ldots,\sL_{d+1}\rangle_{X/S}\to\langle \sL'_1,\sL'_2,\ldots,\sL'_{d+1}\rangle_{X/S} \]
	is also an isometry.
\end{Thm}

\begin{Thm}[Theorem \ref{thm:pull-back formula of metric}]
	Let \(X\) and \(S\) be quasi-projective schemes over \(\bC\), and let \(f\colon X\to S\) be a surjective projective morphism over \(\bC\) of pure relative dimension \(d\). Suppose that \(S\) is integral and normal. Let \((\sL_1,\VCV_1),(\sL_2,\VCV_2),\ldots,(\sL_{d},\VCV_{d})\) be invertible \(\sO_X\)-modules endowed with smooth metrics, and let \((\sM,\VCV)\) be an invertible \(\sO_S\)-module endowed with smooth metrics. We endow \(f^*\sM\) with the smooth metric \(f^*\VCV\). Let \(\delta\) be the intersection number of \(\sL_1,\ldots,\sL_d\) on the generic fiber of \(f\) (if \(d = 0\), then let \(\delta\) be the degree of finite morphism \(f\)). If we endow \(\langle\sL_1,\ldots,f^*\sM,\ldots,\sL_d\rangle_{X/S}\) with the canonical metric, then the natural isomorphism
	\[ \langle\sL_1,\ldots,f^*\sM,\ldots,\sL_d\rangle_{X/S}\to \sM^{\otimes\delta} \]
	is an isometry.
\end{Thm}

The problem of constructing a canonical metric for Deligne pairing of hermitian line bundles is also raised by Deligne in \cite{DelignePairingOriginProceeding}, where he constructed the metric for flat projective morphisms of pure relative dimension 1. For general \(d\), Elkik \cite{Elkik2} constructed the metric for Cohen--Macaulay smooth projective morphisms, and Yuan \cite{YuanALB} constructed the metric for flat projective morphisms. Our theorems may be the first result on the canonical metric on Deligne pairing of hermitian line bundles for equidimensional morphisms.

We hope for a similar result over non-archimedean fields and applications of our results in number theory, arithmetic geometry, and algebraic dynamics.

\subsection{Structure of this article}

In the second section, we review the construction of norm functor for finite morphisms over normal schemes. Although we basically follow \cite{EGA2} \S6.5, we drop the noetherian assumption to prepare for the potential construction of Deligne pairing for equidimensional morphisms over non-noetherian schemes, resembling that for flat morphisms in \cite{DP2024}.

In the third section, we prove a gluing result of  different norm functors for a surjective projective morphism of pure relative dimension 1 over a noetherian normal scheme. Results in relative dimensions 0 and 1 are fundamental tools for our construction of Deligne pairing.

In the forth section, we use Elkik's and Garc\'ia's ideas to construct Deligne pairing for sufficiently ample invertible sheaves, and then extend the functor to all invertible sheaves by linearity. As we mentioned above, we first use the presentations of sufficiently ample invertible sheaves to construct a projective space bundle \(\bP\) over \(S\) and canonical sections of some invertible sheaves on \(X\times_S\bP\). Then we use these canonical sections to construct an invertible sheaf \(\cN\) on some open subset of \(\bP\). After that, we use a descent technique to descend \(\cN\) to an invertible sheaf \(\cJ\) on \(S\). Finally, we use algebraic Hartogs' theorem to show that \(\cJ\) is independent of the choice of the presentations. More precisely, for different choices, there is a unique isomorphism between them. This uniqueness allows us to glue \(\cJ\)'s on different open subsets.

In the fifth section, following the framework of Garcia \cite{Garcia2000}, we prove functorial properties of Deligne pairing. These functorial properties correspond to their counterparts of Deligne pairing for flat morphisms. We will see that the algebraic Hartogs' theorem and Garcia's results \cite{Garcia2000} play an important role in our proof, and such treatment will be used repeatedly in this paper.

For a flat morphism \(X\to S\) , if we have a suitable sequence \(s_1,s_2,\ldots,s_{d+1}\) of sections of the invertible sheaves \(\sL_1,\sL_2,\ldots,\sL_{d+1}\), then we can define a canonical section of \(\langle\sL_1,\sL_2,\ldots,\sL_{d+1}\rangle_{X/S}\), which is usually denoted by \(\langle s_1,s_2,\ldots,s_{d+1}\rangle_{X/S}\). Such section is important for analyzing positivity properties of Deligne pairing, and has many other important applications. In the sixth section, we prove similar results for Deligne pairing for equidimensional morphisms and construct the canonical section from suitable section sequences.

In the seventh section, we construct isomorphisms between Deligne pairings induced by isomorphisms between original invertible sheaves, and show that these isomorphisms are compatible with the functorial properties of Deligne pairing. Therefore, we prove that Deligne pairing is a symmetric multi-additive functor between Picard categories.

In the eighth section, we construct a canonical metric for Deligne pairing of hermitian line bundles for equidimensional morphisms. The basic idea, due to Elkik \cite{Elkik2}, is taking an \(f\)-regular section and then using induction on the relative dimension \(d\).

%

\subsection*{Acknowledgments}
The author appreciates his mentor, Professor Xinyi Yuan, for his patient guidance, constant encouragement, and financial support. The author also thanks Yinchong Song, Ziqi Guo, Binggang Qu, and Jiawei Yu for helpful discussions.

\section{Norm functor}

In this section, we recall some basic facts on the norm functor for finite morphisms over normal schemes.

\begin{Lem}\label{disjoint union of irreducible components}
	Let \(X\) be a scheme such that every point of \(X\) has a neighborhood which meets only finitely many irreducible components of \(X\). Suppose that for each \(x\in X\), the nilradical of \(\sO_{X,x}\) is a prime ideal. Then the irreducible components \(X\) are disjoint, and each irreducible component is open in \(X\).
\end{Lem}
\begin{proof}
	Let \(x\in X\) and \(U = \Spec A\) be an affine open neighborhood of \(x\). Suppose that \(x\) corresponds to prime ideal \(\fp_x\subseteq A\). We know that each minimal prime of \(\sO_{X,x}\) corresponds to a minimal prime of \(A\) contained in \(\fp_x\), and then corresponds to an irreducible component of \(U\) containing \(x\). Since the nilradical of \(\sO_{X,x}\) is prime, by the above correspondence we deduce that every \(x\in X\) in contained in a unique irreducible component of any affine open neighborhood \(U\). Thus the irreducible components of \(X\) are disjoint.
	
	Let \(C\) be an irreducible component of \(X\), and \(x\in C\) be any point. Let \(U_x\) be an open neighborhood of \(x\) which meets only finitely many irreducible components \(C=C_1,C_2,\ldots, C_n\) of \(X\). Each \(C_i\bigcap U_x\) is closed in \(U_x\), and does not intersect other \(C_j\bigcap U_x\ (j\ne i)\). Then \(C\bigcap U_x\) is open in \(U_x\), and thus open in \(X\). Then we deduce that \(C\) is open in \(X\) since \(C = \bigcup_{x\in C}(C\bigcap U_x)\).
\end{proof}

\begin{Def}
	Let \(X\) be a scheme. An \emph{quasi-coherent \(\sO_X\)-algebra of finite type} on \(X\) is a quasi-coherent \(\sO_X\)-module of finite type (in the sense of \cite{EGA1} Chapter 0 \S5) such that for each open subset \(U\subseteq X\), \(\sA(U)\) is an \(\sO_X(U)\)-algebra, and the algebra structure is compatible with restriction morphisms.
\end{Def}

For simplicity, we introduce the following assumption:

\hypertarget{star assumption}{(*)} \(X\) is a normal scheme such that every point of \(X\) has a neighborhood which meets only finitely many irreducible components of \(X\), and \(\sA\) is a quasi-coherent \(\sO_X\)-algebra of finite type on \(X\).

For a scheme \(X\), if \(X\) and \(\sA = \sO_X\) satisfy \hyperlink{star assumption}{(*)}, we will simply say that \(X\) satisfies \hyperlink{star assumption}{(*)}.

For example, if \(X\) is a normal integral scheme or a normal locally noetherian scheme, then \(X\) satisfies \hyperlink{star assumption}{(*)}.

\begin{Prop}\label{raw Norm}
	Suppose that \(X\) and \(\sA\) satisfy \hyperlink{star assumption}{(*)}. Then there is a morphism \(\Nm_{\sA}\colon \sA \to \sO_X\) which satisfies the following properties:
	
	(1) for each open subset \(U\subseteq X\) and \(f,g\in\sA(U)\),
	\[\Nm_{\sA}(fg) = \Nm_{\sA}(f)\Nm_{\sA}(g); \]
	
	(2) for each open subset \(U\subseteq X\)
	\[\Nm_{\sA}(1) = 1; \]
	
	(3) for each connected open subset \(U\subseteq X\) and \(s\in\sO(U)\),
	\[ \Nm_{\sA}(s\cdot1) = s^n, \]
	where \(n\in\bZ_{\geqslant0}\) and only depends on the connected component of \(X\) that \(U\) lies in.
\end{Prop}
\begin{proof}
	By Lemma \ref{disjoint union of irreducible components}, we know that the irreducible components of \(X\) are disjoint, and each irreducible component is open in \(X\). Thus sheaves on \(X\) split into sheaves on each irreducible components of \(X\). Moreover, each irreducible component \(X\) is also a connected component of \(X\). Thus we only need to construct \(\Nm_{\sA}\) on each irreducible component of \(X\). Therefore, we can assume that \(X\) is a normal integral scheme and \(\sA\) is a quasi-coherent \(\sO_X\)-algebra of finite type. It remains to construct a morphism \(\Nm_{\sA}\colon \sA \to \sO_X\) which satisfies (1), (2), and
	
	\((3)'\) for each open subset \(U\subseteq X\) and \(s\in\sO(U)\),
	\[ \Nm_{\sA}(s\cdot1) = s^n, \]
	where \(n\in\bZ_{\geqslant0}\) is a constant.
	
	Let \(\sR_X\) be the sheaf of rational functions on \(X\). Then \(\sR_X\) is a constant sheaf, and the restriction morphisms are identity. For each nonempty open subset \(U\subseteq X\), \(\sO_X(U)\) is an integral domain, \(\sR_X(U)\) is the field of fraction of \(\sO_X(U)\), and \(\sO_X(U)\) is integrally closed in \(\sR_X(U)\). Since \(\sA\) is quasi-coherent \(\sO_X\)-module of finite type, by \cite{EGA1} Chapter 1 (7.3.6), \(\sA\otimes_{\sO_X}\sR_X\) is isomorphic to \(\sR_X^n\) for some \(n\in\bZ_{\geqslant0}\).
	
	For each nonempty open subset, \(\sA\otimes_{\sO_X}\sR_X(U)\) is a \(n\)-dimensional \(\sR_X(U)\)-vector space, and each \(f\in \sA(U)\) defines a \(\sR_X(U)\)-linear homomorphism
	\[ L_f\colon \sA\otimes_{\sO_X}\sR_X(U)\to \sA\otimes_{\sO_X}\sR_X(U),\quad v\mapsto fv. \]
	Then we define
	\[ \Nm_\sA(f) = \det(L_f). \]
	Since \(\sO_X(U)\) is integrally closed, by \cite{EGA1} Chapter 1 (6.4.3), \(\det(L_f)\in \sO_X(U)\). Thus \(\Nm_{\sA}\) sends \(\sA(U)\) to \(\sO_X(U)\). Since \(\sA\) is an \(\sO_X\)-algebra, \(\Nm_\sA\) defines a morphism \(\sA\to \sO_X\). And (1), (2), \((3)'\) can be easily deduced from the properties of determinants of linear homomorphisms. Notes that the \(n\) in \((3)'\) is exactly the rank of \(\sA\otimes_{\sO_X}\sR_X\) as an \(\sR_X\)-module.
\end{proof}

For a general ringed space \((X,\sA)\), let \(\Inv(X,\sA)\) denote the category of invertible \(\sA\)-modules (defined as locally free \(\sA\)-sheaves of rank 1, following \cite{EGA1} Chapter 1 \S5.4) with morphisms being \(\sA\)-module homomorphisms. We avoid using the word ``Picard" since Picard groups and categories usually only concern isomorphisms between invertible sheaves. If \(X\) is a scheme, we abbreviate \(\Inv(X,\sO_X)\) to \(\Inv(X)\).

\begin{Def}\label{def: regular section}
	Let \((X,\sA)\) be a ringed space, \(\sL\) be an invertible \(\sA\)-module, and \(s\in\Gamma(X,\sL)\) be a global section. We say that \(s\) is \(\sA\)-\emph{regular} if the morphism \(\sA_X\xrightarrow{\cdot s}\sL\), multiplying by \(s\), is injective. If \(X\) is a scheme and \(\sA = \sO_X\), we will simply say that \(s\) is \emph{regular}.
\end{Def}

\begin{Prop}
	Suppose that \(X\) and \(\sA\) satisfy \hyperlink{star assumption}{(*)}. Then we can associate each invertible \(\sA\)-module \(\sL\) with an invertible \(\sA\)-module \(\Nm_\sA(\sL)\) such that
	
	(1) \(\Nm_{\sA}(\sA) = \sO_X\), which shows that our abuse of notation is reasonable;
	
	(2) for two invertible \(\sA\)-modules \(\sL_1\) and \(\sL_2\), \(\Nm_\sA(\sL_1\otimes_{\sA}\sL_2)\) is canonically isomorphic to \(\Nm_\sA(\sL_1)\otimes_{\sO_X}\Nm_\sA(\sL_2)\);
	
	(3) for each homomorphism \(h\colon\sL_1\to\sL_2\) of  invertible \(\sA\)-modules, we have a functorial homomorphism \(\Nm_{\sA}(h)\colon\Nm_{\sA}(\sL_1)\to\Nm_{\sA}(\sL_2)\) of  invertible \(\sO_X\)-modules;
	
	(4) if \(h\) in (3) is an isomorphism, then \(\Nm_{\sA}(h)\) is also an isomorphism;
	
	(5) if \(h\) in (4) is injective, then \(\Nm_{\sA}(h)\) is also injective. In particular, we can associate each \(\sA\)-regular section \(s\) of an invertible \(\sA\)-module \(\sL\) with a regular section \(\Nm_\sA(s)\) of \(\Nm_\sA(\sL)\).
	
	In conclusion, \(\Nm_{\sA}\) is a covariant functor from \(\Inv(X,\sA)\) to \(\Inv(X)\), and preserves tensor products and injective homomorphisms.
\end{Prop}
\begin{proof}
	By Lemma \ref{disjoint union of irreducible components}, we can assume that \(X\) is a normal integral scheme.
	
	(1), (2), (3), and (4) can be directly deduced from the arguments in \cite{EGA2} (6.5.2) and (6.5.3). Although \(X\) is assumed to be normal and locally noetherian in \cite{EGA2} (6.5.2) and (6.5.3), the only application of this condition is to ensure that \(\sR_X\)---the sheaf of rational functions on \(X\)---is a quasi-coherent \(\sO_X\)-module. And this can also deduced from our assumption (*) by \cite{EGA1} Chapter 1 (7.3.3). Thus the arguments in \cite{EGA2} (6.5.2) and (6.5.3) work under our conditions.
	
	To prove (5), it suffices to check on small enough affine open subsets. Suppose that \(X = \Spec A\), \(\Gamma(X,\sA) = B\), and \(\sL_1 = \sL_2 = \sA\), where \(A\) is an integrally closed domain, and \(B\) is a finite \(A\)-algebra (in the sense of \cite{ICA} Chapter 2). Then \(h\) corresponds to an injective \(B\)-module endomorphism \(u\colon B\to B\). Let \(f = u(1)\). Then \(\Nm_{B/A}(u)\colon A\to A\) is multiplying by \(\Nm_{B/A}(f)\). Note that \(\Nm_{B/A}(f)\in A\). Let \(K\) be the field of fractions of \(A\). Then \(\Nm_{B/A}(f) = \det(u\otimes_A\id_K\colon B\otimes_A K\to B\otimes_A K)\). Since \(K\) is the localization of \(A\) at the prime ideal (0), \(u\otimes_A\id_K\) is an injective \(K\)-vector space homomorphism, and thus \(\Nm_{B/A}(f)\ne 0\). Therefore, \(\Nm_{B/A}(u)\) is injective, i.e. \(\Nm_\sA(h)\) is injective.
\end{proof}

\begin{Prop}\label{Norm of extension is power}
	Let \(X\) be a normal integral scheme and \(\sA\) be a quasi-coherent \(\sO_X\)-algebra of finite type on \(X\). Then there exists an integer \(n\) such that for each invertible \(\sO_X\)-module \(\sL\), \(\Nm_\sA(\sL\otimes_{\sO_X}\sA)\) is canonically isomorphic to \(\sL^{\otimes n}\).
\end{Prop}
\begin{proof}
	Directly implied by Proposition \ref{raw Norm} (3). 
\end{proof}



\begin{Thm}\label{Norm functor}
	Suppose that \(X\) satisfies \hyperlink{star assumption}{(*)}, and let \(f\colon X'\to X\) be a finite morphism of schemes. Then we have a covariant functor 
	\[ \Nm_f\colon \Inv(X') \to\Inv(X),\]
	which preserves tensor products and injective homomorphisms.
\end{Thm}
\begin{proof}
	Let \(\sA = f_*\sO_X\). Since \(f\) is a finite morphism, \(\sA\) is a quasi-coherent \(\sO_X\)-algebra of finite type. Moreover, for each invertible \(\sO_{X'}\)-module \(\sL\) on \(X'\), by \cite{EGA2} (6.1.2) we know that \(f_*\sL\) is an invertible \(\sA\)-module on \(X\). Thus \(f_*\) induces a covariant functor from \(\Inv(X')\) to \(\Inv(X, \sA)\). It is obvious that \(f_*\) preserves tensor products and injective homomorphisms. We define \(\Nm_f\) to be the composition of \(f_*\) and \(\Nm_{\sA}\), i.e.
		\[ \Nm_{f} = \Nm_{\sA}\circ f_*\colon\Inv(X')\to \Inv(X),\quad \sL\mapsto \Nm_{\sA}(f_*\sL). \]
	Since both \(f_*\) and \(\Nm_\sA\) preserves tensor products and injective homomorphisms, so does \(\Nm_f\).
\end{proof}

Sometimes we also use \(\Nm_{X'/X}\) to denote \(\Nm_f\).

\begin{Prop}\label{Norm of pull-back is power}
	Let \(X\) be a normal integral scheme, and \(f\colon X'\to X\) be a finite morphism of schemes. Then there exists an integer \(n\) such that for each invertible \(\sO_X\)-module \(\sL\), \(\Nm_f(f^*\sL) = \sL^{\otimes n}\).
\end{Prop}
\begin{proof}
	Directly implied by Proposition \ref{Norm of extension is power}. 
\end{proof}


\begin{Prop}[\cite{EGA2} (6.5.8)]\label{base change of Norm}
	Suppose that \(X\) and \(Y\) satisfy \hyperlink{star assumption}{(*)}. Let \(f\colon X'\to X\) be a finite morphism, and \(g\colon Y\to X\) be a morphism of schemes such that \(g\) restricting to each irreducible component of \(Y\) is a dominant morphism to an irreducible component of \(X\). Let \(Y' = X'\times_XY \), and consider the Cartesian diagram
	\[ \begin{tikzcd}
		X'\arrow[d, "f"] & Y'\arrow[l, "g'"']\arrow[d, "f'"]\\
		X & Y\arrow[l, "g"'].
	\end{tikzcd} \]
    Then for each invertible \(\sO_{X'}\)-module \(\sL\),
    \[ \Nm_{f'}(g'^*\sL) = g^*\Nm_f(\sL). \]
\end{Prop}
\begin{proof}
	The noetherian assumption in \cite{EGA2} (6.5.8) only used for the definition of the norm functor. And its proof of such an isomorphism is applicable here.
\end{proof}

\section{Norm functor in relative dimension 1 case}

\begin{Lem}\label{unit exists}
	Let \(A\) be a noetherian integrally closed domain, let \(M,M'\) be finite \(A\)-modules, and let \(u,u'\) be \(A\)-module endomorphisms of \(M, M'\) respectively. If \(\coker(u)\cong \coker(u')\) as \(A\)-modules, then there exists a unit \(a\in A\) such that \(\det(u) = a\det(u')\).
\end{Lem}
\begin{proof} Let \(K\) be the field of fractions of \(A\), and let \(M_{\tor},M'_{\tor}\) be the torsion submodules of \(M,M'\) respectively. Since \(M\otimes_A K\) is canonically isomorphic to \((M/M_{\tor})\otimes_A K\), if we replace \(M,M'\) by \(M/M_{\tor},M'/M'_{\tor}\) respectively, the conditions and conclusions of the lemma does not change. Thus we can assume that \(M\) and \(M'\) are torsionfree.
	
Moreover, we can assume that \(u\) and \(u'\) are injective. Let \(u_K = u\otimes_A\id_K\) and \(u'_K = u'\otimes_A\id_K\). Then \(\coker(u_K) = \coker(u)\otimes_AK\cong\coker(u')\otimes_AK = \coker(u'_K)\). Thus if one of \(u\) and \(u'\) is not injective, so is the other, and then \(\det(u) = \det(u') = 0\). In this case, we can take \(a = 1\).
	
We consider the case that \(A\) is a PID. Since \(M\) and \(M'\) are torsionfree, they are free \(A\)-modules. Since \(u\) and \(u'\) are injective, both \(\coker(u)\) and \(\coker(u')\) are torsion. Let \((q_1),(q_2),\ldots,(q_r)\) be the invariants of \(\coker(u)\) and \(\coker(u')\). Then both \(\det(u)\) and \(\det(u')\) differ from \(q_1q_2\ldots q_r\) by a unit in \(A\).  Thus there exists a unit \(a\in A\) such that \(\det(u) = a\det(u')\).

Now we consider the general case. Since \(u\) and \(u'\) are injective, both \(\det(u)\) and \(\det(u')\) are nonzero. So there exists \(a,b\in K\) such that \(\det(u) = a\det(u')\) and \(\det(u') = b\det(u)\). It suffices to prove that \(a,b\in A\). Since \(A\) is integrally closed,
\[ A = \bigcap_{\text{height}\, \fp = 1} A_{\fp}, \]
where the intersection is taken over all prime ideals \(\fp\subseteq A\) of height 1, we only need to prove that \(a,b\in A_{\fp}\) for all prime ideals \(\fp\subseteq A\) of height 1. Let \(u_\fp = u\otimes_A\id_{A_\fp}\) and \(u'_\fp = u'\otimes_A\id_{A_\fp}\). Then \(u_\fp\) and \(u'_\fp\) are injective, and \(\coker(u_\fp) = \coker(u)\otimes_A A_\fp\cong\coker(u')\otimes_A A_\fp = \coker(u'_{\fp})\). Since \(M_\fp\otimes_A K = M\otimes_A K\) and \(M'_\fp\otimes_AK = M'\otimes_AK\), we have \(\det(u_\fp) = \det(u)\) and \(\det(u'_\fp) = \det(u')\). Since \(A\) is integrally closed and the height of \(\fp\) is 1, \(A_\fp\) is a PID. Therefore, \(a,b\in A_\fp\).
\end{proof}

Let \(X\) be a scheme, let \(\sL\) be an invertible sheaf on \(X\), and let \(s\colon\sO_X\to\sL\) be a global section of \(\sL\). We denote by \(Z(s)\) the closed subscheme defined by \(s\), or more precisely, the closed subscheme defined by ideal sheaf \(\im(s\spcheck\colon \sL\spcheck\to\sO_X)\). Locally, if we identify \(\sL\) with \(\sO_X\), then the ideal sheaf of \(Z(s)\) is \(s\sO_X\).


\begin{Lem}\label{lem: finite morphism and nbhd}
	Let \(f\colon X\to S\) be a finite morphism, and let \(\sL\) be an invertible \(\sO_X\)-module. Then for every point \(t\in S\), there exists an open neighborhood \(U\) of \(t\) such that \(\sL|_{f^{-1}(U)}\cong \sO_{f^{-1}(U)}\).
\end{Lem}
\begin{proof}
	By \cite{EGA2} Chapter II (6.1.12.1), there is an open neighborhood \(V\) of \(f^{-1}(t)\) such that \(\sL_V \cong\sO_V\). Let \(Z = X\setminus V\) and \(U = S\setminus f(Z)\). Then \(Z\) is a closed subset of \(X\), and thus \(U\) is an open neighborhood of \(t\). Obviously, \(U\subseteq f(V)\), and thus \(f^{-1}(U)\subseteq V\). Therefore, \(\sL|_{f^{-1}(U)}\cong \sO_{f^{-1}(U)}\).
\end{proof}

\begin{Prop}\label{glue lemma}
	Let \(X,S\) be schemes, and suppose that \(S\) is normal and locally noetherian. Let \(f\colon X\to S\) be a surjective projective morphism of pure relative dimension \(1\). Let \(\sL_1\) and \(\sL_2\) be two invertible sheaves on \(X\), and let \(s_1\) and \(s_2\) be regular sections of \(\sL_1\) and \(\sL_2\) respectively. Let \(Z_1\) and \(Z_2\) be the subscheme defined by \(s_1\) and \(s_2\) respectively. Suppose that both \(Z_1\) and \(Z_2\) are finite over \(S\), and that both \(s_1|_{Z_2}\) and \(s_2|_{Z_1}\) are regular. Then there exists a unique isomorphism
	\[\alpha\colon\Nm_{Z_1/S}(\sL_2|_{Z_1})\xrightarrow{\sim}\Nm_{Z_2/S}(\sL_1|_{Z_2})\]
	such that
	\[ \alpha(\Nm_{Z_1/S}(s_2|_{Z_1})) = \Nm_{Z_2/S}(s_1|_{Z_2}). \]
\end{Prop}
\begin{proof}
	Since \(S\) satisfy \hyperlink{star assumption}{(*)}, we can assume that \(S\) is an integral scheme by Lemma \ref{disjoint union of irreducible components}.
	
	By Proposition \ref{Norm functor}, \(\Nm_{Z_1/S}(s_2|_{Z_1})\) and \(\Nm_{Z_2/S}(s_1|_{Z_2})\) are regular sections of \(\Nm_{Z_1/S}(\sL_2|_{Z_1})\) and \(\Nm_{Z_2/S}(\sL_1|_{Z_2})\) respectively. By looking at the generic fiber, we deduce that \(\alpha\) is unique. By uniqueness and gluing, we only need to construct \(\alpha\) on small enough affine open subsets.
	
	By Lemma \ref{lem: finite morphism and nbhd}, we can assume that \(S = \Spec A\), \(\sL_1 = \sO_{Z_1}\), and that \(\sL_2 = \sO_{Z_2}\). Then \(\Nm_{Z_2/S}(\sL_1|_{Z_2}) = \Nm_{Z_1/S}(\sL_2|_{Z_1}) = \sO_S\). Let \(\sI_1\) and \(\sI_2\) be the ideal sheaves defined by \(s_1\) and \(s_2\) respectively. We view \(s_1|_{Z_2}\) and \(s_2|_{Z_1}\) as morphisms
	\[ s_1|_{Z_2}\colon \sO_{Z_2}\to \sO_{Z_2}\quad \text{and} \quad s_2|_{Z_1}\colon \sO_{Z_1}\to \sO_{Z_1} \]
	respectively. Then
	\begin{multline*}
		\coker s_2|_{Z_1} = \sO_{Z_1}\otimes \sO_X/\sI_2 = \sO_{Z_1}\otimes \sO_{Z_2}\\
		= \sO_{Z_2}\otimes \sO_{Z_1} = \sO_{Z_2}\otimes \sO_X/\sI_1 = \coker s_1|_{Z_2}.
	\end{multline*}
    By Lemma \ref{unit exists}, there exists a unit \(a\in A\) such that \(\det(s_2|_{Z_1}) = a\det(s_1|_{Z_2})\). Then we can take \(\alpha\) to be multiplying by \(a^{-1}\).
\end{proof}

\section{Deligne pairing for equidimensional morphisms}

\subsection{Preliminary}

Let \(n\) be a positive integer, and let \(\sL_1,\sL_2,\ldots,\sL_n\) be invertible sheaves on \(X\). For each \(1\leqslant i \leqslant n\), let \(s_i\colon\sO_X\to\sL_i\) be a global section of \(\sL_i\). Let \(Z(s_1,\ldots,s_n)=\bigcap_{i=1}^n Z(s_i)\).  Locally, if we identify all \(\sL_i\)'s with \(\sO_X\), then the ideal sheaf of \(Z(s_1,\ldots,s_n)\) is \((s_1,\ldots,s_n)\).

\begin{Def}
	Let \(X\), \(\sL_i\)'s, and \(s_i\)'s be as above. For \(1\leqslant i \leqslant n\), abbreviate \(Z(s_1,\ldots,s_n)\) to \(Z_i\). We say that \(s_1,\ldots,s_n\) is a \emph{regular sequence} if \(s_1\) is a regular section of \(\sL_1\) and for each \(2\leqslant i \leqslant n\), \(s_i|_{Z_{i-1}}\) is a regular section of \(\sL_i|_{Z_{i-1}}\).
\end{Def}

Note that our definition of regular sections here is identical to that in Definition \ref{def: regular section}.

Locally, if \(X = \Spec A\) and we identify all \(\sL_i\)'s with \(\sO_X\), then \(s_1,\ldots,s_n\) is a regular sequence means that \(s_1\) is not a zero-divisor of \(A\) and for each \(2\leqslant i \leqslant n\), \(s_i\) is not a zero-divisor of \(A/(s_1,\ldots,s_{i-1})\), i.e. \(s_1,\ldots,s_n\) is an \(A\)-sequence in the sense of \cite{LAlg} Chapter IV \S A.4. 

\begin{Prop}\label{order lemma}
	Whether the sequence \(s_1,\ldots,s_n\) is regular is independent of its order.
\end{Prop}
\begin{proof}
	Refer to \cite{LAlg} Chapter IV \S A.4 p.60.
\end{proof}


%

\subsection{Deligne pairing of sufficiently ample invertible sheaves}

Let \(X\to S\) be a projective morphism of noetherian schemes, let \(\sL\) be an invertible \(\sO_X\)-module, let \(\sF\) be a locally free \(\sO_S\)-module, and let \(\varphi\colon f^*\sF\to\sL\) a surjective \(\sO_X\)-module homomorphism. Let \(\bP = \bP(\sF\spcheck)\) and \(\xi = \sO_{\bP}(1)\). Consider the following commutative diagram
\[ \begin{tikzcd}
	X_{\bP}\arrow[r,"\pi_X"]\arrow[d, "f_{\bP}"] & X\arrow[d, "f"]\\
	\bP\arrow[r, "\pi"] & S.
\end{tikzcd} \]
We have a natural surjective morphism \(\pi^*\sF\spcheck\to \xi\), which induce a natural morphism \(f_{\bP}^*(\xi\spcheck)\to f_{\bP}^*(\pi^*\sF) = \pi_X^*(f^*\sF)\). Composing with \(\pi_X^*(\varphi)\colon \pi_X^*(f^*\sF)\to\pi_X^*(\sL)\), we get a morphism \(f_{\bP}^*(\xi\spcheck)\to \pi_X^*(\sL)\). Tensoring with \(f_{\bP}^*(\xi)\), we get a morphism
\begin{equation}\label{eq:natrual section}
	\sigma\colon\sO_{X_\bP}=f_{\bP}^*(\xi\spcheck)\otimes f_{\bP}^*(\xi)\to \pi_X^*(\sL)\otimes f_{\bP}^*(\xi).
\end{equation}
We identify \(\sigma\) with a section of \(\sL\boxtimes\xi\coloneq\pi_X^*(\sL)\otimes f_{\bP}^*(\xi)\).

\begin{Prop}[Garc\'{i}a]
	Let \(\sL,\xi,\sigma\) be as above. Then \(\sigma\) is a regular section of \(\sL\boxtimes\xi\).
\end{Prop}
\begin{proof}
	Refer to \cite{Garcia2000} \S2 2.2.
\end{proof}

Now let \(f\colon X\to S\) be a surjective projective morphism of noetherian schemes of pure relative dimension \(d\). Let \(n\) be an integer between \(1\) and \(d\), and let \(\sL_1,\ldots,\sL_n\) be invertible sheaves on \(S\) such that for each \(1\leqslant i \leqslant n\), there exists a locally free \(\sO_S\)-module \(\sF_i\) and a surjection \(\varphi_i\colon f^*\sF_i\to \sL_i\). For each \(1\leqslant i\leqslant n\), let \(\bP_i = \bP(\sF\spcheck_i)\), \(\bP = \prod_{j=1}^n\bP_i\), \(\xi_i = \sO_{\bP_i}(1)\), and let \(\sigma_i\) be the regular section of \(\sL_i\boxtimes\xi_i\) as above. We may view \(\sigma_i\) as a section of \(\sL_i\boxtimes\xi_i\) on \(X_{\bP}= X\times_S\bP\) by base change. Let \(Z_i= Z(\sigma_1,\ldots,\sigma_i) = \bigcap_{j=1}^iZ(\sigma_j)\) be the closed subscheme of \(X_\bP\).

\begin{Prop}[Garc\'{i}a]\label{regular sequence on porduct}
	\(\sigma_1,\sigma_2,\ldots,\sigma_n\) is a regular sequence of \(\sL_1\boxtimes\xi_1,\sL_2\boxtimes\xi_2,\ldots,\sL_n\boxtimes\xi_n\) on \(X_{\bP}\).
\end{Prop}
\begin{proof}
	Refer to \cite{Garcia2000} Proposition 2.3.1.
\end{proof}

Now let \(f\colon X\to S\), \(\sF_i\), \(\varphi_i\) be as above, while \(\sL_1,\ldots,\sL_n\) are \(f\)-sufficiently ample invertible \(\sO_X\)-modules. By \cite{StackProject} Lemma 29.29.2, we know that \(f_{\bP}\colon X_\bP\to \bP\) is also of pure relative dimension \(d\). For each \(1\leqslant i \leqslant n\), Let 
\[\cV_i= \{ x\in\bP\mid f_\bP^{-1}(x)\cap Z_i \text{ is of pure dimension } d-i\}.\]


\begin{Prop}
	\(\cV_i\) is an open subset of \(\bP\).
\end{Prop}
\begin{proof}
	Let \(x\in\bP\), and let \(k(x)\) be its residue field. By \cite{EGA2} (4.6.13), we know that for any \(i\) and subscheme \(Z\) of \(f_{\bP}^{-1}(x)\), \(\sL_i|_Z\) is \(k(x)\)-ample. By \cite{EGA4-2} (5.3.1.3) and (5.1.8), 
	we deduce that each component of \( f_\bP^{-1}(x)\bigcap Z_i\) has dimension \(\geqslant d-i\). Thus
	\begin{equation}\label{eq:fiber of Psion}
		\cV_i = \{ x\in\bP\mid \dim f_\bP^{-1}(x)\cap Z_i \leqslant d-i\}.
	\end{equation}
	By semi-continuity (\cite{EGA4-3} (13.1.5)), we know that \(\cV_i\) is an open subspace of \(\bP\).
\end{proof}

From \eqref{eq:fiber of Psion} we can see that \(\cV_1\supseteq\cV_2\supseteq\cdots\supseteq \cV_n\).

Let \(\pi\colon\bP\to S\) be the natural projection.

\begin{Prop}[Garc\'{i}a]\label{fiber dim}
	For each \(s\in S\), \(\pi^{-1}(s)\setminus\cV_n\) is a closed subset of \(\pi^{-1}(s)\) of codimension \(\geqslant d-n+2\).
\end{Prop}
\begin{proof}
	Refer to \cite{Garcia2000} Proposition 2.4.1.(ii).
\end{proof}

In particular, for each \(1\leqslant i\leqslant n\), \(\cV_i\ne\emptyset\).

Now let \(n=d\), \(Z = Z_d\), and \(\cV = \cV_d\). Then we have a commutative diagram
\[ \begin{tikzcd}
	Z_{\cV}\arrow[r]\arrow[d] & Z\arrow[d] \\
	X_{\cV}\arrow[r]\arrow[d, "f_{\cV}"]  & X_{\bP}\arrow[r, "\pi_X"]\arrow[d, "f_{\bP}"]  & X\arrow[d, "f"]\\
	\cV\arrow[r] & \bP\arrow[r, "\pi"]  & S,
\end{tikzcd} \]
where the three squares are all pull-back diagrams. We know that \(Z_{\cV}\to \cV\) is a projective morphism of noetherian schemes of pure relative dimension \(0\). Thus \(Z_{\cV}\to \cV\) is a finite morphism.

Now suppose that \(S\) is normal. Then \(\bP\) and \(\cV\) are also normal. For any invertible \(\sO_X\)-module \(\sL_{d+1}\), we can define
\begin{equation}\label{raw deligne pairing}
	\cN = \Nm_{Z_\cV/\cV}(\sL_{d+1}|_{Z_{\cV}})
\end{equation}
by Theorem \ref{Norm functor}.

We digress a little to discuss a descent technique. Some letters will be used again to indicate how the technique will be applied.

\begin{Def}\label{descent up to torsion}
	Let \(n\) be an integer and let \(S\) be a noetherian integral scheme. For each \(1\leqslant i\leqslant n\), let \(\bP_i\to S\) be a projective space bundle (in the sense of \cite{GTM52} Chapter 2 \S7) over \(S\) and \(\xi_i = \sO_{\bP_i}(1)\). Let \(\bP = \prod_{i=1}^{n}\bP_i\), let \(\pi\colon\bP\to S\), and let \(\cV\) be an open subset of \(\bP\) such that for each \(s\in S\), \(\dim\pi^{-1}(s) - \dim \pi^{-1}(s)\setminus\cV \geqslant 2\). We say that an invertible \(\sO_\cV\)-module \(\cN\) \emph{descends to \(S\) up to \(\xi\)-torsion} if there exist integers \(r_i\ (1\leqslant i \leqslant n)\), an invertible \(\sO_S\)-module \(\sJ\), and an \(\sO_{\cV}\)-module isomorphism
	\begin{equation}\label{eq:def od descent}
		\lambda\colon \cN\xrightarrow{\sim}\pi^*\sJ\mathop{\boxtimes}_{i=1}^n\xi_i^{r_i}.
	\end{equation}
    Sometimes we will simply say that \(\cN\) descends to \(\sJ\).
\end{Def}

\begin{Lem}\label{isom of hom}
	Let \(\bP,S,\cV\) be as in Definition \ref{descent up to torsion}. Then
	\[ \Hom_S(\sJ,\sJ') = \Hom_\bP(\pi^*\sJ,\pi^*\sJ') = \Hom_\cV(\pi^*\sJ,\pi^*\sJ'). \]
\end{Lem}
\begin{proof}
	Refer to \cite{EGA4-4} (21.13.2) and (21.13.4).
\end{proof}

\begin{Prop}\label{uniqueness of descent}
	In Definition \ref{descent up to torsion}, all \(r_i\)'s are uniquely determined and \(\cJ\) is unique up to a unique isomorphism determined by \(\lambda\).
\end{Prop}
\begin{proof}
	Let \(s\in S\), and let \(k(s)\) be its residue field. Then
	\[ \Pic(\pi^{-1}(s)\cap\cV) = \Pic(\pi^{-1}(s)) = \Pic(\bP_{k(s)}) = \prod_{i=1}^n\Pic(\bP_{i,k(s)}) = \bZ^n. \]
	And we can show that all \(r_i\)'s are uniquely determined by restricting \(\cN\) to a fiber \(\pi^{-1}(s)\) over a point \(s\in S\). The uniqueness of \(\cJ\) is induced by Lemma \ref{isom of hom}.
\end{proof}

\begin{Lem}\label{SGA Lemma}
	Suppose that for each \(s\in S\), \(\dim\pi^{-1}(s) - \dim \pi^{-1}(s)\setminus\cV \geqslant 3\). Then every invertible \(\sO_\cV\)-module descends to \(S\) up to \(\xi\)-torsion.
\end{Lem}
\begin{proof}
	Refer to \cite{SGA2} Expos{\'e} XII Corollary 4.9 or \cite{Elkik} I.2.6 Lemma (c).
\end{proof}

Now we go back to our original problem. We assume that \(S\) is a noetherian normal integral scheme, and that \(\sL_{d+1}\) in \eqref{raw deligne pairing} is also \(f\)-sufficiently ample.

\begin{Prop}\label{descend lemma}
	The \(\cN\) defined by \eqref{raw deligne pairing} descends to \(S\) up to \(\xi\)-torsion.
\end{Prop}
\begin{proof}
	Let \(\bP' = \prod_{i=1}^{d-1}\bP_i\), \(Z' = Z(\sigma_1,\sigma_2,\ldots,\sigma_{d-1})\subseteq X_{\bP'}\), \(\pi'\colon\bP'\to S\), and
	\[\cV'= \{ x\in\bP'\mid f_{\bP'}^{-1}(x)\bigcap Z' \text{ is of pure dimension } 1\}.\]
	We see that \(Z = (Z'\times_S\bP_d)\bigcap Z(\sigma_d)\). Since the dimension of each component of each fiber decreases at most 1 after intersecting with \(Z(\sigma_d)\), we deduce that \(\cV\subseteq \cV'\times_S\bP_d = \bP(\pi'^*(\sF_d)\spcheck|_{\cV'})\), where \(\sF_d\) is the \(\sO_S\)-module generating \(\sL_d\). By Proposition \ref{fiber dim}, for each \(s\in S\), \(\pi'^{-1}(s)\setminus\cV_n\) is a closed subset of \(\pi'^{-1}(s)\) of codimension \(\geqslant 3\). Then we know that every invertible \(\sO_{\cV'}\)-module descends to \(S\) up to \(\xi'\)-torsion by Lemma \ref{SGA Lemma}. It remains to show that every invertible \(\sO_{\cV}\)-module descends to \(\cV'\) up to \(\xi_d\)-torsion. Let \(\bP_{\cV'} = \cV'\times_S\bP_d\), and \(Z'' = Z'_{\cV'}\times_S\bP_{d} =  Z'_{\cV'}\times_{\cV'}\bP_{\cV'} = (Z'\times_S\bP_d)\bigcap f_{\bP}^{-1}(\bP_{\cV'})\). Then we have commutative diagrams
	\[ \begin{tikzcd}
		Z_{\cV}\arrow[r]\arrow[d] & Z_{\bP_{\cV'}}\arrow[r]\arrow[d] & Z\arrow[d]\\
		X_{\cV}\arrow[r]\arrow[d] & X_{\bP_{\cV'}}\arrow[r]\arrow[d] & X_{\bP}\arrow[r]\arrow[d, "f_{\bP}"] & X\arrow[d, "f"]\\
		\cV\arrow[r] & \bP_{\cV'}\arrow[r] & \bP\arrow[r] & S
	\end{tikzcd} \]
    and 
    \[ \begin{tikzcd}
    	&Z_{\bP_{\cV'}}\arrow[d, hookrightarrow]\arrow[r, equal]&Z''\bigcap Z(\sigma_d)\\
    	Z_{\cV}\arrow[r]\arrow[ru, hookrightarrow]\arrow[d] & Z''\arrow[r]\arrow[d, "f_{\bP}"] & Z'_{\cV'}\arrow[d, "f_{\bP'}"]\\
    	\cV\arrow[r, hookrightarrow] & \bP_{\cV'}\arrow[r]& \cV'.
    \end{tikzcd}\]
	Note that all the small squares in the above two diagrams are Cartesian. We see that \(Z''\to\bP_{\cV'}\) is a projective morphism of relative dimension 1, \(\sL'=\sL_d|_{Z'_{\cV'}}\) is an \(f_{\bP'}\)-sufficiently ample invertible sheaf on \(Z'_{\cV'}\), \(\sigma_d\) is a regular section of \(\sL'\boxtimes\xi_d\), \(Z_{\bP_{\cV'}}\subseteq Z''\) is the zero locus of \(\sigma_d\), and \(\cV\) is the open subset of \(\bP_{\cV'}\) above which \(Z_{\cV}\) is finite. Therefore, it suffices to prove Proposition \ref{descend lemma} for \(d = 1\).
	
	Now we prove the \(d = 1\) case. For \(i=1,2\), let \((\varphi_i,\sF_i)\) be a presentation of \(\sL_i\), \(\bP_i = \bP(\sF_i\spcheck)\), \(\xi_i = \sO_{\bP_i}(1)\), \(\sigma_i\) be the natural section defined by \eqref{eq:natrual section}, \(Z_i = Z(\sigma_i)\subseteq X\times_S\bP_i\) the zero locus of \(\sigma_i\), \(\cV_i\) be the nonempty open subset of \(\bP_i\) above which \(Z_i\) is finite, and let \(Z'_i = Z_{i,\cV_i}\). Then we have a commutative diagram
	\[ \begin{tikzcd}
		\cV_1\times_S\cV_2\arrow[r]\arrow[d] & \bP_1\times_S\cV_2\arrow[r]\arrow[d] & \cV_2\arrow[d]\\
		\cV_1\times_S\bP_2\arrow[r]\arrow[d] & \bP_1\times_S\bP_2\arrow[r]\arrow[d] & \bP_2\\
		\cV_1\arrow[r] & \cV_1\times_S\bP_1.
	\end{tikzcd} \]

    Since \(S\) is normal, both \(\cV_1\) and \(\cV_2\) are normal. Then we can define 
    \[\cN_1 = \Nm_{\bP_1\times_SZ'_2/\bP_1\times_S\cV_2}(\sL_1\boxtimes\xi_1)\]
    and
    \[\cN_2 = \Nm_{Z'_1\times_S\bP_2/\cV_1\times_S\bP_2}(\sL_2\boxtimes\xi_2),\]
    with regular sections
    \[s_1=\Nm_{\bP_1\times_SZ'_2/\bP_1\times_S\cV_2}(\sigma_1)\quad\text{and}\quad s_2 =\Nm_{Z'_1\times_S\bP_2/\cV_1\times_S\bP_2}(\sigma_2)\]
    respectively, where we view \(\sigma_1\) and \(\sigma_2\) as sections of \(\sL_1\boxtimes\xi_1\) and \(\sL_2\boxtimes\xi_2\) respectively. Note that the zero locus of \(\sigma_1\) and \(\sigma_2\) on \(X\times_S\bP_1\times_S\bP_2\) are \(Z_1\times_S\bP_2\) and \(\bP_1\times_SZ_2\) respectively. By Proposition \ref{regular sequence on porduct}, \(\sigma_1,\sigma_2\) is a regular sequence. And then by Proposition \ref{order lemma}, \(\sigma_1|_{\bP_1\times_SZ'_2}\) and \(\sigma_2|_{Z'_1\times_S\bP_2}\) are regular sections of
    \[ \sL_1\boxtimes\xi_1|_{\bP_1\times_SZ'_2}\quad\text{and}\quad\sL_2\boxtimes\xi_2|_{Z'_1\times_S\bP_2} \]
    respectively. By Proposition \ref{glue lemma}, identifying \(s_1\) and \(s_2\), we can glue \(\cN_1\) and \(\cN_2\) on \(\cV_1\times_S\bP_2\bigcap \bP_1\times_S\cV_2 = \cV_1\times_S\cV_2\) to get an invertible sheaf \(\cN'\) on \(\cV= \cV_1\times_S\bP_2\bigcup \bP_1\times_S\cV_2 \). Let
    \[ \pi_1\colon\bP_1\to S,\quad \pi_2\colon\bP_2\to S,\quad \pi\colon\bP_1\times_S\bP_2\to S.\]
    be the natural projections, and let \(F_1,F_2,F\) be the complements of \(\cV_1,\cV_2,\cV\) in \(\bP_1,\bP_2,\bP_1\times_S\bP_2\) respectively. By definition, \(F = F_1\times_S F_2\). By Proposition \ref{fiber dim}, for \(i = 1,2\) and each \(s\in S\), \(F_{i,s}\) is a closed subset of \(\pi_i^{-1}(s)\) of codimension \(\geqslant 2\). Thus \(F_{s}\) is a closed subset of \(\pi^{-1}(s)\) of codimension \(\geqslant 4\). Then by Lemma \ref{SGA Lemma}, \(\cN\) descends to \(S\) up to \(\xi\)-torsion, i.e. there exist integers \(r_1,r_2\), and an invertible \(\sO_S\)-module \(\sJ\) such that
    \[ \cN'\cong \pi^*\sJ\boxtimes\xi_1^{r_1}\boxtimes\xi_2^{r_2} \]
    on \(\cV\). Thus
    \[ \cN_2\cong \pi^*\sJ\boxtimes\xi_1^{r_1}\boxtimes\xi_2^{r_2} \]
    on \(\cV_1\times_S\bP_2\). Take a closed point \(x\) of \(\bP_2\). Then we get a closed immersion \(\iota_x\colon\bP_1\to\bP_1\times\bP_2\) and an isomorphism
    \begin{equation}\label{eq:descent of cN_2}
    	\iota_x^*\,\cN_2  \cong \pi_1^*\sJ\boxtimes\xi_1^{r_1}
    \end{equation}
    on \(\cV_1\).
    
    Since \(S\) is an integral scheme, so are \(\bP_1\), \(\cV_1\), and \(\cV_1\times_S\bP_2\). Thus by Proposition \ref{base change of Norm} and the following Cartesian diagram
    \[ \begin{tikzcd}
    	Z'_1\arrow[d] &Z'_1\times_S\bP_2\arrow[l]\arrow[d]\\
    	\cV_1 & \cV_1\times_S\bP_2\arrow[l,"p"'],
    \end{tikzcd} \]
    we have an isomorphism
    \[ \Nm_{Z'_1\times_S\bP_2/\cV_1\times_S\bP_2}(\sL_2) = p^*\Nm_{Z'_1/\cV_1}(\sL_2). \]
    Since \(p\circ\iota_x = \id\), the above equation and \eqref{eq:descent of cN_2} implies that
    \begin{align}\label{eq:descent of cN_2 ii}
    	\Nm_{Z'_1/\cV_1}&(\sL_2)\otimes\iota_x^*\,\Nm_{Z'_1\times_S\bP_2/\cV_1\times_S\bP_2}(p_1^*\,\xi_2)\notag\\
    	&= \iota_x^*\Nm_{Z'_1\times_S\bP_2/\cV_1\times_S\bP_2}(\sL_2)\otimes\iota_x^*\,\Nm_{Z'_1\times_S\bP_2/\cV_1\times_S\bP_2}(p_1^*\,\xi_2)\notag\\
    	&= \iota_x^*\Nm_{Z'_1\times_S\bP_2/\cV_1\times_S\bP_2}(\sL_2\boxtimes\xi_2)\notag\\
    	&=\iota_x^*\cN_2 \cong \pi_1^*\cJ\boxtimes\xi_1^{r_1}
    \end{align}
    on \(\cV_1\), where \(p_1\colon Z'_1\times_S\bP_2\to\bP_2\).
    
    From the following commutative diagram
    \[ \begin{tikzcd}
    	& Z'_1\times_S\bP_2\arrow[r, "p_1"]\arrow[d, "q"] & \bP_2\arrow[d, equal]\\
    	\cV_1\arrow[r, "\iota_x"]& \cV_1\times_S\bP_2\arrow[r, "p_2"] & \bP_2
    \end{tikzcd} \]
    and Proposition \ref{Norm of pull-back is power}, we deduce that
    \begin{multline*}
    	\iota_x^*\,\Nm_{Z'_1\times_S\bP_2/\cV_1\times_S\bP_2}(p_1^*\,\xi_2) = \iota_x^*\,\Nm_{Z'_1\times_S\bP_2/\cV_1\times_S\bP_2}(q^*p_2^*\,\xi_2)\\
    	= \iota_x^*(p_2^*\,\xi_2)^{\otimes m}
    	= \iota_x^*p_2^*(\xi_2^{\otimes m})
    	= \sO_{\cV_1},
    \end{multline*}
    where \(m\) is an integer.
    Thus \eqref{eq:descent of cN_2 ii} becomes
        \begin{equation}\label{eq:descent of cN_2 iii}
    	\Nm_{Z'_1/\cV_1}(\sL_2)\cong \pi_1^*\sJ\boxtimes\xi_1^{r_1}
    \end{equation}
    on \(\cV_1\). Therefore, \(\Nm_{Z'_1/\cV_1}(\sL_2|_{Z'_1})\), which is exactly \(\cN\) defined by \eqref{raw deligne pairing}, descends to \(S\) up to \(\xi_1\)-torsion.
\end{proof}

By Proposition \ref{descend lemma}, there exists an invertible \(\sO_S\)-module \(\sJ\) and integers \(r_i\) such that \eqref{eq:def od descent} holds.

\begin{Prop}\label{indep of presentation}
	The \(\cJ\) defined by Proposition \eqref{descend lemma} is independent of the choice of presentations \((\varphi_i,\sF_i)\ (1\leqslant i\leqslant d)\) of \(\sL_i\)'s. 
	
	A more precise statement is as follows. For each choice of presentations \(\bfV = ((\varphi_i,\sF_i))_{1\leqslant i\leqslant d}\), we denote the \(\cJ\) defined by Proposition \eqref{descend lemma} by \(\cJ_{\bfV}\). Then for any two choices of presentations \(\bfV'\) and \(\bfV''\), we have a canonical isomorphism
	\[ \theta_{\bfV,\bfV'}\colon \cJ_{\bfV}\xrightarrow{\sim}\cJ_{\bfV'}\]
	satisfying \(\theta_{\bfV,\bfV} = \id\) and \(\theta_{\bfV',\bfV''} \circ \theta_{\bfV,\bfV'}=\theta_{\bfV,\bfV''}\). 
\end{Prop}
\begin{proof}
    Let \(\eta\) be the generic point of \(S\). Let \(X_\eta = X\times_S\Spec\sO_{S,\eta}\),
    \(\sL_{i,\eta} = \sL_i\times_S\Spec\sO_{S,\eta} = \sL_i|_{X_\eta}\), \(\sF_{i,\eta} =  \sF_i\times_S\Spec\sO_{S,\eta} = \sF_i|_{\eta}\). We show them in the following diagram
    \[ \begin{tikzcd}[column sep=small]
    	\sL_{i,\eta} & X_{\eta}\arrow[rr, hookrightarrow]\arrow[d, "f_\eta"] && X\arrow[d, "f"] & \sL_i\\
        \sF_{i,\eta} & \eta \arrow[rr, hookrightarrow] && S & \sF_i
    \end{tikzcd} \]
    \[ \varphi_{i,\eta}\colon f_\eta^*\sF_{i,\eta}\twoheadrightarrow \sL_{i,\eta},\qquad \varphi_{i}\colon f^*\sF_{i}\twoheadrightarrow \sL_{i}.\]
    We see that for \(1\leqslant i\leqslant d\), \((\varphi_{i,\eta},\sF_{i,\eta})\) is a presentation of \(\sL_{i,\eta}\). Therefore, we can apply the above operation for 
    \[f\colon X\to S,\quad \sL_i,\ (\varphi_i,\sF_i),\ (1\leqslant i\leqslant d),\ \sL_{d+1}\]
    to
    \[f_\eta\colon X_\eta\to \eta,\quad \sL_{i,\eta},\ (\varphi_{i,\eta},\sF_{\eta}),\ (1\leqslant i\leqslant d),\ \sL_{d+1, \eta},\]
    and then we can get a commutative diagram
    \[ \begin{tikzcd}
    	Z_{\cV, \eta}\arrow[r]\arrow[d] & Z_\eta\arrow[d] \\
    	X_{\cV,\eta}\arrow[r]\arrow[d, "f_{\cV,\eta}"]  & X_{\bP,\eta}\arrow[r, "\pi_{X,\eta}"]\arrow[d, "f_{\bP,\eta}"]  & X_\eta\arrow[d, "f_\eta"]\\
    	\cV_\eta\arrow[r] & \bP_\eta\arrow[r, "\pi_\eta"]  & \eta,
    \end{tikzcd} \]
    an invertible \(\sO_{V_\eta}\)-module
    \[	\cN_\eta = \Nm_{Z_{\cV,\eta}/\cV_\eta}(\sL_{d+1,\eta}|_{Z_{\cV,\eta}}), \]
    and an invertible \(\sO_\eta\)-module \(\cJ_\eta\) which \(\cN_\eta\) descends to up to \(\xi_\eta\) torsion.
    
    Since all the operations are commutative with the base change \(\eta\to S\), the subscript \(\eta\) not only indicates the objects are constructed in the \(\eta\) case, but also means that they are exactly the base change of corresponding objects in the \(S\) case along \(\eta \to S\), i.e.
    \[ \begin{tikzcd}
    Z_{\cV}\arrow[r]\arrow[d] & Z\arrow[d] \\
    X_{\cV}\arrow[r]\arrow[d, "f_{\cV}"]  & X_{\bP}\arrow[r, "\pi_X"]\arrow[d, "f_{\bP}"]  & X\arrow[d, "f"]\\
    \cV\arrow[r] & \bP\arrow[r, "\pi"]  & S
    \end{tikzcd} \xRightarrow{\times_S \eta}  \begin{tikzcd}
    Z_{\cV, \eta}\arrow[r]\arrow[d] & Z_\eta\arrow[d] \\
    X_{\cV,\eta}\arrow[r]\arrow[d, "f_{\cV,\eta}"]  & X_{\bP,\eta}\arrow[r, "\pi_{X,\eta}"]\arrow[d, "f_{\bP,\eta}"]  & X_\eta\arrow[d, "f_\eta"]\\
    \cV_\eta\arrow[r] & \bP_\eta\arrow[r, "\pi_\eta"]  & \eta.
    \end{tikzcd} \]
    We see that \(\cN_\eta = \cN \times_\cV\cV_\eta\). Then by Proposition \ref{uniqueness of descent}, we have a canonical isomorphism \(\cJ_\eta = \cJ\times_S\eta\). Since \(\eta\) is a regular scheme, \(\cJ_\eta\) is exactly the Deligne pairing \(\langle\sL_{1,\eta},\sL_{2,\eta},\ldots,\sL_{d+1,\eta}\rangle_{X_\eta/\eta}\) constructed in \cite{Garcia2000}, and thus is independent of the choice of representation \(\bfV\). For clarification, we use \(\psi_\eta\)  to denote the following canonical isomorphism
    \[ \psi_\eta\colon \cJ\times_S \eta\to \langle\sL_{1,\eta},\sL_{2,\eta},\ldots,\sL_{d+1,\eta}\rangle_{X_\eta/\eta}. \]
    
    Let \(S^{(1)}\) be the set of all codimension 1 points of \(S\). Let \(P\in S^{(1)}\), and let \(A = \sO_{S,P}\), \(X_{A} = X\times_S\Spec A\),
    \(\sL_{i,A} = \sL_i\times_S\Spec A = \sL_i|_{X_A}\), \(\sF_{i,A} =  \sF_i\times_S\Spec A = \sF_i|_{\Spec A}\). Applying the above process for \( \eta\to S\) to \(\Spec A\to S\), we get the following diagrams
    \[ \begin{tikzcd}[column sep=small]
    	\sL_{i,A} & X_{A}\arrow[rr]\arrow[d, "f_A"] && X\arrow[d, "f"] & \sL_i\\
    	\sF_{i,A} & \Spec A \arrow[rr] && S & \sF_i
    \end{tikzcd} \]
 \[ \varphi_{i,A}\colon f_A^*\sF_{i,A}\twoheadrightarrow \sL_{i,A},\qquad \varphi_{i}\colon f^*\sF_{i}\twoheadrightarrow \sL_{i},\]
     \[ \begin{tikzcd}
 	Z_{\cV}\arrow[r]\arrow[d] & Z\arrow[d] \\
 	X_{\cV}\arrow[r]\arrow[d, "f_{\cV}"]  & X_{\bP}\arrow[r, "\pi_X"]\arrow[d, "f_{\bP}"]  & X\arrow[d, "f"]\\
 	\cV\arrow[r] & \bP\arrow[r, "\pi"]  & S
 \end{tikzcd} \xRightarrow{\times_S \Spec A}  \begin{tikzcd}
 	Z_{\cV, A}\arrow[r]\arrow[d] & Z_A\arrow[d] \\
 	X_{\cV,A}\arrow[r]\arrow[d, "f_{\cV,A}"]  & X_{\bP,A}\arrow[r, "\pi_{X,A}"]\arrow[d, "f_{\bP,A}"]  & X_A\arrow[d, "f_A"]\\
 	\cV_A\arrow[r] & \bP_A\arrow[r, "\pi_A"]  & A.
 \end{tikzcd} \]
 an invertible \(\sO_{V_A}\)-module
\[	\cN_A = \Nm_{Z_{\cV,A}/\cV_A}(\sL_{d+1,A}|_{Z_{\cV,A}}), \]
an invertible \(A\)-module \(\cJ_A\) which \(\cN_A\) descends to up to \(\xi_A\) torsion, and a canonical isomorphism \(\cJ_A = \cJ\times_S\Spec A\). Since \(\Spec A\) is a regular scheme, \(\cJ_A\) is exactly the Deligne pairing \(\langle\sL_{1,A},\sL_{2,A},\ldots,\sL_{d+1,A}\rangle_{X_A/\Spec A}\) constructed in \cite{Garcia2000}, and thus is independent of the choice of representation \(\bfV\). We use \(\psi_P\)  to denote the following canonical isomorphism
\[ \psi_P\colon \cJ\times_S \Spec\sO_{S,P}\to \langle\sL_{1,A},\sL_{2,A},\ldots,\sL_{d+1,A}\rangle_{X_A/\Spec A}. \]
Proposition \ref{uniqueness of descent} tells us that the base change of \(\psi_P\) along \(\eta\to \Spec A\) is exactly the canonical isomorphism \(\psi_\eta\), and thus we have the following commutative diagram
\[ \begin{tikzcd}
	\cJ\times_S \Spec\sO_{S,P}\arrow[r, "\psi_P"]\arrow[d, hookrightarrow]& \langle\sL_{1,A},\sL_{2,A},\ldots,\sL_{d+1,A}\rangle_{X_A/\Spec A}\arrow[d, hookrightarrow]\\
	\cJ\times_S \eta\arrow[r, "\psi_\eta"]& \langle\sL_{1,\eta},\sL_{2,\eta},\ldots,\sL_{d+1,\eta}\rangle_{X_\eta/\eta}.
\end{tikzcd} \]

Now let \(\bfV'\) be another presentation. Then we get another \(\cJ_{\bfV'}\), canonical isomorphisms
\[ \psi'_\eta\colon \cJ_{\bfV'}\times_S \eta\to \langle\sL_{1,\eta},\sL_{2,\eta},\ldots,\sL_{d+1,\eta}\rangle_{X_\eta/\eta}\]
and
\[ \psi'_P\colon \cJ_{\bfV'}\times_S \Spec\sO_{S,P}\to \langle\sL_{1,A},\sL_{2,A},\ldots,\sL_{d+1,A}\rangle_{X_A/\Spec A} \]
for each \(P \in S^{(1)}\), and a commutative diagram
\[ \begin{tikzcd}
	\cJ_{\bfV'}\times_S \Spec\sO_{S,P}\arrow[r, "\psi'_P"]\arrow[d, hookrightarrow]& \langle\sL_{1,A},\sL_{2,A},\ldots,\sL_{d+1,A}\rangle_{X_A/\Spec A}\arrow[d, hookrightarrow]\\
	\cJ_{\bfV'}\times_S \eta\arrow[r, "\psi'_\eta"]& \langle\sL_{1,\eta},\sL_{2,\eta},\ldots,\sL_{d+1,\eta}\rangle_{X_\eta/\eta}
\end{tikzcd} \]
for each \(P \in S^{(1)}\).

Let \(\theta_\eta = {\psi'_\eta}^{-1}\circ \psi_\eta\) and \(\theta_P = {\psi'_P}^{-1}\circ \psi_P\) for each \(P\in S^{(1)}\). Then we have canonical isomorphisms
\[ \theta_\eta\colon \cJ_{\bfV}\times_S \eta\to \cJ_{\bfV'}\times_S \eta\]
and
\[ \theta_P\colon \cJ_{\bfV}\times_S \Spec\sO_{S,P}\to\cJ_{\bfV'}\times_S \Spec\sO_{S,P} \]
for each \(P \in S^{(1)}\), and a commutative diagram
\[ \begin{tikzcd}
	\cJ_{\bfV}\times_S \Spec\sO_{S,P}\arrow[r, "\theta_P"]\arrow[d, hookrightarrow]& \cJ_{\bfV'}\times_S \Spec\sO_{S,P}\arrow[d, hookrightarrow]\\
	\cJ_{\bfV}\times_S \eta\arrow[r, "\theta_\eta"]& \cJ_{\bfV'}\times_S\eta.
\end{tikzcd} \]
for each \(P \in S^{(1)}\).

Since \(S\) is a noetherian  normal integral scheme, and \(\cJ_{\bfV}\) and \(\cJ_{\bfV'}\) are invertible \(\sO_S\)-modules, we have
\[ \cJ_{\bfV} = \bigcap_{P\in S^{(1)}}\cJ_{\bfV}\times_S \Spec\sO_{S,P} \]
and
\[ \cJ_{\bfV'} = \bigcap_{P\in S^{(1)}}\cJ_{\bfV'}\times_S \Spec\sO_{S,P}, \]
where the two intersections are taken in \(\cJ_{\bfV}\times_S \eta\) and \(\cJ_{\bfV'}\times_S \eta\) respectively.
Thus \(\theta_{\eta}\) restricts to a canonical isomorphism
\[ \theta_{\bfV,\bfV'}\colon \cJ_{\bfV}\xrightarrow{\sim}\cJ_{\bfV'}.\]
It is easy to see that \(\theta_{\bfV,\bfV} = \id\) and \(\theta_{\bfV',\bfV''} \circ \theta_{\bfV,\bfV'}=\theta_{\bfV,\bfV''}\).
\end{proof}

If \(S\) is not irreducible, then by Lemma \ref{disjoint union of irreducible components} we can apply the above construction over each irreducible component of \(S\). Then we can also get a \(\cJ\) on \(S\).

If \(d = 0\), we directly set \(\langle\sL\rangle_{X/S} \coloneq \Nm_{X/S}(\sL)\).

We summarize our results in this subsection as the following theorem.

\begin{Thm}\label{Deligne Pairing for sufficiently ample line bundles}
	Let \(f\colon X\to S\) be a surjective projective morphism of noetherian schemes of pure relative dimension \(d\). Suppose that \(S\) is normal. Then we can associate any \(d+1\) \(f\)-sufficiently ample invertible \(\sO_X\)-modules \(\sL_1,\sL_2,\ldots,\sL_{d+1}\) with the invertible \(\sO_S\)-module \(\cJ\), and denote it by \(\langle\sL_1,\sL_2,\ldots,\sL_{d+1}\rangle_{X/S}\).
\end{Thm}

\subsection{Deligne pairing of general invertible sheaves}

For simplicity, we introduce some notations.

\begin{Def}
	Let \(f\colon X\to S\) be a surjective projective morphism. Suppose that \(S\) is integral and normal. Let \(\eta\) be the generic point of \(S\), and \(X_\eta = X\times_S\eta = f^{-1}(\eta)\). We use \(P\) to denote a point in \( S^{(1)}\), and let \(A = \Spec \sO_{S,P}\) and \(X_A = X\times_S\Spec A\). For an invertible \(\sO_S\)-module \(\sM\), we define
	\[ \sM_\eta = \sM|_{\eta} = \sM\times_S\eta \]
	and
	\[ \sM_P = \sM|_{\Spec A} = \sM\times_S\Spec A.\] For an invertible \(\sO_X\)-module \(\sL\), we define
	\[ \sL_\eta = \sL|_{X_\eta}\quad \text{and}\quad  \sL_P = \sL|_{X_A}.\]
\end{Def}

We present the above notations in the following diagram.
\[ \begin{tikzcd}[row sep=small]
	\sL_\eta & \sL_P & \sL\\
	X_\eta\arrow[r]\arrow[dd] & X_A\arrow[r]\arrow[dd] & X\arrow[dd]\\&&& \\
	\eta\arrow[r] & \Spec A\arrow[r] & S\\
	\sM_\eta & \sM_P &  \sM.
\end{tikzcd} \]

From the proof of Proposition \ref{indep of presentation}, we can deduce the following propositions.

\begin{Prop}\label{codim 1 of f.s.a equidimensional deligne pairing}
	Let \(f\colon X\to S\) be a surjective projective morphism of noetherian schemes of pure relative dimension \(d\). Suppose that \(S\) is integral and normal, and that \(\sL_1,\sL_2,\ldots,\sL_{d+1}\) are \(d+1\) \(f\)-sufficiently ample invertible \(\sO_X\)-modules. Then we have a commutative diagram
	\[ \begin{tikzcd}
		\langle\sL_{1},\sL_{2},\ldots,\sL_{n+1}\rangle_{X/S,\, P}\arrow[r, hookrightarrow]\arrow[d] & \langle\sL_{1},\sL_{2},\ldots,\sL_{n+1}\rangle_{X/S,\,\eta}\arrow[d]\\
		\langle\sL_{1,P},\sL_{2,P},\ldots,\sL_{n+1,P}\rangle_{X_A/\Spec A}\arrow[r, hookrightarrow] & \langle\sL_{1,\eta},\sL_{2,\eta},\ldots,\sL_{n+1,\eta}\rangle_{X_\eta/\eta},
	\end{tikzcd} \]
	where \(P\in S^{(1)}\), two vertical arrows are canonical isomorphisms, and two bottom objects are Deligne pairings defined in \cite{Garcia2000}.
\end{Prop}

The following proposition explains why we call our construction ``Deligne pairing" and justifies our notation.

\begin{Prop}\label{EDP = DP for f.s.a. lb}
	Let \(f\colon X\to S\) be a surjective projective morphism of noetherian schemes of pure relative dimension \(d\). Suppose that \(S\) is normal, \(f\) is of finite Tor-dimension, and that \(\sL_1,\sL_2,\ldots,\sL_{d+1}\) are \(d+1\) \(f\)-sufficiently ample invertible \(\sO_X\)-modules. Then our \(\langle \sL_1,\sL_2,\ldots,\sL_{d+1}\rangle_{X/S}\) is identical to the Deligne pairing defined in \cite{Garcia2000}.
\end{Prop}
\begin{proof}
	We can assume that \(S\) is integral, and let \(\eta\) be the generic point of \(S\). Let \(\sM\) be the Deligne pairing of \(\sL_1,\sL_2,\ldots,\sL_{d+1}\) defined by \cite{Garcia2000}. Then we have
	\[ \sM_\eta = \langle\sL_{1,\eta},\sL_{2,\eta},\ldots,\sL_{d+1,\eta}\rangle_{X_\eta/\eta}\]
	and
	\[ \sM_P =  \langle\sL_{1,P},\sL_{2,P},\ldots,\sL_{d+1,P}\rangle_{X_A/\Spec A}, \]
	where both \(\langle\sL_{1,\eta},\sL_{2,\eta},\ldots,\sL_{d+1,\eta}\rangle_{X_\eta/\eta}\) and \(\langle\sL_{1,P},\sL_{2,P},\ldots,\sL_{d+1,P}\rangle_{X_A/\Spec A}\) are Deligne pairings defined by \cite{Garcia2000}. Moreover, we have
	\[ \sM = \bigcap_{P\in S^{(1)}} \sM_P, \]
	where the intersection is taken in \(\sM_\eta\).
	
	Let \(\sN = \langle \sL_1,\sL_2,\ldots,\sL_{d+1}\rangle_{X/S}\) be our Deligne pairing. Then by Proposition \ref{codim 1 of equidimensional deligne pairing}, \(\sN_\eta = \sM_\eta\), \(\sN_P = \sM_P\) for every \(P\in S^{(1)}\), and
	\[ \sN = \bigcap_{P\in S^{(1)}} \sN_P = \bigcap_{P\in S^{(1)}} \sM_P = \sM. \qedhere\]
\end{proof}

From the above reasoning, we can summarize the following very useful proposition.

\begin{Prop}\label{codim 1 miracle for line bundles}
	Let \(S\) be a normal integral scheme, and let \(\eta\) be the generic point of \(S\). Let \(\sM\) and \(\sN\) be two invertible \(\sO_S\)-modules. Suppose that we have isomorphisms \(\varphi_\eta\colon \sM_{\eta}\to \sN_{\eta}\) and \(\varphi_P\colon \sM_{P}\to \sN_{P}\) for every \(P\in S^{(1)}\). Moreover, for every \(P\in S^{(1)}\), they satisfy the following commutative diagram
	\[ \begin{tikzcd}
		\sM_P\arrow[r, hookrightarrow]\arrow[d,"\varphi_P"] & \sM_\eta\arrow[d,"\varphi_\eta"]\\
		\sN_P\arrow[r, hookrightarrow] & \sN_\eta.
	\end{tikzcd} \]
	Then \(\varphi_\eta\) restricts to an isomorphism \(\varphi\colon \sM\to \sN\).
\end{Prop}

\begin{Def}
	Let \(f\colon X\to S\) be a projective morphism of noetherian schemes. Define \(\Pic(X)_{\fsa}\) to be the subset of isomorphism classes of \(f\)-sufficiently ample invertible \(\sO_X\)-modules of \(\Pic(X)\).
\end{Def}

Let \(f\colon X\to S\) be a surjective projective morphism of noetherian schemes of pure relative dimension \(d\). Suppose that \(S\) is normal. Then Theorem \ref{Deligne Pairing for sufficiently ample line bundles} defines a map
\begin{equation}\label{eq:EDP of f.s.a lb}
	\langle\sL_1,\sL_2,\ldots,\sL_{d+1}\rangle_{X/S}\colon\Pic(X)^{d+1}_{\fsa}\to\Pic(S).
\end{equation}
Note that if \(\sL,\sL'\in\Pic(X)_{\fsa}\), then \(\sL\otimes\sL'\in\Pic(X)_{\fsa}\).

\begin{Prop}\label{multi-additive of EDP of f.s.a lb}
	The map \eqref{eq:EDP of f.s.a lb} is multi-additive.
\end{Prop}
\begin{proof}
	We can assume that \(S\) is integral. Let \(1\leqslant i\leqslant d\) be an arbitrary integer, and let
	\[ \sM = \langle \sL_1,\ldots,\sL_i,\ldots\sL_{d+1}\rangle_{X/S}\otimes\langle \sL_1,\ldots,\sL'_i,\ldots\sL_{d+1}\rangle_{X/S} \]
	and
    \[ \sN = \langle \sL_1,\ldots,\sL_i\otimes\sL'_i,\ldots\sL_{d+1}\rangle_{X/S}, \]
	where all the terms are the same except the \(i\)-th term. Let \(\eta\) be the generic point of \(S\).
	Then we have a canonical isomorphism
	\[ \begin{split}
		\sN_\eta &= \langle \sL_{1,\eta},\ldots,\sL_{i,\eta}\otimes\sL'_{i,\eta},\ldots\sL_{d+1,\eta}\rangle_{X_\eta/\eta}\\
		&= \langle\sL_{1,\eta},\ldots,\sL_{i,\eta},\ldots\sL_{d+1,\eta}\rangle_{X_\eta/\eta}\\
		&\phantom{\mathrel{=}\langle\sL_{1,\eta},\ldots,\sL_{i,\eta},}\otimes\langle\sL_{1,\eta},\ldots,\sL'_{i,\eta},\ldots\sL_{d+1,\eta}\rangle_{X_\eta/\eta}\\
		&=  \langle \sL_1,\ldots,\sL_i,\ldots\sL_{d+1}\rangle_{X/S,\eta}\otimes \langle \sL_1,\ldots,\sL'_i,\ldots\sL_{d+1}\rangle_{X/S,\eta}\\
		&= \sM_\eta.
	\end{split} \]
	We denote this isomorphism by \(\Sigma_{i,\eta}\colon\sM_{\eta}\to\sN_{\eta}\). For every \(P\in S^{(1)}\), let \(A = \sO_{S,P}\), and then we have a canonical isomorphism
	\[ \begin{split}
		\sN_P &= \langle \sL_{1,P},\ldots,\sL_{i,P}\otimes\sL'_{i,P},\ldots\sL_{d+1,P}\rangle_{X_A/\Spec A}\\
		&= \langle\sL_{1,P},\ldots,\sL_{i,P},\ldots\sL_{d+1,P}\rangle_{X_A/\Spec A}\\
		&\phantom{\mathrel{=}\langle\sL_{1,P},\ldots,\sL_{i,P},}\otimes\langle\sL_{1,\eta},\ldots,\sL'_{i,\eta},\ldots\sL_{d+1,\eta}\rangle_{X_A/\Spec A}\\
		&=  \langle \sL_1,\ldots,\sL_i,\ldots\sL_{d+1}\rangle_{X/S,P}\otimes \langle \sL_1,\ldots,\sL'_i,\ldots\sL_{d+1}\rangle_{X/S,P}\\
		&= \sM_P.
	\end{split} \]
	We denote this isomorphism by \(\Sigma_{i,P}\colon\sM_P\to \sN_P\). Moreover, for every \(P\in S^{(1)}\), by \cite{Garcia2000} Theorem 4.2.6, we have a commutative diagram
	\[ \begin{tikzcd}
		\sM_P\arrow[r, hookrightarrow]\arrow[d,"\Sigma_{i,P}"] & \sM_\eta\arrow[d,"\Sigma_{i,\eta}"]\\
		\sN_P\arrow[r, hookrightarrow] & \sN_\eta.
	\end{tikzcd} \]
	By Proposition \ref{codim 1 miracle for line bundles}, \(\Sigma_{i,\eta}\) restricts to a canonical isomorphism \(\Sigma_{i}\colon\sM\to \sN\), i.e.
	\begin{multline}\label{eq:additive of EDP of f.s.a. lb}
		\Sigma_i\colon\langle \sL_1,\ldots,\sL_i,\ldots\sL_{d+1}\rangle_{X/S}\otimes\langle \sL_1,\ldots,\sL_i,\ldots\sL'_{d+1}\rangle_{X/S}\\
		\to  \langle \sL_1,\ldots,\sL_i\otimes\sL'_i,\ldots\sL_{d+1}\rangle_{X/S}.
	\end{multline}
\end{proof}

Multi-additive functors are the categorical generalization of multi-additive group homomorphism, as Picard groupoids are the categorical generalization of abelian groups. For the precise definition of a multi-monoidal functors, readers can refer to \cite{Ducrot2005} \S1.2.

For functorial properties, we discuss a little more on \(\Sigma_i\). After we prove that our Deligne pairing is a symmetric multi-additive functor, we can see that \(\Sigma_{i}\) is a natural isomorphism of two functors.

\begin{Prop}\label{commutative monoidal transformation of f.s.a lb}
	Let \(f\colon X\to S\) be a surjective projective morphism of noetherian schemes of pure relative dimension \(d\). Suppose that \(S\) is normal. Let \(1\leqslant i < j \leqslant d\) be two integers, \(\sL_1,\sL_2,\ldots,\sL_{d+1},\sL'_{i},\sL'_{j}\) be \(f\)-sufficiently ample invertible \(\sO_X\)-modules, and define
	\[ F(\sM,\sN) = \langle\sL_{1},\ldots,\sL_{i-1},\sM,\sL_{i+1},\ldots,\ldots,\sL_{j-1},\sN,\sL_{j+1},\ldots,\sL_{d+1}\rangle_{X/S} \]
	for simplicity. Then we have the following commutative diagram
	\[ \begin{tikzcd}[column sep=small, row sep=large]
		F(\sL_i,\sL_j)\otimes F(\sL_i,\sL'_j)\otimes F(\sL'_i,\sL_j)\otimes F(\sL'_i,\sL'_j)\arrow[dddd]\arrow[rd, "F(\sL_i\text{,}\Sigma_j)\otimes F(\sL'_i\text{,}\Sigma_j)"] \\
		&F(\sL_i,\sL_j\otimes\sL'_j)\otimes F(\sL'_i,\sL_j\otimes\sL'_j)\arrow[d, "F(\Sigma_i\text{,}\sL_j\otimes\sL'_j)"]\\
		&F(\sL_i\otimes\sL'_i,\sL_j\otimes\sL'_j)\\
		&F(\sL_i\otimes\sL'_i,\sL_j)\otimes F(\sL_i\otimes\sL'_i,\sL'_j)\arrow[u, "F(\sL_i\otimes\sL'_i\text{,}\Sigma_j)"']\\
		F(\sL_i,\sL_j)\otimes F(\sL'_i,\sL_j)\otimes F(\sL_i,\sL'_j)\otimes F(\sL'_i,\sL'_j),\arrow[ru, "F(\Sigma_i\text{,}\sL_j)\otimes F(\Sigma_i\text{,}\sL'_j)"']
	\end{tikzcd} \]
    where \(\Sigma_{i}\) and \(\Sigma_{j}\) are defined by \eqref{eq:additive of EDP of f.s.a. lb}, and the left vertical arrow is the unique canonical isomorphism in the symmetric monoidal category---\(\cPic(S)\).
\end{Prop}
\begin{proof}
	We can assume that \(S\) is integral, and let \(\eta\) be the generic point of \(S\). Note that \(\Sigma_i\) and \(\Sigma_j\) are induced by \(\Sigma_{i,\eta}\) and \(\Sigma_{j,\eta}\) on the generic fiber \(X_\eta\to \eta\) respectively, and that \(\Sigma_{i,\eta}\) and \(\Sigma_{j,\eta}\) satisfy the corresponding commutative diagram on the generic fiber since they are natural isomorphisms of the multi-additive functor---Deligne pairings for flat morphisms.
\end{proof}

Proposition \ref{commutative monoidal transformation of f.s.a lb} can be viewed as \(\Sigma_j\circ\Sigma_i = \Sigma_i\circ\Sigma_j\), which implies the order of operations of canceling parentheses (i.e. the order of \(\Sigma_{i}\)'s) for an expression \(\langle\ldots\rangle_{X/S}\) with \(\otimes\)-products is not important since there is only one isomorphism from the original expression to the final expression, regardless of the order.

\begin{Def}
	Let \(f\colon X\to S\) be a projective morphism of noetherian schemes. Define
	\begin{multline*}
		\Pic(X)_{f} = \{\sL\in\Pic(X)\mid \text{there exist } \sL_0,\sL_1\in\Pic(X)_{\fsa}\\
		 \text{ such that } \sL = \sL_0\otimes\sL_1^{-1}\}.
	\end{multline*}
\end{Def}

\begin{Prop}\label{EDP of good lb}
	Let \(f\colon X\to S\) be a surjective projective morphism of noetherian schemes of pure relative dimension \(d\). Suppose that \(S\) is normal. Then we can associate any \(d+1\) invertible \(\sO_X\)-modules \(\sL_1,\sL_2,\ldots,\sL_{d+1}\in\Pic(X)_f\) with an invertible \(\sO_S\)-module, which we define to be \(\langle\sL_1,\sL_2,\ldots,\sL_{d+1}\rangle_{X/S}\).
\end{Prop}
\begin{proof}
	For any integer \(1\leqslant i \leqslant d+1\), suppose that \(\sL_i = \sL_{i,0}\otimes\sL_{i,1}^{-1}\), where \(\sL_{i,0},\sL_{i,1}\in\Pic(X)_{\fsa}\). Let \(E\) be the set of all the maps \(h\colon \{1,2,\ldots,d+1\}\to \{0,1\}\). For \(h\in E\), let \(\epsilon_h = \prod_{i=1}^{d}(-1)^{h(i)}\). We define
	\[ \langle\sL_1,\sL_2,\ldots,\sL_{d+1}\rangle_{X/S}=\bigotimes_{h\in E}\langle\sL_{1,h(1)},\sL_{2,h(1)},\ldots,\sL_{d+1,h(d+1)}\rangle^{\epsilon_h}_{X/S}. \]
	By Proposition \ref{multi-additive of EDP of f.s.a lb} and \ref{commutative monoidal transformation of f.s.a lb}, \(\langle\sL_1,\sL_2,\ldots,\sL_{d+1}\rangle_{X/S}\) is independent of the choice of \(\sL_{i,0}\) and \(\sL_{i,1}\). More precisely, for two different choices, there is a unique isomorphism between the two results.
\end{proof}

\begin{Prop}\label{lb over affine is good}
	Let \(f\colon X\to S\) be a projective morphism of noetherian schemes. Suppose that \(S\) is an affine scheme. Then \(\Pic(X) = \Pic(X)_f\).
\end{Prop}
\begin{proof}
	Let \(\sO_X(1)\) be an \(f\)-very ample invertible \(\sO_X\)-module. For any \(\sL\in\Pic(X)\), by \cite{EGA2} Chapter 2 (4.6.8), we know that there exists a positive integer \(n_0\) such that for every integer \(n\geqslant n_0\), the canonical morphisms 
	\[ f^*f_*\sO_X(n)\to\sO_X(n)\quad \text{and}\quad f^*f_*\sL(n)\to\sL(n) \]
	are both surjective. Let \(\sL_0 = \sL(n_0)\) and \(\sL_1 = \sO_X(n_0)\). Then we see that \(\sL_0,\sL_{1}\in\Pic(X)_{\fsa}\), and thus \(\sL = \sL_0\otimes\sL_1^{-1}\in\Pic(X)_f\).
\end{proof}

\begin{Thm}\label{EDP of general l.b}
	Let \(f\colon X\to S\) be a surjective projective morphism of noetherian schemes of pure relative dimension \(d\). Suppose that \(S\) is normal. Then we can associate any \(d+1\) invertible \(\sO_X\)-modules \(\sL_1,\sL_2,\ldots,\sL_{d+1}\) with an invertible \(\sO_S\)-module, which we define to be \(\langle\sL_1,\sL_2,\ldots,\sL_{d+1}\rangle_{X/S}\).
\end{Thm}
\begin{proof}
	Let \(\{U_i\}_{i\in I}\) be an affine open cover of \(S\), and let \(f_i = f|_{U_i}\colon f^{-1}(U_i)\to U_i\). For each \(i\in I\) and integer \(1\leqslant j\leqslant d+1\), \(\sL_j\in\Pic(f^{-1}(U_i))_{f_i}\) by Proposition \ref{lb over affine is good}. Thus we can get an invertible \(\sO_{U_i}\)-module \(\cJ_i\) by Theorem \ref{EDP of good lb}. For \(i,i'\in I\), by Proposition \ref{indep of presentation}, we can find a canonical isomorphism \[\theta_{i,i'}\colon\cJ_i|_{U_i\cap U_{i'}}\to \cJ_{i'}|_{U_i\cap U_{i'}}\]
	such that \(\theta_{ii} = \id\), \(\theta_{ii'} = \theta_{i'i}^{-1}\), and for another \(i''\in I\), \(\theta_{i'i''}\circ\theta_{ii'} = \theta_{ii''}\) on \(U_i\cap U_{i'}\cap U_{i''}\). Therefore, we can glue \(\cJ_i\)'s on the overlaps \(U_i\cap U_{i'}\) to obtain an invertible \(\sO_S\)-module, which we define to be \(\langle\sL_1,\sL_2,\ldots,\sL_{d+1}\rangle_{X/S}\).
\end{proof}

\begin{Def}
	Let \(f\colon X\to S\) be a surjective projective morphism of noetherian schemes of pure relative dimension \(d\). Suppose that \(S\) is normal. Our \emph{Deligne pairing} is defined to be
	\begin{equation}\label{eq:def of equidimensional Deligne pairing}
		\langle\sL_1,\sL_2,\ldots,\sL_{d+1}\rangle_{X/S}\colon\Pic(X)^{d+1}\to\Pic(S).
	\end{equation}
\end{Def}

\subsection{Symmetry and multi-additivity}

The following three propositions are natural generalization of Proposition \ref{codim 1 of f.s.a equidimensional deligne pairing},  \ref{EDP = DP for f.s.a. lb}, \ref{multi-additive of EDP of f.s.a lb}, and \ref{commutative monoidal transformation of f.s.a lb}. Proposition \ref{codim 1 of equidimensional deligne pairing} will be crucial for the proof of symmetry, muiti-additivity, and functorial properties of our Deligne pairing.

\begin{Prop}\label{codim 1 of equidimensional deligne pairing}
	Let \(f\colon X\to S\) be a surjective projective morphism of noetherian schemes of pure relative dimension \(d\). Suppose that \(S\) is integral and normal, and that \(\sL_1,\sL_2,\ldots,\sL_{d+1}\) are \(d+1\) invertible \(\sO_X\)-modules. Then we have a commutative diagram
	\[ \begin{tikzcd}
		\langle\sL_{1},\sL_{2},\ldots,\sL_{n+1}\rangle_{X/S,\, P}\arrow[r, hookrightarrow]\arrow[d] & \langle\sL_{1},\sL_{2},\ldots,\sL_{n+1}\rangle_{X/S,\,\eta}\arrow[d]\\
		\langle\sL_{1,P},\sL_{2,P},\ldots,\sL_{n+1,P}\rangle_{X_A/\Spec A}\arrow[r, hookrightarrow] & \langle\sL_{1,\eta},\sL_{2,\eta},\ldots,\sL_{n+1,\eta}\rangle_{X_\eta/\eta},
	\end{tikzcd} \]
	where \(P\in S^{(1)}\), two vertical arrows are canonical isomorphisms, and two bottom objects are Deligne pairings defined by \cite{Garcia2000}.
\end{Prop}

\begin{Thm}\label{EDP = DP for all lb}
	Let \(f\colon X\to S\) be a surjective projective morphism of noetherian schemes of pure relative dimension \(d\). Suppose that \(S\) is normal, \(f\) is of finite Tor-dimension, and that \(\sL_1,\sL_2,\ldots,\sL_{d+1}\) are \(d+1\) invertible \(\sO_X\)-modules. Then our \(\langle \sL_1,\sL_2,\ldots,\sL_{d+1}\rangle_{X/S}\) is identical to the corresponding Deligne pairing defined in \cite{Garcia2000}.
\end{Thm}

\begin{Prop}
	Our Deligne pairing \eqref{eq:def of equidimensional Deligne pairing} is multi-additive, i.e. we have a canonical isomorphism 
	\begin{multline}\label{eq:additive of EDP of lb}
		\Sigma_i\colon\langle \sL_1,\ldots,\sL_i,\ldots\sL_{d+1}\rangle_{X/S}\otimes\langle \sL_1,\ldots,\sL'_i,\ldots\sL_{d+1}\rangle_{X/S}\\
		\to  \langle \sL_1,\ldots,\sL_i\otimes\sL'_i,\ldots\sL_{d+1}\rangle_{X/S}
	\end{multline}
    for each integer \(1\leqslant i \leqslant d+1\). Moreover, let
    \[ F(\sM,\sN) = \langle\sL_{1},\ldots,\sL_{i-1},\sM,\sL_{i+1},\ldots,\ldots,\sL_{j-1},\sN,\sL_{j+1},\ldots,\sL_{d+1}\rangle_{X/S}, \]
    and then we have the following commutative diagram
    \begin{equation}
    	\begin{tikzcd}[column sep=small, row sep=large]
    	F(\sL_i,\sL_j)\otimes F(\sL_i,\sL'_j)\otimes F(\sL'_i,\sL_j)\otimes F(\sL'_i,\sL'_j)\arrow[dddd]\arrow[rd, "F(\sL_i\text{,}\Sigma_j)\otimes F(\sL'_i\text{,}\Sigma_j)"] \\
    	&F(\sL_i,\sL_j\otimes\sL'_j)\otimes F(\sL'_i,\sL_j\otimes\sL'_j)\arrow[d, "F(\Sigma_i\text{,}\sL_j\otimes\sL'_j)"]\\
    	&F(\sL_i\otimes\sL'_i,\sL_j\otimes\sL'_j)\\
    	&F(\sL_i\otimes\sL'_i,\sL_j)\otimes F(\sL_i\otimes\sL'_i,\sL'_j)\arrow[u, "F(\sL_i\otimes\sL'_i\text{,}\Sigma_j)"']\\
    	F(\sL_i,\sL_j)\otimes F(\sL'_i,\sL_j)\otimes F(\sL_i,\sL'_j)\otimes F(\sL'_i,\sL'_j).\arrow[ru, "F(\Sigma_i\text{,}\sL_j)\otimes F(\Sigma_i\text{,}\sL'_j)"']
    \end{tikzcd}
\end{equation}
\end{Prop}
\begin{proof}
	Directly by Proposition \ref{multi-additive of EDP of f.s.a lb}, Proposition \ref{commutative monoidal transformation of f.s.a lb}, and the definition of our Deligne pairing.
\end{proof}

\begin{Prop}\label{symmetric of EDP}
	Our Deligne pairing \eqref{eq:def of equidimensional Deligne pairing} is symmetric, i.e. we have a canonical isomorphism 
	\begin{equation}\label{eq:symmetric of EDP of lb}
		\gamma_\phi\colon\langle \sL_1,\sL_2,\ldots,\sL_{d+1}\rangle_{X/S}\to\langle \sL_{\phi(1)},\sL_{\phi(2)},\ldots,\sL_{{\phi(d+1)}}\rangle_{X/S}
	\end{equation}
	for each \(\phi\in S_{d+1}\), where \(S_{d+1}\) is the permutation group of \(\{1,2,\ldots,d+1\}\). Moreover, 
	\begin{equation}\label{eq:composition of symmetric transformation}
		\gamma_{\psi\phi} = \gamma_{\phi}\circ\gamma_{\psi}
    \end{equation}
	for \(\phi,\psi\in S_{d+1}\), \(\gamma_{\id} = \id\), and we have the following commutative diagram
	\begin{equation*}
		\begin{tikzcd}
			\langle \sL_1,\ldots,\sL_i,\ldots\sL_{d+1}\rangle_{X/S}\otimes\langle \sL_1,\ldots,\sL'_i,\ldots\sL_{d+1}\rangle_{X/S} \xrightarrow{\ \ \Sigma_i \ \ } \arrow[d,"\gamma_{\phi}\otimes\gamma_{\phi}"]&
			\\
			\langle \sL_{\phi(1)},\ldots,\sL_i,\ldots,\sL_{\phi(d+1)}\rangle_{X/S}\otimes\langle \sL_{\phi(d+1)},\ldots,\sL'_i,\ldots\sL_{{\phi(1)}}\rangle_{X/S} \xrightarrow{\Sigma_{\phi^{-1}(i)}}&
		\end{tikzcd}
	\end{equation*}
	\begin{equation}\label{eq:commutative of Sigme and iota}
	\phantom{MMMMMMMMM}\begin{tikzcd}
			&&\xrightarrow{\ \ \Sigma_i \ \ } \langle \sL_1,\ldots,\sL_i\otimes\sL'_i,\ldots\sL_{d+1}\rangle_{X/S}\arrow[d,"\gamma_{\phi}"]
			\\
			&&\xrightarrow{\Sigma_{\phi^{-1}(i)}}\langle \sL_{\phi(1)},\ldots,\sL_i\otimes\sL'_i,\ldots\sL_{{\phi(d+1)}}\rangle_{X/S},
		\end{tikzcd}
	\end{equation}
    where \(\Sigma_{i}\) and \(\Sigma_{\phi^{-1}(i)}\) are defined by \eqref{eq:additive of EDP of lb}.
\end{Prop}
\begin{proof}
	We can assume that \(S\) is integral, and let \(\eta\) be the generic point of \(S\). Let
	\[ \sM = \langle \sL_1,\sL_2,\ldots\sL_{d+1}\rangle_{X/S} \]
	and
	\[ \sN = \langle\sL_{\phi(1)},\sL_{\phi(2)},\ldots\sL_{{\phi(d+1)}}\rangle_{X/S}. \]
	Then we have a canonical isomorphism
	\[ \begin{split}
		\sM_\eta &= \langle \sL_{1,\eta},\sL_{2,\eta},\ldots,\sL_{d+1,\eta}\rangle_{X_\eta/\eta}\\
		&=  \langle \sL_{\phi(1),\eta},\sL_{\phi(2),\eta},,\ldots\sL_{\phi(d+1),\eta}\rangle_{X_\eta/\eta}\\
		&= \sN_\eta.
	\end{split} \]
	We denote this isomorphism by \(\gamma_{\phi,\eta}\colon\sM_{\eta}\to\sN_{\eta}\). For every \(P\in S^{(1)}\), let \(A = \sO_{S,P}\), and then we have a canonical isomorphism
	\[ \begin{split}
		\sM_P &= \langle \sL_{1,P},\sL_{2,P},\ldots,\sL_{d+1,P}\rangle_{X_P/\Spec A}\\
		&=  \langle \sL_{\phi(1),P},\sL_{\phi(2),P},,\ldots\sL_{\phi(d+1),P}\rangle_{X_P/\Spec  A}\\
		&= \sN_P.
	\end{split} \]
	We denote this isomorphism by \(\gamma_{\phi,P}\colon\sM_P\to \sN_P\). Moreover, for every \(P\in S^{(1)}\), by \cite{Garcia2000} Theorem 4.2.7, we have a commutative diagram
	\[ \begin{tikzcd}
		\sM_P\arrow[r, hookrightarrow]\arrow[d,"\gamma_{\phi,P}"] & \sM_\eta\arrow[d,"\gamma_{\phi,\eta}"]\\
		\sN_P\arrow[r, hookrightarrow] & \sN_\eta.
	\end{tikzcd} \]
	By Proposition \ref{codim 1 miracle for line bundles}, \(\gamma_{\phi,\eta}\) restricts to a canonical isomorphism \(\gamma_{\phi}\colon\sM \to \sN\), i.e.
	\[ \gamma_{\phi}\colon\langle \sL_1,\sL_2,\ldots,\sL_{d+1}\rangle_{X/S}\to\langle \sL_{\phi(1)},\sL_{\phi(2)},\ldots,\sL_{{\phi(d+1)}}\rangle_{X/S}. \]
	
	Note that \(\Sigma_i\) and \(\gamma_\phi\) are induced by \(\Sigma_{i,\eta}\) and \(\gamma_{\phi,\eta}\) on the generic fiber \(X_\eta\to \eta\) respectively, and that \(\Sigma_{i,\eta}\) and \(\gamma_{\phi,\eta}\) satisfy the corresponding properties and commutative diagram on the generic fiber since they are natural isomorphisms of Deligne pairings for flat morphisms.
\end{proof}

Diagram \eqref{eq:commutative of Sigme and iota} can be viewed as \(\Sigma\circ\iota = \iota\circ\Sigma\). Proposition \ref{commutative monoidal transformation of f.s.a lb} and \ref{symmetric of EDP} implies that the order of operations of canceling parentheses and interchanging factors  (i.e. the order of \(\Sigma_{i}\)'s and \(\gamma_{\phi}\)'s) for an expression \(\langle\ldots\rangle_{X/S}\) with \(\otimes\)-products is not important since there is only one isomorphism from the original expression to the final expression, regardless of the order.

We summarize our results in this section as the following theorem.

\begin{Thm}\label{EDP}
	Let \(f\colon X\to S\) be a surjective projective morphism of noetherian schemes of pure relative dimension \(d\). Suppose that \(S\) is normal. Then \eqref{eq:def of equidimensional Deligne pairing} defines a symmetric multi-additive homomorphism
	\begin{equation*}
		\langle\sL_1,\sL_2,\ldots,\sL_{d+1}\rangle_{X/S}\colon\Pic(X)^{d+1}\to\Pic(S).
	\end{equation*}
\end{Thm}

\section{Functorial property}

First we introduce a useful lemma.

\begin{Lem}\label{prism lemma}
	Suppose that we have the follow diagram in a category.
	\[ \begin{tikzcd}
		E\arrow[dd]\arrow[rd] & & F\arrow[dd]\arrow[ld]\\
		& G\\
		B\arrow[rr]\arrow[rd] && C\arrow[ld]\\
		& A\arrow[from = uu, crossing over]
	\end{tikzcd} \]
	Suppose that left and right faces of the prism are pull-back squares. Then there is a canonical arrow \(E\to F\) making the back face of the prism to be a pull-back square.
\end{Lem}
\begin{proof}
	Easy diagram-chasing.
\end{proof}

Let \(f\colon X\to S\) be a surjective projective morphism of noetherian schemes of pure relative dimension \(d\). Suppose that \(S\) is normal. Let \(\cU\) be the open subset of \(S\) over which \(f\) is flat, and \(\cV\) be the subset of \(S\) over which \(f\) is of finite Tor-dimension, i.e.
\[ \cU = \{s\in S\mid X\times\Spec\sO_{S,s}\to \Spec\sO_{S,s}\text{ is flat}\} \]
and
\[ \cV = \{s\in S\mid X\times\Spec\sO_{S,s}\to \Spec\sO_{S,s}\text{ is of finite Tor-dimension}\}. \]
Note that \(\cU\cup S^{(1)}\subseteq \cV\). Let \(g\colon S'\to S\) be a morphism of noetherian normal schemes.

\begin{Def}\label{def: good base change}
	If \(S\) and \(S'\) are both integral, then we say that \(g\colon S'\to S\) is a \emph{good base change} if \(g(\eta')\in\cU\) and \(g(S'^{(1)})\subseteq \cV\), where \(\eta'\) is the generic point of \(S'\). For general \(S\) and \(S'\), we say that \(g\colon S'\to S\) is a \emph{good base change} if for each pair of irreducible components \(C\) of \(S\) and \(C'\) of \(S'\) such that \(g(C')\subseteq C\), \(g|_{C'}\colon C'\to C\) is a good base change.
\end{Def}

\begin{Thm}\label{base change of normal Deligne pairing}
	Let \(g\colon S'\to S\) be a good base change. Let \(X\) be a scheme and \(f\colon X\to S\) be a surjective projective morphism of pure relative dimension \(d\). Let \(X' = X\times_SS'\to S'\) and \(g'\colon X'\to X\). Then we have a natural isomorphism
	\[ \langle g'^*\sL_1,g'^*\sL_2,\ldots,g'^*\sL_{d+1} \rangle_{X'/S'}\to g^*\langle \sL_1,\sL_2,\ldots,\sL_{d+1} \rangle_{X/S}. \]
\end{Thm}
\begin{proof}
	We can assume that both \(S\) and \(S'\) are integral. Let \(\eta'\) be the generic point of \(S\), \(Q  = g(\eta')\in\cU\), \(B = \Spec \sO_{S,Q}\), and let \(X_B =  S\times_S\Spec B\). Then we have the following commutative diagram
	\[ \begin{tikzcd}
		X'_{\eta'}\arrow[dd]\arrow[rd] & & & X_B\arrow[dd]\arrow[ld]\\
		& X'\arrow[r] & X\\
		\eta'\arrow[rrr]\arrow[rd] & & & \Spec B\arrow[ld]\\
		& S'\arrow[r]\arrow[from = uu, crossing over] & S, \arrow[from = uu, crossing over]
	\end{tikzcd} \]
	where the front, left, and right faces of the prism are all pull-back squares. Then by Lemma \ref{prism lemma}, there is a canonical arrow \(X'_{\eta'} \to X_B\) such that the following diagram
	\[ \begin{tikzcd}
		X'_{\eta'}\arrow[r]\arrow[d] & X_B\arrow[d]\\
		\eta'\arrow[r] & \Spec B
	\end{tikzcd} \]
	is a pull-back square. Let \(h\colon\Spec B\to S\). From the following commutative diagram
	\[ \begin{tikzcd}
		X'_{\eta'}\arrow[r]\arrow[d] & X_B\arrow[d]\arrow[r] & X\arrow[d]\\
		\eta'\arrow[r] & \Spec B\arrow[r] & S
	\end{tikzcd} \]
	we obtain a canonical isomorphism
	\begin{align}
		(g^*\langle \sL_1,\sL_2,\ldots,\sL_{d+1} \rangle_{X/S})_{\eta'} &= (h^*(\langle \sL_1,\sL_2,\ldots,\sL_{d+1} \rangle_{X/S}))_{\eta'}\notag\\
		&= \langle \sL_{1,B},\sL_{2,B},\ldots,\sL_{d+1,B} \rangle_{X_B/\Spec B,\eta'}\notag\\
		&= \langle g'^*\sL_{1})_{\eta'},(g'^*\sL_{2})_{\eta'},\ldots,(g'^*\sL_{d+1})_{\eta'} \rangle_{X'_{\eta'}/\eta'}\notag\\
		&= \langle g'^*\sL_{1},g'^*\sL_{2},\ldots,g'^*\sL_{d+1} \rangle_{X'/S', \eta'}.\label{eq:functoriality 1-1}
	\end{align}
	Let \(P'\in S'^{(1)}\) and \(P = g(P')\in \cV\). Let \(A = \sO_{S,P}\), \(A' = \sO_{S',P'}\), \(X_A = X\times_S\Spec A\), and let \(X'_{A'} = X'\times_{S'}\Spec A'\). Then we have the following commutative diagram
	\[ \begin{tikzcd}
		X'_{A'}\arrow[dd]\arrow[rd] & & & X_A\arrow[dd]\arrow[ld]\\
		& X'\arrow[r] & X\\
		\Spec A'\arrow[rrr]\arrow[rd] & & & \Spec A\arrow[ld]\\
		& S'\arrow[r]\arrow[from = uu, crossing over] & S \arrow[from = uu, crossing over]
	\end{tikzcd} \]
	where the front, left, and right faces of the prism are all pull-back squares. Then by Lemma \ref{prism lemma}, there is a canonical arrow \(X'_{A'} \to X_A\) such that the following diagram
	\[ \begin{tikzcd}
		X'_{A'}\arrow[r]\arrow[d] & X_A\arrow[d]\\
		\Spec A'\arrow[r] & \Spec A
	\end{tikzcd} \]
	is a pull-back square. Then we have canonical isomorphisms
	\begin{align}
		(g^*\langle \sL_1,\sL_2,\ldots,\sL_{d+1} \rangle_{X/S})_{P'} &= g^*(\langle \sL_1,\sL_2,\ldots,\sL_{d+1} \rangle_{X/S,P})\notag\\
		&= g^*(\langle \sL_{1,P},\sL_{2,P},\ldots,\sL_{d+1,P} \rangle_{X_A/\Spec A})\notag\\
		&= \langle g'^*\sL_{1,P},g'^*\sL_{2,P},\ldots,g'^*\sL_{d+1,P}\rangle_{X'_{A'}/\Spec A'}\notag\\
		&= \langle (g'^*\sL_{1})_{P'},(g'^*\sL_{2})_{P'},\ldots,(g'^*\sL_{d+1})_{P'} \rangle_{X'_{A'}/\Spec A'}\notag\\
		&= \langle g'^*\sL_{1},g'^*\sL_{2},\ldots,g'^*\sL_{d+1} \rangle_{X'/S', P'}.\label{eq:functoriality 1-2}
	\end{align}
	Since \(P\) lies in the Zariski closure of \(\{Q\}\), we have the following commutative diagram
	\begin{equation}\label{diagram: good base change}
		\begin{tikzcd}
			X'_{\eta'}\arrow[rrrrr]\arrow[rd]\arrow[ddd] &&&&& X_B\arrow[ddd]\arrow[ld]\\
			&X'_{A'}\arrow[rrr]\arrow[rd] & & & X_A\arrow[ld]\\
			&& X'\arrow[r] & X\\
			\eta'\arrow[rrrrr]\arrow[rd] &&&&& \Spec B\arrow[ld]\\
			&\Spec A'\arrow[rrr]\arrow[from=uuu, crossing over]\arrow[rd] & & & \Spec A\arrow[ld]\arrow[from=uuu, crossing over]\\
			&& S'\arrow[r]\arrow[from = uuu, crossing over] & S. \arrow[from = uuu, crossing over]
		\end{tikzcd}
	\end{equation} 
	Thus the isomorphisms \eqref{eq:functoriality 1-1} and \eqref{eq:functoriality 1-2} satisfy the commutative diagram in Lemma \ref{codim 1 miracle for line bundles}. Then by Proposition \ref{codim 1 miracle for line bundles}, we get an isomorphism 
	\[ \langle g'^*\sL_1,g'^*\sL_2,\ldots,g'^*\sL_{d+1} \rangle_{X'/S'}\to g^*\langle \sL_1,\sL_2,\ldots,\sL_{d+1} \rangle_{X/S}.\qedhere \]
\end{proof}

Diagram \eqref{diagram: good base change} is very useful, and will be used to deal with nearly all the issues concerning good base changes.

\begin{Thm}
	All the canonical isomorphisms \(\Sigma\)'s \eqref{commutative monoidal transformation of f.s.a lb} and \(\gamma\)'s \eqref{symmetric of EDP} are compatible with any good base change \(S'\to S\).
\end{Thm}
\begin{proof}
	The morphisms are already there. For the compatibility, it suffices to check the diagrams over the generic point of each irreducible component of \(S'\), where the morphism becomes a flat morphism. Then the result follows from \cite{Garcia2000} Theorem 4.2.6 and 4.2.7.
\end{proof}

\begin{Thm}\label{pull-back formula of normal Deligne pairing}
	Let \(f\colon X\to S\) be a surjective projective morphism of noetherian scheme of pure relative dimension \(d\). Suppose that \(S\) is integral and normal. Let \(\sL_1,\ldots,\sL_d\) be invertible \(\sO_X\)-modules, and let \(\delta\) be their intersection number on the generic fiber of \(X\to S\). Let \(\sM\) an invertible \(\sO_S\)-module. Then we have a natural isomorphism
	\[ \langle \sL_1,\ldots,f^*\sM,\ldots,\sL_d \rangle_{X/S}\to \sM^{\otimes\delta},\]
	and the isomorphism is compatible with any good base change \(g\colon S'\to S\), where \(S'\) is also a noetherian integral normal scheme.
\end{Thm}
\begin{proof}
	Let \(\eta\) be the generic point of \(S\). Then we have a canonical isomorphism
	\[ \varphi_\eta\colon \langle \sL_1,\ldots,f^*\sM,\ldots,\sL_d \rangle_{X/S,\eta}\to \sM_\eta^{\otimes\delta}. \]
	Let \(P\in S^{(1)}\). Then we have a canonical isomorphism
	\[ \varphi_P\colon \langle \sL_1,\ldots,f^*\sM,\ldots,\sL_d \rangle_{X/S,P}\to \sM_P^{\otimes\delta}. \]
	Moreover, for every \(P\in S^{(1)}\), by \cite{Garcia2000} Theorem 5.2.1.a, the isomorphisms \(\varphi_{\eta}\) and \(\varphi_P\) satisfy the following commutative diagram
	\[ \begin{tikzcd}
		\langle \sL_1,\ldots,f^*\sM,\ldots,\sL_d \rangle_{X/S,P}\arrow[r, hookrightarrow]\arrow[d,"\varphi_P"] & \langle \sL_1,\ldots,f^*\sM,\ldots,\sL_d \rangle_{X/S,\eta}\arrow[d,"\varphi_\eta"]\\
		\sM_P^{\otimes\delta}\arrow[r, hookrightarrow] & \sM_\eta^{\otimes\delta}.
	\end{tikzcd} \]
	By Proposition \ref{codim 1 miracle for line bundles}, \(\phi_\eta\) restricts to a canonical isomorphism
	\[ \varphi\colon\langle \sL_1,\ldots,f^*\sM,\ldots,\sL_d \rangle_{X/S}\to \sM^{(\sL_{1,\eta},\ldots,\sL_{d,\eta})}.\]
	
	Let \(\eta'\) be the generic point of \(S'\), \(Q  = g(\eta')\in\cU\), and let \(k\) be the residue field of \(Q\). Let \(\delta_k\) and \(\delta'\) be the intersection number of \(\sL_1,\ldots,\sL_d\) over \(k\) and \(\eta'\) respectively. By properties of flat families and flat base changes, we deduce that \(\delta = \delta_k = \delta'\).
	
	For the compatibility with good base change, it suffices to check the diagram over the generic point \(\eta'\), where the morphism becomes a flat morphism. Then the result follows from \cite{Garcia2000} Proposition 5.2.1.a.
\end{proof}

\begin{Thm}\label{projection formula of normal Deligne pairing}
	Let \(f\colon X\to S\) be a surjective projective morphism of noetherian schemes of pure relative dimension \(d\), and \(g\colon Y\to X\) be a surjective projective morphism of noetherian schemes of pure relative dimension \(d'\). Suppose that \(X\) and \(S\) are normal. Then we have a natural isomorphism
	\begin{multline*}
		\langle g^*\sL_1,\ldots,g^*\sL_d,\sM_{d+1},\ldots,\sM_{d+d'+1} \rangle_{Y/S}\\
		\to \langle\sL_1,\ldots,\sL_d,\langle\sM_{d+1},\ldots,\sM_{d+d'+1}\rangle_{Y/X} \rangle_{X/S},
	\end{multline*}
    and the isomorphism is compatible with any good base change \(g\colon S'\to S\).
\end{Thm}
\begin{proof}
	We can assume that \(S\) is integral, and let \(\eta\) be the generic point of \(S\). Then we have a canonical isomorphism
	\begin{align}
		\langle\sL_1,\ldots,\sL_d,&\langle\sM_{d+1},\ldots,\sM_{d+d'+1}\rangle_{Y/X} \rangle_{X/S,\eta}\notag\\
		&= \langle\sL_{1,\eta},\ldots,\sL_{d,\eta},\langle\sM_{d+1},\ldots,\sM_{d+d'+1}\rangle_{Y/X,\eta} \rangle_{X_\eta/\eta}\notag\\
		&= \langle\sL_{1,\eta},\ldots,\sL_{d,\eta},\langle\sM_{d+1},\ldots,\sM_{d+d'+1}\rangle_{Y/X,\eta} \rangle_{X_\eta/\eta}\notag\\
		&= \langle\sL_{1,\eta},\ldots,\sL_{d,\eta},\langle\sM_{d+1,\eta},\ldots,\sM_{d+d'+1,\eta}\rangle_{Y_\eta/X_\eta} \rangle_{X_\eta/\eta}\notag\\
		&= \langle g^*\sL_{1,\eta},\ldots,g^*\sL_{d,\eta},\sM_{d+1,\eta},\ldots,\sM_{d+d'+1,\eta}\rangle_{Y_\eta/\eta}\notag\\
		&= \langle (g^*\sL_1)_\eta,\ldots,(g^*\sL_d)_\eta,\sM_{d+1,\eta},\ldots,\sM_{d+d'+1,\eta}\rangle_{Y_\eta/\eta}\notag\\
		&= \langle g^*\sL_1,\ldots, g^*\sL_d,\sM_{d+1},\ldots,\sM_{d+d'+1}\rangle_{Y/S, \eta},\label{eq:functoriality3-1} 
	\end{align}
	where the third isomorphism follows from Theorem \ref{base change of normal Deligne pairing}.
	
	For each \(P\in S^{(1)}\), let \(A = \Spec\sO_{S,P}\), and we have a canonical isomorphism
	\begin{align}
		\langle\sL_1,\ldots,\sL_d,&\langle\sM_{d+1},\ldots,\sM_{d+d'+1}\rangle_{Y/X} \rangle_{X/S,P}\notag\\
		&= \langle\sL_{1,P},\ldots,\sL_{d,P},\langle\sM_{d+1},\ldots,\sM_{d+d'+1}\rangle_{Y/X,P} \rangle_{X_A/\Spec A}\notag\\
		&= \langle\sL_{1,P},\ldots,\sL_{d,P},\langle\sM_{d+1},\ldots,\sM_{d+d'+1}\rangle_{Y/X,P} \rangle_{X_A/\Spec A}\notag\\
		&= \langle\sL_{1,P},\ldots,\sL_{d,P},\langle\sM_{d+1,P},\ldots,\sM_{d+d'+1,P}\rangle_{Y_A/X_A} \rangle_{X_A/\Spec A}\notag\\
		&= \langle g^*\sL_{1,P},\ldots,g^*\sL_{d,P},\sM_{d+1,P},\ldots,\sM_{d+d'+1,P}\rangle_{Y_A/\Spec A}\notag\\
		&= \langle (g^*\sL_1)_P,\ldots,(g^*\sL_d)_P,\sM_{d+1,P},\ldots,\sM_{d+d'+1,P}\rangle_{Y_A/\Spec A}\notag\\
		&= \langle g^*\sL_1,\ldots, g^*\sL_d,\sM_{d+1},\ldots,\sM_{d+d'+1}\rangle_{Y/S, P},\label{eq:functoriality3-2} 
	\end{align}
	where the third isomorphism follows from Theorem \ref{base change of normal Deligne pairing}. The isomorphisms \eqref{eq:functoriality3-1} and \eqref{eq:functoriality3-2} satisfy the commutative diagram in Proposition \ref{codim 1 miracle for line bundles}. Then by Proposition \ref{codim 1 miracle for line bundles}, we obtain a canonical isomorphism
	\begin{multline*}
		\langle g^*\sL_1,\ldots,g^*\sL_n,\sM_{n+1},\ldots,\sM_{n+m+1} \rangle_{Y/S}\\
		\to \langle\sL_1,\ldots,\sL_n,\langle\sM_{n+1},\ldots,\sM_{n+m+1}\rangle_{Y/X} \rangle_{X/S}.
	\end{multline*}
    For the compatibility with good base change \(S'\to S\), we can assume that \(S'\) is integral, and then it suffices to check the diagram over \(\eta'\to \eta\), where \(\eta'\) is the generic point of \(S'\). Let \(X' = X\times_S S'\). Since both \(X'_{\eta'}\to\eta'\) and \(X_{\eta}\to\eta\) are flat morphisms, the result follows from \cite{Garcia2000} Proposition 5.2.3.b.
\end{proof}

\begin{Thm}\label{projection formula of taking closed subscheme}
	Let \(f\colon X\to S\) be a surjective projective morphism of noetherian schemes of pure relative dimension \(d\geqslant 1\), and \(D\) be an effective Cartier divisor on \(X\) such that \(D\to S\) is of pure relative dimension \(d-1\). Suppose that \(S\) is normal. Then we have a natural isomorphism
	\begin{equation*}
		\langle \sL_1,\ldots,\sO(D),\ldots,\sL_d \rangle_{X/S}
		\to \langle\sL_1|_D,\ldots,\sL_d|_D \rangle_{D/S},
	\end{equation*}
    where we just delete \(\sO(D)\) in the second expression. Moreover, this isomorphism is compatible with any good base change \(g\colon S'\to S\).
\end{Thm}
\begin{proof}
	We can assume that \(S\) is integral, and let \(\eta\) be the generic point of \(S\). From the following commutative diagram
	\[\begin{tikzcd}
		D_\eta\arrow[r]\arrow[d] & D\arrow[d]\\
		X_\eta\arrow[r] & X
	\end{tikzcd}\]
	we deduce that for each \(\sL\in\Pic(X)\), \((\sL|_D)_\eta = \sL_\eta|_{D_\eta}\), where \(\sL_\eta\) is the abbreviation for \(\sL_{X_\eta}\). Since \(D\) is an effective Cartier divisor, the canonical section \(s_D\) of \(\sO(D)\) is a regular global section. Then by \cite{Garcia2000} Proposition 4.3.1, we have a canonical isomorphism
	\begin{align}\label{eq: iso of taking closed subscheme at generic pt}
		\langle \sL_1,\ldots,\sO(D),\ldots,\sL_d \rangle_{X/S, \eta} &= \langle \sL_1,\ldots,\sO(D),\dots,\sL_d \rangle_{X/S, \eta}\notag\\
		&= \langle \sL_{1,\eta},\ldots,\sO(D)_\eta,\ldots,\sL_{d,\eta} \rangle_{X_\eta/\eta}\notag\\
		&= \langle \sL_{1,\eta},\ldots,\sO(D_\eta),\ldots,\sL_{d,\eta} \rangle_{X_\eta/\eta}\notag\\
		&= \langle \sL_{1,\eta}|_{D_\eta},\ldots,\sL_{d,\eta}|_{D_\eta} \rangle_{D_\eta/\eta}\notag\\
		&= \langle (\sL_{1}|_D)_\eta,\ldots,(\sL_{d}|_D)_\eta \rangle_{D_\eta/\eta}\notag\\
		&= \langle \sL_{1}|_D,\ldots,\sL_{d}|_D \rangle_{D/S,\eta}.
	\end{align}
	
	Let \(P\in S^{(1)}\), and \(A = \sO_{S,P}\). From the following commutative diagram
	\[\begin{tikzcd}
		D_A\arrow[r]\arrow[d] & D\arrow[d]\\
		X_A\arrow[r]\arrow[d] & X\arrow[d]\\
		\Spec A\arrow[r] & S
	\end{tikzcd}\]
	we deduce that for any \(\sL\in\Pic(X)\), \((\sL|_D)_A = \sL_A|_{D_A}\), where \(\sL_A\) is the abbreviation for \(\sL_{X_A}\). Then we have a canonical isomorphism
	\begin{align}\label{eq: iso of taking closed subscheme at codim 1 pt}
		\langle \sL_1,\ldots,\sO(D),\ldots,\sL_d \rangle_{X/S, P} &= \langle \sL_1,\ldots,\sO(D),\dots,\sL_d \rangle_{X/S, P}\notag\\
		&= \langle \sL_{1,A},\ldots,\sO(D)_A,\ldots,\sL_{d,A} \rangle_{X_A/A}\notag\\
		&= \langle \sL_{1,A},\ldots,\sO(D_A),\ldots,\sL_{d,A} \rangle_{X_A/A}\notag\\
		&= \langle \sL_{1,A}|_{D_A},\ldots,\sL_{d,A}|_{D_A} \rangle_{D_A/A}\notag\\
		&= \langle (\sL_{1}|_D)_A,\ldots,(\sL_{d}|_D)_A \rangle_{D_A/A}\notag\\
		&= \langle \sL_{1}|_D,\ldots,\sL_{d}|_D \rangle_{D/S,P}.
	\end{align}
	Then isomorphisms \eqref{eq: iso of taking closed subscheme at generic pt} and \eqref{eq: iso of taking closed subscheme at codim 1 pt} satisfy the commutative diagram in Proposition \ref{diagram: good base change}. Then by Proposition \ref{diagram: good base change}, we deduce the required isomorphism.
	
	For the compatibility with good base change, we can assume that \(S'\) is integral, and it suffices to check the diagram over \(\eta'\to \eta\), where \(\eta'\) is the generic point of \(S'\). Let \(X' = X\times_S S'\). Since both \(X'_{\eta'}\to\eta'\) and \(X_{\eta}\to\eta\) are flat morphisms, the result follows from \cite{Garcia2000} Proposition 4.3.1.
\end{proof}

\begin{Thm}\label{birational invariance}
	Let \(f\colon X\to S\) be a surjective projective morphism of noetherian schemes of pure relative dimension \(d\), and \(g\colon Y\to X\) be a projective morphism. Suppose that \(S\) is normal, \(Y\to S\) is a surjective projective morphism of pure relative dimension \(d\), and that there exists an open subset \(U\subseteq X\) which is dense in each fiber of \(f\) and satisfies that \(g\colon g^{-1}(U)\to U\) is an isomorphism. Then we have a natural isomorphism
	\begin{equation*}
		\langle \sL_1,\sL_2,\ldots,\sL_{d+1} \rangle_{X/S}\to \langle g^*\sL_1,g^*\sL_2,\ldots,g^*\sL_{d+1} \rangle_{Y/S}.
	\end{equation*}
\end{Thm}
\begin{proof}
	Let \(\eta\) be the generic point of \(S\). Note that \(Y_\eta\to X_\eta\) is a birational projective morphism. Then by \cite{Garcia2000} Theorem 5.3.1, we have a canonical isomorphism
	\begin{align}\label{eq:birational inv 1}
		\langle \sL_1,\sL_2,\ldots,\sL_{d+1} \rangle_{X/S,\eta} &= \langle \sL_{1,\eta},\sL_{2,\eta},\ldots,\sL_{d+1,\eta} \rangle_{X_\eta/\eta}\notag\\
		&= \langle g^*\sL_{1,\eta},g^*\sL_{2,\eta},\ldots,g^*\sL_{d+1,\eta} \rangle_{Y_\eta/\eta}\notag\\
		&=  \langle g^*\sL_1,g^*\sL_2,\ldots,g^*\sL_{d+1} \rangle_{Y/S, \eta}.
	\end{align}

    For each \(P\in S^{(1)}\), let \(S = \Spec\sO_{S,P}\), and by \cite{Garcia2000} Theorem 5.3.1 we have a canonical isomorphism  
    \begin{align}\label{eq:birational inv 2}
    	\langle \sL_1,\sL_2,\ldots,\sL_{d+1} \rangle_{X/S,P} &= \langle \sL_{1,P},\sL_{2,P},\ldots,\sL_{d+1,P} \rangle_{X_A/\Spec A}\notag\\
    	&= \langle g^*\sL_{1,P},g^*\sL_{2,P},\ldots,g^*\sL_{d+1,P} \rangle_{Y_A/\Spec A}\notag\\
    	&=  \langle g^*\sL_1,g^*\sL_2,\ldots,g^*\sL_{d+1} \rangle_{Y/S, P}.
    \end{align}
     The isomorphisms \eqref{eq:birational inv 1} and \eqref{eq:birational inv 2} satisfy the commutative diagram in Proposition \ref{codim 1 miracle for line bundles}. Then by Proposition \ref{codim 1 miracle for line bundles}, we obtain a canonical isomorphism
     \[\langle \sL_1,\sL_2,\ldots,\sL_{d+1} \rangle_{X/S}\to \langle g^*\sL_1,g^*\sL_2,\ldots,g^*\sL_{d+1} \rangle_{Y/S}.\qedhere\]
\end{proof}

\section{Divisor sequence and section sequence}

\begin{Def}
		Let \(f\colon X\to S\) be a surjective projective morphism of noetherian schemes of pure relative dimension \(d\), let \(n\leqslant d\) be a positive integer, and let \(D_1,\ldots,D_n\) be effective Cartier divisors on \(X\). Let \(Z_i = D_1\cap \cdots \cap D_i\) and let \(Z_0 = X\). We say that \(D_1,\ldots,D_n\) is an \emph{\(f\)-regular sequence} if for each \(1\leqslant i \leqslant n\), \(Z_i\) is an effective Cartier divisor on \(Z_{i-1}\), and \(Z_i\to S\) is of pure relative dimension \(d-i\). 
\end{Def}

\begin{Prop}
	Whether the sequence \(D_1,\ldots,D_n\) is \(f\)-regular is independent of its order.
\end{Prop}
\begin{proof}
	For each \(1\leqslant i \leqslant n\), let \(s_{i}\) be the canonical section of \(\sO(D_i)\). Then \(Z_i\) is an effective Cartier divisor on \(Z_{i-1}\) is equivalent to that \(s_{i}|_{Z_{i-1}}\) is a regular section of \(\sO(D)|_{Z_{i-1}}\).
	
	By \cite{EGA4-2} (5.1.8), the dimension of each component of each fiber of \(X\to S\) decreases at most 1 after intersecting with each \(D_i\). Thus \(Z_i\to S\) is of pure relative dimension \(d-i\) for every \(1\leqslant i\leqslant n\) is equivalent to that \(Z_n\to S\) is of pure relative dimension \(d-n\).
	
	Therefore, \(D_1,\ldots,D_n\) is \(f\)-regular is equivalent to that \(s_{1},\ldots,s_{n}\) is regular, and that \(D_1\cap \cdots \cap D_n\to S\) of pure relative dimension \(d-n\). The latter condition is independent of the order.
\end{proof}

    Let \(f\colon X\to S\) be a surjective projective morphism of noetherian schemes of pure relative dimension \(d\), and suppose that \(S\) is normal. Let \(n\leqslant d\), and let \(D_1,D_2,\ldots,D_n\) be an \(f\)-regular sequence of effective Cartier divisors on \(X\). Let \(Y = D_1\cap D_2\cap\cdots\cap D_n\). Then by Theorem \ref{projection formula of taking closed subscheme} repeatedly, we get have an isomorphism
    \begin{equation}\label{eq:restriction to cartier divisor sequence}
    	\langle \sL_1,\ldots,\sO(D_1),\ldots,\sO(D_n),\ldots,\sL_{d+1-n} \rangle_{X/S}
    	\to \langle\sL_1|_Y,\ldots,\sL_{d+1-n}|_Y \rangle_{Y/S},
    \end{equation}
      where we just delete \(\sO(D_1),\dots,\sO(D_n)\) in the second expression. 

\begin{Prop}\label{restriction to divisor sequence indep of order}
	Assumptions as above. The isomorphism \eqref{eq:restriction to cartier divisor sequence} is independent of the order of restrictions. 
\end{Prop}
\begin{proof}
	we can assume that \(S\) integral, and then only need to verify the theorem over the generic point of \(S\). Then we reduce to the flat case, and then the result follows from \cite{DP2024} Proposition 6.1.
\end{proof}

\begin{Prop}\label{restriction compatible with symmetry}
	Assumptions and notations as above. For each \(1\leqslant j \leqslant n\), let \(i_j\) be the index of \(\sO(D_j)\) in the left side of \eqref{eq:restriction to cartier divisor sequence}, and let \(\sL_{i_j} = \sO(D_j)\). Let \(\phi\) be any permutation of \(\{1,2\ldots,d+1\}\). Let \(\sigma\) be the corresponding permutation of \(\{1,2,\dots,d+1-n\}\) such that \(\gamma_\sigma\) sends \(\langle\sL_1|_Y,\ldots,\hat{\sL}_{i_1},\ldots,\hat{\sL}_{i_n},\ldots,\sL_{d+1}|_Y \rangle_{Y/S}\) to \(\langle \sL_{\phi(1)}|_{Y},\ldots,\hat{\sL}_{\phi(i_1)},\ldots,\hat{\sL}_{\phi(i_n)},\ldots,\sL_{\phi(d+1)}|_{Y} \rangle_{Y/S} \). Then we have the following commutative diagram
	\[ \begin{tikzcd}
		\langle \sL_1,\ldots,\sL_{i_1},\ldots,\sL_{i_n},\ldots,\sL_{d+1} \rangle_{X/S}\arrow[d,"\gamma_{\phi}"]\arrow[r]& \langle\sL_1|_Y,\ldots,\hat{\sL}_{i_1},\ldots,\hat{\sL}_{i_n},\ldots,\sL_{d+1}|_Y \rangle_{Y/S}\arrow[d,"\gamma_\sigma"]\\
		\langle \sL_{\phi(1)},\ldots,\sL_{\phi(i_1)},\ldots,\sL_{\phi(i_n)},\ldots,\sL_{\phi(d+1)} \rangle_{X/S}\arrow[r] & \langle \sL_{\phi(1)}|_Y,\ldots,\hat{\sL}_{\phi(i_1)},\ldots,\hat{\sL}_{\phi(i_n)},\ldots,\sL_{\phi(d+1)}|_Y \rangle_{Y/S} 
	\end{tikzcd} \]
    where the two horizontal arrows are the canonical isomorphisms defined as \eqref{eq:restriction to cartier divisor sequence}.
\end{Prop}
\begin{proof}
	We can assume that \(S\) integral, and then only need to verify the theorem over the generic point of \(S\). Then we reduce to the flat case, and then the result follows from \cite{DP2024} Proposition 6.1.
\end{proof}

\begin{Def}
	Let \(f\colon X\to S\) be a surjective projective morphism of noetherian schemes of pure relative dimension \(d\), \(n\leqslant d\), let \(\sL_1,\ldots,\sL_n\) be invertible \(\sO_X\)-modules, and for each \(1\leqslant i\leqslant n\), let \(s_i\) be a global section of \(\sL_i\). We say that \(s_1,\ldots,s_n\) is an \emph{\(f\)-regular sequence} if \(s_1,\ldots,s_n\) is a regular sequence, and \(\) \(Z(s_1)\cap\cdots\cap Z(s_n)\to S\) is of pure relative dimension \(d-n\).
\end{Def}

    Note that if \(s_1,\ldots,s_n\) is an \emph{\(f\)-regular sequence}, then \(Z(s_1),\ldots,Z(s_n)\) is an \emph{\(f\)-regular sequence} of effective Cartier divisors.

	Let \(f\colon X\to S\) be a surjective projective morphism of noetherian schemes of pure relative dimension \(d\), \(n\leqslant d\), let \(\sL_1,\ldots,\sL_n\) be invertible \(\sO_X\)-modules, and for each \(1\leqslant i\leqslant n\), let \(s_i\) be a global section of \(\sL_i\). Suppose that \(S\) is normal, and that \(s_1,\ldots,s_n\) is an \(f\)-regular sequence. Let \(Y = Z(s_1)\cap Z(s_2)\cap\cdots\cap Z(s_n)\). Then by \eqref{eq:restriction to cartier divisor sequence}, we have a canonical isomorphism
    \begin{equation*}
    	\langle\sL_1,\sL_2,\ldots,\sL_{d+1}\rangle_{X/S}\to \langle\sL_{n+1},\sL_{n+2},\ldots,\sL_{d+1}\rangle_{Y/S},
    \end{equation*}
    which we denote by \([s_1,\ldots,s_n]_{X/S}\). In particular, if \(n=d\), then we have a canonical isomorphism
    \[ \langle\sL_1,\sL_2,\ldots,\sL_{d+1}\rangle_{X/S}\to\Nm_{Y/S}(\sL_{d+1}). \]
    Moreover, let \(s_{d+1}\) be a global section of \(\sL_{d+1}\), then \([s_1,\ldots,s_d]^{-1}_{X/S}(\Nm_{Y/S}(s_{d+1}))\) is a global section of \(\langle\sL_1,\sL_2,\ldots,\sL_{d+1}\rangle_{X/S}\), and we denote it by
    \begin{equation}\label{eq:section of regular sequence}
    	\langle s_1,\ldots,s_d,s_{d+1}\rangle_{X/S}
    \end{equation} Then Proposition \ref{restriction compatible with symmetry} implies the following proposition.
    
\begin{Prop}
	Assumptions as above. Let \(\phi\) be a permutation of the set \(\{1,2,\ldots,d+1\}\) such that \(\phi(d+1) = d+1\). Then
	\[ \gamma_{\phi}(\langle s_1,\ldots,s_d,s_{d+1}\rangle_{X/S}) = \langle s_{\phi(1)},\ldots,s_{\phi(d)},s_{d+1}\rangle_{X/S} \]
\end{Prop}

\begin{Thm}\label{thm: non regular section sequence}
	Let \(X\) and \(S\) be quasi-projective over schemes a field of characteristic 0, and let \(X\to S\) be a surjective projective morphism of pure relative dimension \(d\). Let \(\sL_1,\ldots,\sL_{d+1}\) be invertible \(\sO_X\)-modules, and for each \(1\leqslant i\leqslant d+1\), let \(s_i\) be a global section of \(\sL_i\). Suppose that \(S\) is normal, and for each \(1\leqslant i\leqslant d\), \(Z(s_1)\cap Z(s_{2})\cdots\cap Z(s_i)\to S\) is of pure relative dimension \(d-i\). Then we can define a global section of \(\langle\sL_1,\ldots,\sL_{d+1}\rangle_{X/S}\), which we denote by \(\langle s_1,\ldots,s_{d+1}\rangle_{X/S}\). Moreover, if \(\bigcap_{i=1}^{d+1}Z(s_i)  = \emptyset\), then \(\langle s_1,\ldots,s_{d+1}\rangle_{X/S}\) is a regular section of \(\langle\sL_1,\ldots,\sL_{d+1}\rangle_{X/S}\).
\end{Thm}
\begin{proof}
	We prove by induction on \(d\). For \(d = 0\), the proposition is trivial. Now suppose that \(d>0\), and we can assume that \(S\) is integral. Let \(X_1,\ldots,X_r\) be the irreducible components of \(X\) which dominates \(S\), and we endow each \(X_i\) with the reduced structure. For each \(X_i\), let \(\delta(X_i)\) be its multiplicity in \(X\). Let \(\eta\) be the generic point of \(S\), and let \(Z_i = Z(s_1)\cap X_i\). Note that \(s_1\) is a regular section on each \(X_i\). Then by \cite{YuanALB} Lemma 4.2.2, we have a canonical isomorphism
	\begin{align*}
		\langle\sL_{1},\ldots,\sL_{d+1}\rangle_{X/S,\eta} &= \langle\sL_1|_{X_\eta},\ldots,\sL_{d+1}|_{X_\eta}\rangle_{X_\eta/\eta}\\
		&=\otimes_{i=1}^r \langle\sL_{1}|_{X_{i,\eta}},\ldots,\sL_{d+1}|_{X_{i,\eta}}\rangle_{X_{i,\eta}/\eta}^{\otimes\delta(X_i)}\\
		&=\otimes_{i=1}^r \langle\sL_{2}|_{Z_{i,\eta}},\ldots,\sL_{d+1}|_{Z_{i,\eta}}\rangle_{Z_{i,\eta}/\eta}^{\otimes\delta(X_i)}
	\end{align*}
    By inductive hypothesis, for each \(1\leqslant i\leqslant r\), we have a global section \(\langle s_2,\ldots,s_{d+1}\rangle_{Z_{i,\eta}/\eta}\) of \(\langle\sL_{2}|_{Z_{i,\eta}},\ldots,\sL_{d+1}|_{Z_{i,\eta}}\rangle_{Z_{i,\eta}/\eta}\). Then by the above isomorphism, we get a section \(s_\eta\) of \(\langle\sL_{1},\ldots,\sL_{d+1}\rangle_{X/S,\eta}\). 
    
    For each codimension 1 point \(P\), let \(A = \Spec \sO_{S,P}\). Then by \cite{YuanALB} Lemma 4.2.2, we have a canonical isomorphism
    \begin{align*}
    	\langle\sL_{1},\ldots,\sL_{d+1}\rangle_{X/S, P} &= \langle\sL_1|_{X_A},\ldots,\sL_{d+1}|_{X_A}\rangle_{X_A/\Spec A}\\
    	&=\otimes_{i=1}^r \langle\sL_{1}|_{X_{i,A}},\ldots,\sL_{d+1}|_{X_{i,A}}\rangle_{X_{i,A}/\Spec A}^{\otimes\delta(X_i)}\\
    	&=\otimes_{i=1}^r \langle\sL_{2}|_{Z_{i,A}},\ldots,\sL_{d+1}|_{Z_{i,A}}\rangle_{Z_{i,A}/\Spec A}^{\otimes\delta(X_i)}
    \end{align*}
    By inductive hypothesis, for each \(1\leqslant i\leqslant r\), we have a global section \(\langle s_2,\ldots,s_{d+1}\rangle_{Z_{i,A}/\Spec A}\) of \(\langle\sL_{2}|_{Z_{i,A}},\ldots,\sL_{d+1}|_{Z_{i,A}}\rangle_{Z_{i,A}/\Spec A}\). Then by the above isomorphism, we get a section \(s_P\) of \(\langle\sL_{1},\ldots,\sL_{d+1}\rangle_{X/S,P}\).
    
    It is obvious that \(s_P = s_\eta\) on the generic fiber, and thus \(s_\eta\) is actually a global section of \(\langle\sL_1,\ldots,\sL_{d+1}\rangle_{X/S}\). We define \(\langle s_1,\ldots,s_{d+1}\rangle_{X/S}\) to be this global section. 
    
    Moreover, if \(\bigcap_{i=1}^{d+1}Z(s_i)  = \emptyset\), then for every \(1\leqslant i \leqslant r\), \(\langle s_2,\ldots,s_{d+1}\rangle_{Z_{i,\eta}/\eta}\) is  regular. Therefore, \(\langle s_1,\ldots,s_{d+1}\rangle_{X/S}\) is also regular.
\end{proof}

\section{Deligne pairing of morphisms}

\subsection{Flat case}\label{subsetion:morphism of FDP}

In this subsection, we recall some base facts on Deligne pairing of morphisms in the flat case. Let \(f\colon X\to S\) be a surjective flat projective morphism of noetherian schemes of pure relative dimension \(d\). \cite{Ducrot2005} defines the Deligne pairing to be
\begin{equation}\label{def:Ducrot's def of DP}
\langle\sL_1,\sL_2\ldots,\sL_{d+1}\rangle_{X/S} = \bigotimes_{I\subseteq\{1,2,\ldots,d+1\}}\left(\det Rf_*\bigotimes_{i\in I}\sL_i\right)^{(-1)^{d+1-\# I}},	
\end{equation}
where \(Rf_*\) is the derived complex. \cite{Ducrot2005} and \cite{ExplicitDP} prove the equivalence between this definition and \cite{Elkik}'s definition. Moreover, we can see the equivalence from the generation and relation presentation given in \cite{DP2024} \S 6.1.2.

For each \(1\leqslant i\leqslant d+1\), let \(u_i\colon\sL_i'\to \sL_i\) be a isomorphism of invertible \(\sO_X\)-modules. From the definition \eqref{def:Ducrot's def of DP}, we can easily get a morphism of invertible \(\sO_S\)-modules
\[ \langle u_1,\ldots,u_{d+1}\rangle_{X/S}\colon\langle \sL'_1,\ldots,\sL'_{d+1}\rangle_{X/S}\to\langle \sL_1,\ldots,\sL_{d+1}\rangle_{X/S}. \]
Moreover, if \(s_1,\ldots,s_d\) is an \(f\)-regular sequence,
%
then we have
\[ \langle u_1,\ldots,u_{d+1}\rangle_{X/S}(\langle s_1,\ldots,s_{d+1}\rangle_{X/S}) = \langle u_1(s_1),\ldots,u_{d+1}(s_{d+1})\rangle_{X/S}. \]
We have the following commutative diagrams
\begin{equation*}
	\begin{tikzcd}
		\langle \sL_1,\sL_2,\ldots,\sL_{d+1}\rangle_{X/S}\arrow[r, "\gamma_\phi"]\arrow[d,"\langle u_1{,} u_2{,} \ldots{,}u_{d+1}\rangle_{X/S}"]&\langle \sL_{\phi(1)},\sL_{\phi(2)},\ldots,\sL_{{\phi(d+1)}}\rangle_{X/S}\arrow[d,"\langle u_{\phi(1)}{,} u_{\phi(2)}{,} \ldots{,}u_{\phi(d+1)}\rangle_{X/S}"]\\
		\langle\sL'_1,\sL'_2,\ldots,\sL'_{d+1}\rangle_{X/S}\arrow[r, "\gamma_\phi"]&\langle \sL'_{\phi(1)},\sL'_{\phi(2)},\ldots,\sL'_{{\phi(d+1)}}\rangle_{X/S},
	\end{tikzcd}
\end{equation*}
\begin{equation*}
	\begin{tikzcd}
		\langle \sL_1,\ldots,\sL_i,\ldots\sL_{d+1}\rangle_{X/S}\otimes\langle \sL_1,\ldots,\sL'_i,\ldots\sL_{d+1}\rangle_{X/S} \arrow[r, "\Sigma_i"]\arrow[d,"\langle u_1{,}\ldots u_i{,} \ldots{,}u_{d+1}\rangle_{X/S}\otimes\langle u_1{,}\ldots u'_i{,} \ldots{,}u_{d+1}\rangle_{X/S}"]&
		\langle \sL_1,\ldots,\sL_i\otimes\sL'_i,\ldots\sL_{d+1}\rangle_{X/S}\arrow[d,"\langle u_1{,}\ldots u_i\otimes u'_i{,} \ldots{,}u_{d+1}\rangle_{X/S}"]\\
		\langle \sM_{1},\ldots,\sM_i,\ldots,\sM_{d+1}\rangle_{X/S}\otimes\langle \sM_{1},\ldots,\sM'_i,\ldots\sM_{d+1}\rangle_{X/S} \arrow[r, "\Sigma_{i}"]&
		\langle \sM_{1},\ldots,\sM_i\otimes\sM'_i,\ldots\sL_{d+1}\rangle_{X/S},
	\end{tikzcd}
\end{equation*}
\begin{equation*}
	\begin{tikzcd}
		\langle \sL_1,\ldots,\sL_i,\ldots,\sL_{d+1}\rangle_{X/S}\arrow[r, "{[}s_i{]}_{X/S}"]\arrow[d,"\langle u_1{,}\ldots u_i{,} \ldots{,}u_{d+1}\rangle_{X/S}"]&\langle \sL_{1}|_D,\ldots,\hat{\sL}_i,\ldots\sL_{d+1}|_D\rangle_{D/S}\arrow[d,"\langle u_{1}{,}\ldots{,}\hat{u}_i{,} \ldots{,}u_{d+1}\rangle_{X/S}"]\\
		\langle\sL'_1,\ldots,\sL'_i,\ldots,\sL'_{d+1}\rangle_{X/S}\arrow[r, "{[}u_i(s_i){]}_{X/S}"]&\langle \sL'_{1}|_D,\ldots,\hat{\sL}'_i,\ldots\sL'_{d+1}|_{{D}}\rangle_{{D}/S},
	\end{tikzcd}
\end{equation*}
where \(\gamma_\phi\) and \(\Sigma_i\) are the counterparts of \eqref{eq:symmetric of EDP of lb} and \eqref{eq:additive of EDP of lb} for Deligne pairing for flat morphisms, \(u_i\) is an isomorphism, and in the last diagram, \(s_i\) is an \(f\)-regular section of \(\sL_i\) (so \(u_i(s_i)\) is also an \(f\)-regular section), and \(D = Z(s_i) = Z(u_i(s_i))\).

\subsection{Equidimensional case}

Let \(f\colon X\to S\) be a surjective projective morphism of noetherian schemes of pure relative dimension \(d\). Suppose that \(S\) is integral and normal, and let \(\eta\) be the generic point of \(S\). Let \(f_\eta\colon X_\eta \to \eta\) be the base change of \(f\) to \(\eta\), and we see that \(f_\eta\) is a surjective flat projective morphism of noetherian schemes of pure relative dimension \(d\).

Let \(\sL_1,\sL_2,\ldots,\sL_{d+1},\sL'_1,\sL'_2,\ldots,\sL'_{d+1}\) be invertible \(\sO_{X}\) modules, and for each \(1\leqslant i\leqslant d+1\), let \(u_{i}\colon\sL_i\to\sL'_i\) be an isomorphism of \(\sO_X\)-modules. 

For an isomorphism \(u\colon \sL\to \sL'\) of invertible sheaves, we have \(u^{-1}\colon\sL'\to \sL\), which induces an isomorphism \(\sL^{-1}\to \sL'^{-1}\). We denote this isomorphism by \(u^{\otimes(-1)}\). Note that \(u\otimes u^{\otimes(-1)} = u^{\otimes(-1)}\otimes u = \id\).

\begin{Thm}\label{thm:morphism of EDP}
	There is a unique isomorphism
	\[ \varphi\colon\langle\sL_1,\sL_2,\ldots,\sL_{d+1}\rangle_{X/S}\to \langle\sL'_1,\sL'_2,\ldots,\sL'_{d+1}\rangle_{X/S} \]
	such that
	\[ \varphi_\eta = \langle u_{1,\eta},u_{2,\eta},\ldots,u_{d+1,\eta}\rangle_{X_\eta/\eta}. \]
\end{Thm}
\begin{proof}
	The uniqueness is obvious. For the existence, we can assume that \(S\) is affine.

\begin{Lem}\label{lem:differenece lemma}
	Suppose that for \(1\leqslant i\leqslant d+1\), \(\sL_i = \sL_{i,0}\otimes\sL_{i,1}^{-1}\), \(\sL'_i = \sL'_{i,0}\otimes{\sL'_{i,1}}^{-1}\), and \(u_i = u_{i,0}\otimes u_{i,1}^{\otimes(-1)}\), where both \(u_{i,0}\colon \sL_{i,0}\to \sL'_{i,0}\) and \(u_{i,1}\colon \sL_{i,1}\to \sL'_{i,1}\) are isomorphisms. Let \(E\) is the set of all the maps from \(\{1,2,\ldots,d+1\}\) to \(\{0,1\}\). Suppose that for each \(h\in E\), there is an isomorphism
	\[ \varphi_h\colon\langle\sL_{1,h(1)},\sL_{2,h(2)},\ldots,\sL_{d+1,h(d+1)}\rangle_{X/S}\to \langle\sL'_{1,h(1)},\sL'_{2,h(2)},\ldots,\sL'_{d+1,h(d+1)}\rangle_{X/S} \]
	such that
	\[ \varphi_{h,\eta} = \langle u_{1,h(1),\eta},u_{2,h(2),\eta},\ldots,u_{d+1,h(d+1),\eta}\rangle_{X_\eta/\eta}. \]
	Then there is an isomorphism
	\[ \varphi\colon\langle\sL_1,\sL_2,\ldots,\sL_{d+1}\rangle_{X/S}\to \langle\sL'_1,\sL'_2,\ldots,\sL'_{d+1}\rangle_{X/S} \]
	such that
	\[ \varphi_\eta = \langle u_{1,\eta},u_{2,\eta},\ldots,u_{d+1,\eta}\rangle_{X_\eta/\eta}. \]
\end{Lem}
\begin{proof}
	For each \(h\in E\), let \(\epsilon_h = \prod_{i=1}^{d}(-1)^{h(i)}\). We know that
	\[ \langle\sL_1,\sL_2,\ldots,\sL_{d+1}\rangle_{X/S}=\bigotimes_{h\in E}\langle\sL_{1,h(1)},\sL_{2,h(1)},\ldots,\sL_{d+1,h(d+1)}\rangle^{\epsilon_h}_{X/S} \]
	and
	\[ \langle\sL'_1,\sL'_2,\ldots,\sL'_{d+1}\rangle_{X/S}=\bigotimes_{h\in E}\langle\sL'_{1,h(1)},\sL'_{2,h(1)},\ldots,\sL'_{d+1,h(d+1)}\rangle^{\epsilon_h}_{X/S}. \]
	Let \( \varphi = \bigotimes_{h\in E} \varphi_h^{\otimes\epsilon_h}\). Then \(\varphi\) satisfies all requirements.
\end{proof}

Since we can find an ample invertible \(\sO_X\)-module \(\sM\) such that \(\sM,\sL_1\otimes\sM,\sL_2\otimes\sM,\ldots,\sL_{d+1}\otimes\sM,\sL'_1\otimes\sM,\sL'_2\otimes\sM,\ldots,\sL'_{d+1}\otimes\sM\) are all generated by global sections, by Lemma \ref{lem:differenece lemma}, we can assume that \(\sL_1,\sL_2,\ldots,\sL_{d+1},\sL'_1,\sL'_2,\ldots,\sL'_{d+1}\) are all globally-generated invertible \(\sO_X\)-modules. By the same argument, we can assume that \(\sL_1,\sL_2,\ldots,\sL_{d+1},\sL'_1,\sL'_2,\ldots,\sL'_{d+1}\) are all \(f\)-very ample invertible \(\sO_X\)-modules.

\begin{Lem}\label{lem:existence of f-regular section}
	Let \(f\colon X\to S\) be a surjective projective morphism of noetherian schemes of pure relative dimension \(d\geqslant 1\). Suppose that \(S\) is an affine scheme. Let \(\sL\) be an \(f\)-very ample invertible \(\sO_X\)-module. Then there exists an integer \(n_0>0\) such that for all integers \(n\geqslant n_0\), there exists a regular global section \(s\) of \(\sL^{\otimes n}\) such that \(Z(s)\to S\) is of pure relative dimension \(d-1\).
\end{Lem}
\begin{proof}
	This is an easy application of \cite{Gabber-Liu-Lorenzini} Theorem 5.1.
\end{proof}

By apply Lemma \ref{lem:existence of f-regular section} for \(\sL_1,\sL_2,\ldots,\sL_d\) in turn, we have find positive integer numbers \(n_1,n_2,\ldots,n_d\) such that for each map \(h\in E\), the sequence \(\sL_1^{\otimes (n_1+h(1))},\linebreak\sL_2^{\otimes (n_2+h(2))},\ldots,\sL_d^{\otimes (n_d+h(d))}\) admits an \(f\)-regular sequence of global sections. Then by Lemma \ref{lem:existence of f-regular section}, we can assume that \(\sL_1,\sL_2,\ldots,\sL_{d}\) admits an \(f\)-regular sequence of global sections.

Now we prove Theorem \ref{thm:morphism of EDP} under the assumptions that \(\sL_1,\sL_2,\ldots,\sL_{d+1},\sL'_1,\sL'_2,\linebreak\ldots,\sL'_{d+1}\) are all \(f\)-very ample and that \(\sL_1,\sL_2,\ldots,\sL_{d}\) admits an \(f\)-regular sequence of global sections. Let \(s_1,s_2,\ldots,s_d\) be an \(f\)-regular sequence of global sections. Note that \(u_1(s_1),u_2(s_2),\ldots,u_d(s_d)\) is also an \(f\)-regular sequence of global sections. Then define \(\varphi\) to be the morphism such that for every open subset \(U\subseteq S\) and every section \(s_{d+1}\in\Gamma(U,\sL_{d+1})\), \(\varphi\) sends \(\langle s_1,s_2,\ldots,s_d,s_{d+1} \rangle_{X_U/U}\) to \(\langle u_1(s_1),u_2(s_2),\ldots,u_d(s_d),u_{d+1}(s_{d+1}) \rangle_{X_U/U}\). It is easy to see that \(\varphi\) is a well-defined isomorphism and that \(\varphi_\eta = \langle u_{1,\eta},u_{2,\eta},\ldots,u_{d+1,\eta}\rangle_{X_\eta/\eta}\).\qedhere
\end{proof}

We denote the isomorphism \(\varphi\) in Theorem \ref{thm:morphism of EDP} by \(\langle u_1,u_2,\ldots,u_{d+1}\rangle_{X/S}\). Since \(\langle u_1,u_2,\ldots,u_{d+1}\rangle_{X/S}\) is identical to \(\langle u_{1,\eta},u_{2,\eta},\ldots,u_{d+1,\eta}\rangle_{X_\eta/\eta}\) over the generic fiber, we see that \(\langle u_1,u_2,\ldots,u_{d+1}\rangle_{X/S}\) naturally satisfies the commutative diagrams listed at the end of \S\ref{subsetion:morphism of FDP}.

By looking at the generic fibers, we see that \(\Sigma\)'s, \(\gamma\)'s, isomorphisms in Theorem \ref{base change of normal Deligne pairing}, \ref{pull-back formula of normal Deligne pairing}, \ref{projection formula of normal Deligne pairing}, \ref{projection formula of taking closed subscheme}, and \ref{birational invariance} are all compatible with isomorphisms in \(\cPic(X)\), and they are also compatible with each other (see the precise statements in \cite{DP2024} Proposition 6.1, 6.10, and 6.13). 

We summarize our results in this section as the following categorical version of Theorem \ref{EDP}.

\begin{Thm}\label{thm:functorial EDP}
	Let \(f\colon X\to S\) be a surjective projective morphism of noetherian schemes of pure relative dimension \(d\). Suppose that \(S\) is normal. Then our Deligne pairing defines a symmetric multi-additive functor
	\begin{equation*}
		\langle\sL_1,\sL_2,\ldots,\sL_{d+1}\rangle_{X/S}\colon\cPic(X)^{d+1}\to\cPic(S).
	\end{equation*}
    And \(\Sigma\)'s, \(\gamma\)'s, isomorphisms in Theorem \ref{base change of normal Deligne pairing}, \ref{pull-back formula of normal Deligne pairing}, \ref{projection formula of normal Deligne pairing}, \ref{projection formula of taking closed subscheme}, and \ref{birational invariance} are all natural isomorphisms of functors.
\end{Thm}

\section{Deligne pairing of hermitian line bundles}

In this section, we follow ideas from \cite{Elkik2} and \cite{YuanALB} to construction a canonical metric on Deligne pairing of hermitian line bundles for equidimensional morphisms.

Readers can refer to \cite{agbook} Chapter 2 \S5.1 and \S4.3 for the definitions of a complex analytic space \(X\) and its regular locus \(X_{\reg}\). For concepts on complex analytic spaces, we follow the definitions in \cite{YuanALB} \S2.1.1 and \S4.2.2., and we recall some important concepts as follows.

\begin{Def}
	Let \(X\) be a complex analytic space. A function \(f\colon X\to \bC\) is called \emph{smooth} if for each point \(x\in X\), there is  an open neighborhood \(U\) of \(x\) and an analytic map \(U\to M\), where \(M\) is a complex manifold and \(U\) is embedded as a closed analytic subspace, such that \(f\) is the restriction of some smooth function on \(M\).
\end{Def}

\begin{Def}
	Let \(X\) be a complex analytic space. A \emph{smooth form} \(\alpha\) on \(X\) is a smooth form \(\alpha\) on \(X_{\reg}\) such that for each point \(x\in X\), there is an open neighborhood \(U\) of \(x\) and an analytic map \(U\to M\), where \(M\) is a complex manifold and \(U\) is embedded as a closed analytic subspace, such that \(\alpha|_{U_{\reg}}\) is the restriction of some smooth form on \(M\).
\end{Def}

Let \(X\) be an integral quasi-projective scheme over \(\bC\), and let \(L\) be an invertible \(\sO_X\)-module. For every point \(x\in X\), let \(k(x)\) be the residue field of \(x\), and let \(L(x) = L\times_X \Spec k(x)\) be the fiber over \(x\), which is a 1-dimensional vector space over \(k(x)\). A continuous (resp. smooth) metric on \(L\) is a families of metrics, one on each \(L(x)\), such that for every open subset \(U\subseteq X\) and every section \(s\in\Gamma(U,L)\) (such \(s\) is called a \emph{rational section} of \(L\)), \(\Vert s(x)\Vert^2\) is a continuous (resp. smooth) function on \(U\). For any smooth metric \(\Vert\cdot\Vert\) on \(L\), the Chern form, which locally can be written as
\[c_1(L,\Vert\cdot\Vert) = \frac{1}{\uppi\myi}\partial\bar{\partial}\log\Vert s\Vert,\]
where \(s\) is a rational section of \(L
\), is a (1,1)-smooth form on \(X\).

Let \((\sL_1,\VCV_1)\) and \((\sL_2,\VCV_2)\) be two invertible \(\sO_X\)-modules endowed with smooth metrics. Let \(\varphi\colon \sL_1 \to \sL_2\) be an isomorphism of invertible sheaves. Then the norm of \(\varphi\) is defined to be
\[ \Vert\varphi\Vert = \frac{\varphi^*\VCV_2}{\VCV_1}. \]
Note that \(\Vert\varphi\Vert\) is a nonzero smooth function on \(X\). Locally, if \(s\) is a rational section of \(\sL_1\), then
\[ \Vert\varphi\Vert = \frac{\Vert\varphi(s)\Vert_2}{\Vert s \Vert_1}. \]
Therefore
\[ c_1(L_2,\Vert\cdot\Vert_2) = c_1(L_1,\Vert\cdot\Vert_1) + \frac{1}{\uppi\myi}\partial\bar{\partial}\Vert\varphi\Vert. \]

For convenience, we generalize the conception of \(f\)-regular sequence for \(f\) of pure relative dimension 1.

\begin{Def}
	Let \(f\colon X\to S\) be a surjective projective morphism of noetherian schemes of pure relative dimension 1, let \(\sL_1\) and \(\sL_2\) be invertible \(\sO_X\)-modules, and let \(s_1\) and \(s_2\) be global sections of \(\sL_1\) and \(\sL_2\) respectively. Let \(Y = Z(s_1)\cap Z(s_2)\). We say that \(s_1,s_2\) is an \emph{\(f\)-regular sequence} if \(s_1,s_2\) is a regular sequence, both \(Y_1\to S\) and \(Y_2\to S\) are finite, and \(Y = \emptyset\).
\end{Def}
    
Locally, if we identify both \(\sL_1\) and \(\sL_2\) with \(\sO_X\), then \(Y = \emptyset\) means that \(\sO_{X} = (s_{1}, s_{2})\).

Let \(f\colon X\to S\) be a surjective projective morphism of noetherian schemes of pure relative dimension 1, let \(\sL_1\) and \(\sL_2\) be invertible \(\sO_X\)-modules, and let \(s_1\) and \(s_2\) be global sections of \(\sL_1\) and \(\sL_2\) respectively. Suppose that \(S\) is normal and that \(s_1,s_2\) is an \(f\)-regular sequence. Since \(s_1|_{Z_2}\) is regular, \(s_1|_{Z_2}\colon\sO_{Z_2}\to \sL_1|_{Z_2}\) is injective. Since locally \(\sO_{X} = (s_{1}, s_{2})\), we deduce that \(s_1|_{Z_2}\colon\sO_{Z_2}\to \sL_1|_{Z_2}\) is surjective. Thus \(s_1|_{Z_2}\colon\sO_{Z_2}\to \sL_1|_{Z_2}\) is an isomorphism. By the same argument, we know that \(s_2|_{Z_1}\colon\sO_{Z_1}\to \sL_2|_{Z_1}\) is also an isomorphism. Then both \(\Nm_{Z_2/S}(s_1|_{Z_2})\colon\sO_{S}\to \Nm_{Z_2/S}(\sL_1|_{Z_2})\) and \(\Nm_{Z_1/S}(s_2|_{Z_1})\colon\sO_{S}\to \Nm_{Z_1/S}(\sL_2|_{Z_1})\) are isomorphisms. Let \([s_1]_{Z_2/S}\) and \([s_2]_{Z_1/S}\) denote the inverses of \(\Nm_{Z_2/S}(s_1|_{Z_2})\) and \(\Nm_{Z_1/S}(s_2|_{Z_1})\) respectively.

\begin{Prop}
	Assumptions and notations as above. Then have the following commutative diagrams
\end{Prop}
\begin{equation}\label{eq:restriction on 2 divisors indep of oeder}
	\begin{tikzcd}[row sep=small]
		&\langle\sL_2\rangle_{Y_1/S}\arrow[rd, "{[}s_2{]}_{Y_1/S}"]& \\
		\langle\sL_1,\sL_2\rangle_{X/S}\arrow[ur, "{[}s_1{]}_{X/S}"']\arrow[rd, "{[}s_2{]}_{X/S}"]& & \sO_S,\\
		&\langle\sL_1\rangle_{Y_2/S}\arrow[ru, "{[}s_1{]}_{Y_2/S}"']&
	\end{tikzcd}
\end{equation}
where \([s_2]_{X/S}\colon\langle\sL_1,\sL_2\rangle_{X/S}\to \langle\sL_1\rangle_{Y_2/S}\) is the isomorphism in Theorem \ref{projection formula of taking closed subscheme} for \(D = Y_2 = Z(s_2)\).
\begin{proof}
	we can assume that \(S\) integral, and then only need to verify the theorem over the generic point of \(S\). Then we reduce to the flat case, and then the result follows from \cite{Garcia2000} Proposition 4.3.1 and Theorem 4.2.7.
\end{proof}

The following is our main theorem regarding Deligne pairing of hermitian line bundles for equidimensional morphisms.

\begin{Thm}\label{thm:metric on EDP}
	Let \(X\) and \(S\) quasi-projective \(\bC\)-schemes, and let \(f\colon X\to S\) be a surjective projective \(\bC\)-morphism of pure relative dimension \(d\). Suppose that \(S\) is normal. Let \((\sL_1,\VCV_1),(\sL_2,\VCV_2),\ldots,(\sL_{d+1},\VCV_{d+1})\) be invertible sheaves on \(X\) endowed with smooth metrics. Then we can endow \(\langle\sL_1,\sL_2,\ldots,\sL_{d+1}\rangle_{X/S}\) with a unique smooth metric \(\VCV\) such that
	
	(1)
	\[ c_1(\langle\sL_1,\sL_2,\ldots,\sL_{d+1}\rangle_{X/S},\VCV) = \int_{X/S}\prod_{i=1}^{d+1}c_1(\sL_i,\VCV_i).\]
	
	(2) Let \(d\geqslant 1\), and \(s_1\) be an \(f\)-regular section of \(\sL_1\), then
	\begin{equation}\label{eq:norm of restrictions}
		\log \Vert[s_1]_{X/S}\Vert = -\int_{X/S}\log\Vert s_1\Vert_1 \prod_{i=2}^{d+1}c_1(\sL_i,\VCV_i).
	\end{equation}
	
	(3) If \(d = 0\), then
	\[ \log \Vert\Nm_{X/S}(s)\Vert = \int_{X/S}\log\Vert s\Vert_1 \]
	for each nonzero rational section \(s\) of \(\sL_1\).
	
	(4) The \(\Sigma\)'s and \(\iota\)'s defined in \eqref{eq:additive of EDP of lb} and \eqref{eq:symmetric of EDP of lb} are isometries.
	
	(5) For each good base change (in the sense of Definition \ref{def: good base change}) \(g\colon S'\to S\), the isomorphism
	\begin{equation}\label{eq:isom of good base change}
		\langle g'^*\sL_1,g'^*\sL_2,\ldots,g'^*\sL_{d+1} \rangle_{X'/S'}\to g^*\langle \sL_1,\sL_2,\ldots,\sL_{d+1} \rangle_{X/S}.
	\end{equation}
	is an isometry, where \(X' = X\times_SS'\), \(g'\colon X'\to X\), and for each \(1\leqslant i\leqslant d+1\), the metric on \(g'^*\sL_i\) is \(g'^*\VCV_i\).
\end{Thm}

The symbol \(\int_{X/S}\) should be understood as follows. Let \(n = \dim S\). If \(\alpha\) is a smooth \((n',n'')\)-form on \(S\). Then \(\int_{X/S}\alpha\) is the unique smooth \((n'-d,n''-d)\)-form on \(X\) such that for any smooth \((n+d-n',n+d-n'')\)-form \(\beta\) with compact support on \(S\), we have the following equation
\[ \int_X\alpha f^*\beta = \int_S\left(\int_{X/S}\alpha\right)\beta. \]
Under suitable assumptions, the symbol \(\int_{X/S}\) can also be interpreted as the integral on fibers with multiplicities (refer to \cite{YuanALB} \S4.2.2).

\begin{proof}
	We prove by induction on \(d\). For \(d= 0\), the metric on \(\Nm_{X/S}(\sL_1)\) is uniquely determined by (3). We see that
	\begin{multline*}
		c_1(\Nm_{X/S}(\sL_1), \VCV) = \frac{1}{\uppi\myi}\pbp\log \Vert\Nm_{X/S}(s)\Vert =\\
		\int_{X/S}\frac{1}{\uppi\myi}\pbp\log\Vert s\Vert_1 = \int_{X/S}c_1(\sL_1,\VCV_1).
	\end{multline*}
	Obviously, the additive isomorphism \(\Sigma_1\) is an isometry.
	
	For each good base change \(g\colon S'\to S\) and each nonzero rational section \(s\) of \(\sL_1\),
	\begin{multline*}
		\log \Vert\Nm_{X'/S'}(g'^*s)\Vert_{S'} = \int_{X'/S'}\log g'^*\Vert g'^*s\Vert_1\\
		= \int_{X/S}\log \Vert s\Vert_1 = \log \Vert\Nm_{X/S}(s)\Vert_{S},
	\end{multline*}
    thus \(\VCV_{S'} = g^*\VCV_S\).
	
	Now suppose that \(d>0\).
	
	First we consider the case that \(\sL_1\) admits an \(f\)-regular global section. Let \(s_1\) be an \(f\)-regular global section of \(\sL_1\), and let \(Y_1 = Z(s_1)\). Then we have an isomorphism
	\[ [s_1]_{X/S}\colon\langle\sL_1,\sL_2,\ldots,\sL_{d+1}\rangle_{X/S}\to\langle\sL_2,\ldots,\sL_{d+1}\rangle_{Y_1/S} \]
	By inductive hypothesis, there is a unique smooth metric \(\VCV_{Y_1}\) on \(\langle\sL_2,\ldots,\sL_{d+1}\rangle_{Y_1/S}\). Then we can endow \(\langle\sL_1,\sL_2,\ldots,\sL_{d+1}\rangle_{X/S}\) with a smooth metric such that \(\Vert[s_1]_{X/S}\Vert\) satisfies \eqref{eq:norm of restrictions}. By Poincar\'e--Lelong formula, we have an equation of currents
	\[ \inversepii\pbp\log\Vert s_1\Vert_1 = c_1(\sL_1,\VCV_1) - [Y_1]. \]
	By inductive hypothesis,
	\[ c_1(\langle\sL_2,\ldots,\sL_{d+1}\rangle_{Y_1/S},\VCV_{Y_1}) = \int_{Y_1/S}\prod_{i=2}^{d+1}c_1(\sL_i,\VCV_i). \]
	Thus
	\begin{align*}
		c_1(\langle\sL_1,\sL_2,\ldots,&\sL_{d+1}\rangle_{X/S},\VCV)\\
		&= c_1(\langle\sL_2,\ldots,\sL_{d+1}\rangle_{Y_1/S},\VCV_{Y_1}) - \inversepii\pbp\Vert[s_1]_{X/S}\Vert\\
		&= \int_{Y_1/S}\prod_{i=2}^{d+1}c_1(\sL_i,\VCV_i) + \int_{X/S}\inversepii\pbp\log\Vert s_1\Vert_1 \prod_{i=2}^{d+1}c_1(\sL_i,\VCV_i)\\
		&= \int_{X/S}\prod_{i=1}^{d+1}c_1(\sL_i,\VCV_i).
	\end{align*}

If we fix \(\sL_1\) and vary \(\VCV_1, (\sL_2,\VCV_2),(\sL_3,\VCV_3),\ldots,(\sL_{d+1},\VCV_{d+1})\), then by the above construction, we can endow all Deligne pairings of the form \(\langle\sL_1,\sL_2,\ldots,\sL_{d+1}\rangle_{X/S}\) with a smooth metric. By inductive hypothesis, we see that \(\Sigma_i\, (i>1)\) and \(\gamma_\phi\) (\(\phi\) is a permutation of the set \{2,3,\ldots,d+1\}) are all isometries.

Let \(g\colon S'\to S\) be a good base change, \(Y = Z(s_1)\), and \(Y' = Z(g'^*s_1)\). Then \(g'^*s\) is an \(f'\)-regular section (\(f'\colon X'\to S'\)) of \(g'^*\sL_1\), and \(Y' = g'^*(Y)\).  Then we have the following commutative diagram
\[ \begin{tikzcd}
	\langle g'^*\sL_1,g'^*\sL_2,\ldots,g'^*\sL_{d+1} \rangle_{X'/S'}\arrow[r]\arrow[d, "{[}g'^*s_1{]}_{X'/S'}"]& g^*\langle \sL_1,\sL_2,\ldots,\sL_{d+1} \rangle_{X/S}\arrow[d, "g^*{[}s_1{]}_{X/S}"]\\
	\langle g'^*\sL_2,\ldots,g'^*\sL_{d+1} \rangle_{Y'/S'}\arrow[r]& g^*\langle \sL_2,\ldots,\sL_{d+1} \rangle_{Y/S}.
\end{tikzcd} \]
By inductive hypothesis, the bottom arrow is an isometry. Since
\begin{multline*}
	\log \Vert[g'^*s_1]_{X'/S'}\Vert = -\int_{X'/S'}\log g'^*\Vert g'^*s_1\Vert_1 \prod_{i=2}^{d+1}c_1(g'^*\sL_i,g'^*\VCV_i)\\
	= -\int_{X/S}\log \Vert s_1\Vert_1 \prod_{i=2}^{d+1}c_1(\sL_i,\VCV_i) = \log\Vert[s_1]_{X/S}\Vert,
\end{multline*}
we deduce that the top arrow is also an isometry.

\begin{Lem}\label{lem:norm of s_2}
	Let \(s_1\) and \(s_2\) be global sections of \(\sL_1\) and \(\sL_2\) respectively, and suppose that \(s_1,s_2\) is an \(f\)-regular sequence. Let \(Y = Z(s_1)\), \(Y_2 = Z(s_2)\), and \(Y = Y_1\cap Y_2\). If we endow \(\langle\sL_1,\sL_2,\ldots,\sL_{d+1}\rangle_{X/S}\) with the metric we constructed above, then
	\begin{equation}\label{eq:norm of s_2}
		\log\Vert[s_2]_{X/S}\Vert = -\int_{X/S}\log\Vert s_2\Vert_2\prod_{i\neq 2}c_1(\sL_i,\VCV_i),
	\end{equation}
	where \([s_2]_{X/S}\colon\langle\sL_1,\sL_2,\ldots,\sL_{d+1}\rangle_{X/S}\to \langle\sL_1,\sL_3,\ldots,\sL_{d+1}\rangle_{Y_2/S}\) is the isomorphism in Theorem \ref{projection formula of taking closed subscheme} for \(D = Z(s_2)\). In particular, \(\Vert[s_2]_{X/S}\Vert\) is independent of the choice of \(s_1\) (though we do not know that whether \(\VCV\) depends on the choice of \(s_1\) at present).
\end{Lem}
\begin{proof}
	We have the following commutative diagram
	\begin{equation}\label{eq:cd for metric}
		\begin{tikzcd}[row sep=small]
			&\langle\sL_2,\sL_3,\ldots,\sL_{d+1}\rangle_{Y_1/S}\arrow[rd, "{[}s_2{]}_{Y_1/S}"]& \\
			\langle\sL_1,\sL_2,\sL_3,\ldots,\sL_{d+1}\rangle_{X/S}\arrow[ur, "{[}s_1{]}_{X/S}"']\arrow[rd, "{[}s_2{]}_{X/S}"]& & \langle\sL_3,\ldots,\sL_{d+1}\rangle_{Y/S}.\\
			&\langle\sL_1,\sL_3,\ldots,\sL_{d+1}\rangle_{Y_2/S}\arrow[ru, "{[}s_1{]}_{Y_2/S}"']&
		\end{tikzcd}
	\end{equation}
    If \(d>1\), \eqref{eq:cd for metric} follows from Theorem \ref{restriction to divisor sequence indep of order}. If \(d=1\), \eqref{eq:cd for metric} is exactly \eqref{eq:restriction on 2 divisors indep of oeder}. In this case, if we endow \(\sO_S\) with the metric \(\VCV_S\) satisfying \(\Vert1 \Vert_S = 1\), then we see that
    \[ \log \Vert[s_1]_{Y_2/S}\Vert = -\int_{Y_2/S}\log\Vert s_1\Vert_1 \quad\text{and}\quad \log \Vert[s_2]_{Y_1/S}\Vert = -\int_{Y_1/S}\log\Vert s_2\Vert_2.\]
    These notations are compatible with that in \eqref{eq:norm of restrictions}, which enables us to give a unified treatment for the cases \(d >1\) and \(d = 1\).

	Let \(\omega = \prod_{i>2}c_1(\sL_i, \VCV_i)\). By \cite{YuanALB} Chapter 4 \S4.2.2 Lemma 4.2.1,
	\begin{multline*}
		\int_{X/S}\log\Vert s_1\Vert_1c_1(\sL_2,\VCV_2)\omega - \int_{X/S}\log\Vert s_2\Vert_2c_1(\sL_1,\VCV_1)\omega \\
		= \int_{Y_2/S}\log\Vert s_1\Vert_1\omega - \int_{Y_1/S}\log\Vert s_2\Vert_2\omega.
	\end{multline*}
    Then
    \begin{align*}
    	\log \Vert[s_2]_{X/S}\Vert &= \log \Vert[s_1]_{X/S}\Vert + \log \Vert[s_2]_{Y_1/S}\Vert - \Vert\log [s_1]_{Y_2/S}\Vert\\
    	&= -\int_{X/S}\log\Vert s_1\Vert_1c_1(\sL_2,\VCV_2)\omega - \int_{Y_1/S}\log\Vert s_2\Vert_2\omega\\
    	&\phantom{= -\int_{X/S}\log\Vert s_1\Vert_1c_1(\sL_2,\VCV_2)\omega - \int_{Y_1/S}}+ \int_{Y_2/S}\log\Vert s_1\Vert_1\omega\\
    	&= \int_{X/S}\log\Vert s_2\Vert_2c_1(\sL_1,\VCV_1)\omega,
    \end{align*}
    where the second equation follows from the inductive hypothesis.
\end{proof}

\begin{Lem}\label{lem:existence of simultaneously regular section}
	Let \(f\colon X\to S\) be a surjective projective morphism of noetherian schemes of pure relative dimension \(d\geqslant 1\), let \(\sL_1\) be an invertible \(\sO_X\)-module, and let \(s_1\) and \(s'_1\) be two \(f\)-regular global sections of \(\sL_1\). Let \(\sL_2\) be an \(f\)-very ample invertible \(\sO_X\)-module. Then for every point \(t\in S\), there exists an open neighborhood \(U\) of \(t\) and an integer \(n>0\) such that there exists a regular section \(s_2\in\Gamma( X_U,\sL_2^{\otimes n})\) such that both \(s_1,s_2\) and \(s'_1,s_2\) are \(f_U\)-regular sequences, where \(f_U\colon X_U\to U\).
\end{Lem}
\begin{proof}
	Let \(Y_1 = Z(s_1)\) and \(Y_2 = Z(s_2)\). By shrinking \(S\), we can assume that \(S\) is an affine scheme. If \(d\geqslant 2\), then the results is an easy application of \cite{Gabber-Liu-Lorenzini} Theorem 5.1. So we can assume that \(d = 1\).
	
	Let \(\Ass(X)\) be the finite set of all associated points of \(X\), and similarly for \(\Ass(Y_1)\) and \(\Ass(Y'_1)\). Let \(A = Y_{1,t}\cup Y_{2,t}\cup \Ass(X)\cup\Ass(Y_1)\cup \Ass(Y'_1) \), where \(Y_{1,t} = Y_1\cap f^{-1}(t)\) and \(Y_{2,t} = Y_2\cap f^{-1}(t)\). Note that \(A\) is also a finite set. Then by \cite{Gabber-Liu-Lorenzini} Theorem 5.1, we can find an integer \(n>0\) and a section \(s_2\in\Gamma(X,\sL_2)\) such that \(Z(s_2)\cap A = \emptyset\). Let \(Y_2 = Z(s_2)\). It is obvious that both \(s_2|_{Y_1}\) and \(s_2|_{Y'_1}\) are regular. It remains to show that we can find an open neighborhood of \(t\) such that \(Y_{2,U}\to U\) is of pure relative dimension \(0\).
	
	Let \(f_t\) denote the projective morphism \(X_t\to t\), where \(X_t = f^{-1}(t)\). Since \(\sL_{2,X_t}\) is \(f_t\)-very ample, by \cite{EGA4-2} Chapter 4 (5.3.1.3), \(Y_{2,t}\ne\emptyset\). Then by \cite{EGA4-2} (5.1.8), \(\dim Y_{2,t}\geqslant 0\). Since \(Y_{1,t}\cap Y_{2,t} = \emptyset\), by \cite{EGA4-2} (5.1.8), \(\dim Y_{2,t} < 1\). Thus \(\dim Y_{2,t} = 0\). Therefore, by \cite{EGA4-3} (13.1.5), there exists an open neighborhood \(U\) of \(t\) such that \(Y_{2,U}\to U\) is of pure relative dimension \(0\).
%
%
%
%
%
\end{proof}

Now we prove that the metric we constructed on \(\langle\sL_1,\sL_2,\ldots,\sL_{d+1}\rangle_{X/S}\) is independent of the choice of the regular global section \(s_1\). Let \(s'_1\) be another regular global section \(s'_1\), and our goal is to show that the metric \(\VCV'\) on \(\langle\sL_1,\sL_2,\ldots,\sL_{d+1}\rangle_{X/S}\) defined by \([s'_1]_{X/S}\) equals \(\VCV\) defined by \([s_1]_{X/S}\). Since this problem is local and \(\Sigma_2\) is an isometry, by Lemma \ref{lem:existence of simultaneously regular section}, we can assume that \(\sL_2\) admits a global section \(s_2\) such that both \(s_1,s_2\) and \(s'_1,s_2\) are \(f\)-regular. Let \(Y_2 = Z(s_2)\). Then we have the following isomorphism
\begin{equation*}
[s_2]_{X/S}\colon\langle\sL_1,\sL_2,\ldots,\sL_{d+1}\rangle_{X/S}\to \langle\sL_1,\sL_3,\ldots,\sL_{d+1}\rangle_{Y_2/S}
\end{equation*}
Let \(\VCV_{Y_2}\) be the unique smooth metric on \(\langle\sL_1,\sL_3,\ldots,\sL_{d+1}\rangle_{Y_2/S}\). Then by Lemma \ref{lem:norm of s_2},
\begin{equation*}
	\log\VCV  = \log\VCV_{Y_2} - \log\Vert[s_2]_{X/S}\Vert = \log\VCV'.
\end{equation*}
Thus \(\VCV = \VCV'\).

Now let \((\sL_1,\VCV'_1)\) be another invertible \(\sO_X\)-modules endowed with a smooth metric, Suppose that both \(\sL_1\) and \(\sL'_1\) admit an \(f\)-regular global section. Let \(s_1\) and \(s'_1\) be \(f\)-regular global sections of \(\sL_1\) and \(\sL'_1\) respectively. Note that \(\sL_1\otimes\sL'_1\) is endowed with a natural smooth metric \(\VCV_1\otimes\VCV'_1\), and that \(s_1\otimes s'_1\) is an \(f\)-regular global section of \(\sL_1\otimes\sL'_1\). Thus we can endow each of \(\langle\sL_1,\sL_2,\ldots,\sL_{d+1}\rangle_{X/S}\), \(\langle\sL'_1,\sL_2,\ldots,\sL_{d+1}\rangle_{X/S}\), and \(\langle\sL_1\otimes\sL'_1,\sL_2,\ldots,\sL_{d+1}\rangle_{X/S}\) with a unique smooth metric by our above argument. 
Our claim is that the following isomorphism
\[ \langle\sL_1,\sL_2,\ldots,\sL_{d+1}\rangle_{X/S}\otimes \langle\sL'_1,\sL_2,\ldots,\sL_{d+1}\rangle_{X/S}\to \langle\sL_1\otimes\sL'_1,\sL_2,\ldots,\sL_{d+1}\rangle_{X/S}, \]
which we denote by \(\varphi\), is an isometry. The proof is as follows. By similar argument as the proof of Lemma \ref{lem:existence of simultaneously regular section}, we can prove the follow statement: if \(\sL_2\) is an \(f\)-very ample invertible \(\sO_X\)-module, then for every point \(t\in S\), there exists an open neighborhood \(U\) of \(t\) and an integer \(n>0\) such that there exists a regular section \(s_2\in\Gamma( X_U,\sL_2^{\otimes n})\) such that \((s_1,s_2)\), \((s'_1,s_2)\), and \((s_1\otimes s'_1, s_2)\) are \(f_U\)-regular sequences, where \(f_U\colon X_U\to U\). Using this statement, and the facts that this problem is local and that \(\Sigma_2\) is an isometry, we can assume that \(\sL_2\) admits a regular global section \(s_2\) such that \((s_1,s_2)\), \((s'_1,s_2)\), and \((s_1\otimes s'_1, s_2)\) are \(f\)-regular sequences. Let \(Y_2 = Z(s_2)\). Then we have the following commutative diagram
\[ \begin{tikzcd}
	\langle\sL_1,\sL_2,\ldots,\sL_{d+1}\rangle_{X/S}\otimes \langle\sL'_1,\sL_2,\ldots,\sL_{d+1}\rangle_{X/S}\arrow[r,"\varphi"]\arrow[d, "{[}s_2{]}_{X/S,\mathrm{i}}\otimes{[}s_2{]}_{X/S,\mathrm{ii}}"]& \langle\sL_1\otimes\sL'_1,\sL_2,\ldots,\sL_{d+1}\rangle_{X/S}\arrow[d, "{[}s_2{]}_{X/S,\mathrm{iii}}"]\\
	\langle\sL_1,\ldots,\sL_{d+1}\rangle_{Y_2/S}\otimes \langle\sL'_1,\ldots,\sL_{d+1}\rangle_{Y_2/S}\arrow[r,"\psi"]& \langle\sL_1\otimes\sL'_1,\ldots,\sL_{d+1}\rangle_{X/S},
\end{tikzcd} \]
where
\[ [s_2]_{X/S,\mathrm{i}}\colon\langle\sL_1,\sL_2,\ldots,\sL_{d+1}\rangle_{X/S}\to \langle\sL_1,\ldots,\sL_{d+1}\rangle_{Y_2/S},\]
\[ [s_2]_{X/S,\mathrm{ii}}\colon\langle\sL_1,\sL_2,\ldots,\sL_{d+1}\rangle_{X/S}\to \langle\sL'_1,\ldots,\sL_{d+1}\rangle_{Y_2/S}, \]
and
\[ [s_2]_{X/S,\mathrm{iii}}\colon\langle\sL_1,\sL_2,\ldots,\sL_{d+1}\rangle_{X/S}\to \langle\sL_1\otimes\sL'_1,\ldots,\sL_{d+1}\rangle_{Y_2/S}. \]
Note that \([s_2]_{X/S,\mathrm{i}}\), \([s_2]_{X/S,\mathrm{ii}}\), \([s_2]_{X/S,\mathrm{iii}}\), \(\psi\), and \(\varphi\) are all isomorphisms. By inductive hypothesis, \(\psi\) is an isometry. So it suffices to prove that
\[ \log\Vert[s_2]_{X/S,\mathrm{iii}}\Vert = [s_2]_{X/S,\mathrm{i}} + [s_2]_{X/S,\mathrm{ii}}, \]
which follows from Lemma \ref{lem:norm of s_2} and the fact that \[ c_1(\sL_1\otimes\sL'_1,\VCV_1\otimes\VCV'_1) = c_1(\sL_1,\VCV_1) + c_1(\sL'_1,\VCV'_1). \]

Since locally every invertible \(\sO_X\)-module is the difference of two \(f\)-ample invertible \(\sO_X\)-module, and by Lemma \ref{lem:existence of f-regular section} some positive power of an \(f\)-ample invertible \(\sO_X\)-module admits an \(f\)-regular global section. Thus the above claim allows us to endow the Deligne pairing \(\langle\sL_1,\sL_2,\ldots,\sL_{d+1}\rangle_{X/S}\) of general hermitian line bundles \(\sL_1,\sL_2,\ldots,\sL_{d+1}\) with a smooth metric, regardless of whether \(\sL_1\) admits regular sections. And we see that \(\Sigma_i\, (1\leqslant i\leqslant d+1)\), \(\gamma_\phi\) (\(\phi\) is a permutation of the set \{2,3,\ldots,d+1\}), and \eqref{eq:isom of good base change} are all isometries. Moreover, since both sides of \eqref{eq:norm of s_2} are linear, we can strengthen Lemma \ref{lem:norm of s_2} as follows.

\begin{Lem}\label{lem:norm of general s_2}
	Let \(s_2\) be an \(f\)-regular global sections of \(\sL_2\) and \(Y_2 = Z(s_2)\)
	If we endow \(\langle\sL_1,\sL_2,\ldots,\sL_{d+1}\rangle_{X/S}\) with the metric we constructed above, then
	\[ \log\Vert[s_2]_{X/S}\Vert = -\int_{X/S}\log\Vert s_2\Vert_2\prod_{i\neq 2}c_1(\sL_i,\VCV_i), \]
	where \([s_2]_{X/S}\colon\langle\sL_1,\sL_2,\ldots,\sL_{d+1}\rangle_{X/S}\to \langle\sL_1,\sL_3,\ldots,\sL_{d+1}\rangle_{Y_2/S}\) is the isomorphism in Theorem \ref{projection formula of taking closed subscheme} for \(D = Z(s_2)\).
\end{Lem}

It remains to show that \(\gamma_{(12)}\) is an isometry, where \((12)\) denotes the transposition. Since the problem is local, and \(\Sigma_1\) is an isometry, we can assume that \(\sL_1\) admits an \(f\)-regular sequence \(s_1\), and let \(Y_1 = Z(s_1)\). By Theorem \ref{restriction compatible with symmetry}, we have the following commutative diagram
\begin{equation}\label{cd for permutation (12)}
	\begin{tikzcd}[row sep=small]
		\langle\sL_1,\sL_2,\dots,\sL_{d+1}\rangle_{X/S}\arrow[dd,"\gamma_{(12)}"]\arrow[rd, "{[}s_1{]}_{X/S,1}"]& \\
		& \langle\sL_2,\dots,\sL_{d+1}\rangle_{Y_1/S}.\\
		\langle\sL_2,\sL_1,\sL_3,\dots,\sL_{d+1}\rangle_{X/S}\arrow[ru, "{[}s_1{]}_{X/S,2}"'] &
	\end{tikzcd}
\end{equation}
By Lemma \ref{lem:norm of general s_2}, \(\log\Vert[s_1]_{X/S,1}\Vert = \log\Vert[s_1]_{X/S,2}\Vert\). Therefore, \(\gamma_{(12)}\) is an isometry.
\end{proof}

Let \(U\subseteq S\) be the maximal open subset over which \(f\) is flat. By \cite{YuanALB} Theorem 4.2.3, wan can construct a canonical metric \(\VCV_{X_U/U}\) on \(\langle \sL_1,\ldots,\sL_{n+1}\rangle|_{X_U/U}\). By comparing the constructions in \cite{YuanALB} \S4.2, \S4.3, and our paper, we can see that the restriction of our metric \(\VCV\) on \(U\) is exactly \(\VCV_{X_U/U}\). In addition, if the metrics on \(\sL_1,\sL_2,\ldots,\sL_{d+1}\) are only continuous integrable, \cite{YuanALB} Theorem 4.2.3 shows that we can still construct a metric on \(\langle \sL_1,\ldots,\sL_{n+1}\rangle_{X_U/U}\), which is continuous integrable. Since a continuous integrable metric is a uniform limit of smooth metrics, it is not hard to extend our results to continuous integrable metrics, and we omit the details.

The following theorems show that the canonical metric is compatible with functorial properties.

\begin{Thm}\label{DEP of isometry is isometry}
	Let \(X\) and \(S\) quasi-projective schemes over \(\bC\), and let \(f\colon X\to S\) be a surjective projective morphism over \(\bC\) of pure relative dimension \(d\). Suppose that \(S\) is normal. Let \((\sL_1,\VCV_1),(\sL_2,\VCV_2),\ldots,(\sL_{d+1},\VCV_{d+1}),(\sL'_1,\VCV'_1),(\sL'_2,\VCV'_2),\ldots,(\sL'_{d+1},\VCV'_{d+1})\) be invertible \(\sO_X\)-modules endowed with smooth metrics. For each \(1\leqslant i \leqslant d+1\), let \(u_i\colon \sL_i\to \sL'_i\) be an isometry. If we endow \(\langle \sL_1,\sL_2,\ldots,\sL_{d+1}\rangle_{X/S}\) and \(\langle \sL'_1,\sL'_2,\ldots,\sL'_{d+1}\rangle_{X/S}\) with the metrics we constructed in Theorem \ref{thm:metric on EDP}. Then the following isomorphism of invertible \(\sO_S\)-modules
	\[ \langle u_1,u_2,\ldots,u_{d+1}\rangle_{X/S}\colon\langle \sL_1,\sL_2,\ldots,\sL_{d+1}\rangle_{X/S}\to\langle \sL'_1,\sL'_2,\ldots,\sL'_{d+1}\rangle_{X/S} \]
	is also an isometry.
\end{Thm}
\begin{proof}
	We prove the result by induction on \(d\). Let \(\VCV\) and \(\VCV'\) denote the metrics on \(\langle \sL_1,\sL_2,\ldots,\sL_{d+1}\rangle_{X/S}\) and \(\langle \sL'_1,\sL'_2,\ldots,\sL'_{d+1}\rangle_{X/S}\) respectively. If \(d=0\), then for each nonzero rational section \(s\) of \(\sL\),
	\begin{equation*}
		\log \Vert\Nm_{X/S}(u_1(s))\Vert' = \int_{X/S}\log \Vert u_1(s)\Vert'_1 = \int_{X/S}\log \Vert s\Vert_1 = \log \Vert\Nm_{X/S}(s)\Vert.
	\end{equation*}
	Thus \(\VCV = \VCV'\).
	
	Suppose that \(d>0\). We can assume that \(\sL_1\) admits an \(f\)-regular section \(s_1\). Let \(Y = Z(s_1) = Z(u_1(s_1))\).  Then we have the following commutative diagram
	\begin{equation*}
		\begin{tikzcd}
			\langle \sL_1,\sL_2,\ldots,\sL_{d+1}\rangle_{X/S}\arrow[r, "{[}s_1{]}_{X/S}"]\arrow[d,"\langle u_1{,}u_2{,} \ldots{,}u_{d+1}\rangle_{X/S}"]&\langle \sL_{2}|_Y,\ldots,\ldots\sL_{d+1}|_Y\rangle_{Y/S}\arrow[d,"\langle u_{2}|_Y{,} \ldots{,}u_{d+1}|_Y\rangle_{Y/S}"]\\
			\langle\sL'_1,\sL'_2,\ldots,\sL'_{d+1}\rangle_{X/S}\arrow[r, "{[}u_1(s_1){]}_{X/S}"]&\langle \sL'_{2}|_Y,\ldots\sL'_{d+1}|_{{Y}}\rangle_{{Y}/S}.
		\end{tikzcd}
	\end{equation*}
	By inductive hypothesis, the right vertical arrow is an isometry. Moreover,
	\begin{multline*}
		\log \Vert[u_1(s_1)]_{X/S}\Vert = -\int_{X/S}\log\Vert u_1(s_1)\Vert'_1 \prod_{i=2}^{d+1}c_1(\sL'_i,\VCV'_i)
		\\= -\int_{X/S}\log\Vert s_1\Vert_1 \prod_{i=2}^{d+1}c_1(\sL_i,\VCV_i) = \log \Vert[s_1]_{X/S}\Vert.
	\end{multline*}
	Thus the left vertical arrow is also an isometry. 
\end{proof}
 
\begin{Thm}\label{thm:pull-back formula of metric}
	Let \(X\) and \(S\) quasi-projective schemes over \(\bC\), and let \(f\colon X\to S\) be a surjective projective morphism over \(\bC\) of pure relative dimension \(d\). Suppose that \(S\) is integral and normal. Let \((\sL_1,\VCV_1),(\sL_2,\VCV_2),\ldots,(\sL_{d},\VCV_{d})\) be invertible \(\sO_X\)-modules endowed with smooth metrics, and let \((\sM,\VCV)\) be an invertible \(\sO_S\)-modules endowed with smooth metrics. We endow \(f^*\sM\) with the smooth metric \(f^*\VCV\). Let \(\delta\) be the intersection number of \(\sL_1,\ldots,\sL_d\) on the generic fiber of \(f\) (if \(d = 0\), then let \(\delta\) be the degree of the finite morphism \(f\)). If we endow \(\langle\sL_1,\ldots,f^*\sM,\ldots,\sL_d\rangle_{X/S}\) with the canonical metric, then the natural isomorphism
	\[ \langle\sL_1,\ldots,f^*\sM,\ldots,\sL_d\rangle_{X/S}\to \sM^{\otimes\delta} \]
	is an isometry.
\end{Thm}
\begin{proof}
	We prove the result by induction on \(d\). If \(d = 0\), then for any nonzero rational section \(s\) of \(\sM\),
	\begin{equation*}
		\log \Vert f^*s\Vert = \int_{X/S}\log f^*\Vert s\Vert = \delta\log\Vert s\Vert.
	\end{equation*}
    Suppose that \(d>0\). Then we can assume that \(\sL_1\) admits an \(f\)-regular section \(s_1\). Let \(Y = Z(s_1)\). By properties of intersection numbers, \(\delta\) equals the intersection number of \(\sL_2|_{Y},\ldots,\sL_{d}|_{Y}\) on the generic fiber of \(Y\to S\). Then we have the following commutative diagram
    \[ \begin{tikzcd}[row sep=small]
    	\langle\sL_1,\ldots,f^*\sM,\ldots,\sL_d\rangle_{X/S}\arrow[rd,"\varphi"]\arrow[dd,"{[}s_1{]}_{X/S}"]& \\
    	&\sM^{\otimes\delta}.\\
    	\langle\sL_2|_{Y},\ldots,f^*\sM|_{Y},\ldots,\sL_d|_{Y}\rangle_{Y/S}\arrow[ru,"\psi"']&
    \end{tikzcd} \]
    By inductive hypothesis, \(\psi\) is an isometry. Moreover,
    \[ \log\Vert{[s_1]_{X/S}}\Vert = -\int_{X/S}\log\vert s_1 \Vert_{1}f^*c_1(\sM,\VCV)\prod_{i=2}^{d}c_1(\sL_i,\VCV_i) = 0,\]
    since \(c_1(\sM,\VCV_{\mathrm{m}})\) is a smooth form on \(S\). Therefore, \(\varphi\) is also an isometry.
\end{proof}

\printbibliography
\end{document}